\documentclass[11pt]{article}
\newcounter{defcounter}

\usepackage{hyperref}
\usepackage{enumerate, amsthm}
\usepackage{mathrsfs}
\usepackage{rotating}
\usepackage{cancel}
\usepackage{xcolor}
\usepackage[margin=1 in]{geometry}
\usepackage[margin={1.5cm, 1.5cm},font=small, labelsep=endash]{caption}
\usepackage{mathtools}
\usepackage{tikz}
\usetikzlibrary{patterns}
\usepackage{tikz-dimline}

      \usepackage{amssymb}
      \usepackage{amsmath}
      
      \numberwithin{equation}{section}
      \theoremstyle{plain}
      \newtheorem{theorem}{Theorem}[section]
      \newtheorem{lemma}[theorem]{Lemma}
      
      \newtheorem{proposition}[theorem]{Proposition}
      \newtheorem{observation}[theorem]{Observation}
      \theoremstyle{definition}

      \theoremstyle{remark}
      \newtheorem{remark}[theorem]{Remark}
      \newtheorem{claim}[theorem]{Claim}

      \newcommand{\R}{{\mathbb R}}
      \newcommand{\E}{\mathbb E}
      \newcommand{\e}{\mathrm{e}}
      \renewcommand{\P}{\mathbb P}
      
      \newcommand{\var}{\mathrm{Var}}
      \newcommand{\cov}{\mathrm{Cov}}
      \newcommand{\cross}{\textrm{cross}}
      
      \newcommand{\up}{\textrm{up}}
      \newcommand{\down}{\textrm{down}}
      \newcommand{\Z}{\mathbb{Z}}
      
      \newcommand{\N}{\mathbb{N}}
      \newcommand{\A}{\mathcal{A}}
      \newcommand{\B}{\mathcal{B}}
      \newcommand{\T}{\mathcal{T}}

      \newcommand{\inte}{\mathrm{int}}
      \newcommand{\Hl}{\mathbb{HL}}
      \newcommand{\Vl}{\mathbb{VL}}

      \newcommand{\Taylor}{\mathrm{Taylor}}
      
      \newcommand{\quadr}{\mathrm{Quad}}
      \newcommand{\resid}{\mathrm{Resid}}

      \newcommand{\gram}{\mathrm{Gram}}
      
      \newcommand{\increment}{\mathrm{Increment}}

      \newcommand{\midd}{\mathrm{mid}}
      \newcommand{\ancest}{\mathrm{ancest}}
      \newcommand{\descend}{\mathrm{Descend}}
      \newcommand{\vertical}{\mathrm{vertical}}
      \newcommand{\gain}{\mathrm{Gain}}
      \newcommand{\side}{\mathrm{side}}
      \newcommand{\principal}{\mathrm{principal}}
      \newcommand{\abs}{\mathrm{Absolute}}
      \newcommand{\join}{\mathrm{join}}
      
      \newenvironment{hypothesis}{ 
      \addtocounter{equation}{-1}
      \refstepcounter{defcounter}
      
      \begin{eqnarray}}
      {\end{eqnarray}}
      
      \makeatletter
      \def\@setcopyright{}
      \def\serieslogo@{}
      \makeatother
      
\begin{document}
\title{Liouville first passage percolation: the weight exponent is strictly less than 1 at high temperatures}

\author{ Jian Ding\thanks{Partially supported by NSF grant DMS-1455049 and Alfred Sloan fellowship.} \\ University of Chicago \and Subhajit Goswami\footnotemark[1]  \\
University of Chicago
}
\date{}

\maketitle

\begin{abstract}
Let $\{\eta_{N, v}: v\in V_N\}$ be a discrete Gaussian
free field in a two-dimensional box $V_N$ of side length
$N$ with Dirichlet boundary conditions. We study the Liouville first
passage percolation, i.e., the shortest path metric where each vertex is given a weight of $e^{\gamma \eta_{N, v}}$
for some $\gamma>0$. We show that for sufficiently small but fixed
$\gamma>0$, the expected Liouville FPP distance between any pair of vertices is $O(N^{1-\gamma^2/10^3})$.
\end{abstract}

\maketitle

\section{Introduction}

In this paper we study Liouville first-passage percolation (which was explicitly mentioned in \cite{Benjamini}); i.e., first-passage
percolation on the exponential of the planar discrete Gaussian free field (GFF).
Given a two-dimensional box $V_N$ of side length $N$, the discrete Gaussian free field $\{\eta_{N, v}: v\in V_N\}$ with Dirichlet
boundary conditions is a mean-zero
Gaussian process such that 
$$\eta_{N, v}=0
\mbox{ for all } v\in \partial V_N\,, \mbox{ and }\E \eta_{N, x} \eta_{N, y}=G_{V_N}(x,y)
\mbox{ for all } x,y\in V_N\,,$$ where $G_{V_N}(x,y)$
is the Green's function of simple random walk on $V_N$. For a fixed inverse-temperature parameter $\gamma>0$, we define the \emph{Liouville first-passage percolation} (Liouville FPP) metric $D_N(\cdot, \cdot)$
on $V_N$ by
\begin{equation}\label{eq-LFPP-def}
D_{N, \gamma}(x_{1},x_{2})=\min_{\pi}\sum_{x\in\pi}e^{\gamma \eta_{N, x}},
\end{equation}
where $\pi$ ranges over all paths in $V_N$ connecting
$x_{1}$ and $x_{2}$. 
\begin{theorem}
\label{theo:main}
There exists $C_\gamma > 0$ and a small, positive absolute constant $\gamma_0$ such that for all $\gamma \leq \gamma_0$, we have 
$$\max_{x, y\in V_N}\E D_{N, \gamma}(x, y) \leq C_\gamma N^{1 - \gamma^2/10^3}\,.$$
\end{theorem}
\begin{remark}
\label{remark:main_theorem}
Theorem~\ref{theo:main} applies if we consider $x, y \in V_{\delta N}$ for any fixed $0 < \delta < 1$ and restrict $\pi$ to be a path within $V_{\delta N}$ in \eqref{eq-LFPP-def}. 
\end{remark}

Theorem~\ref{theo:main} is mostly related to our previous work \cite{DG15} where we proved a similar result when the underlying Gaussian field is a branching random walk (BRW). Also, combined with \cite{DZ15}, Theorem~\ref{theo:main} shows that the strong universality for the  first passage percolation exponent does not hold among the family of log-correlated Gaussian fields. That is to say, the weight exponents may differ for different families of log-correlated Gaussian fields.

\subsection{Backgrounds and related works}
Much effort has been devoted to understanding classical first-passage
percolation (FPP), with independent and identically distributed edge/vertex
weights. We refer the reader to \cite{ADH15,GK12} and their references
for reviews of the literature on this subject. We argue that FPP with
strongly-correlated weights is also a rich and interesting subject,
involving questions both analogous to and divergent from those asked
in the classical case. Since the Gaussian free field is in some sense
the canonical strongly-correlated random medium, we see strong motivation
to study Liouville FPP.

More specifically, Liouville FPP is thought to play a key role in
understanding the random metric associated with the Liouville quantum
gravity (LQG) \cite{P81,DS11,RV14}.  We remark that the random metric of LQG is a major open problem, even just to make rigorous sense of it (we refer to \cite{RV16} for a rather up-to-date review). In a recent series of works of Miller and Sheffield, much understanding has been obtained (more on the continuum set up) in the special case of $\gamma = \sqrt{8/3}$\footnote{We learned from R{\'e}mi Rhodes and Vincent Vargas that, according to \cite{Watabiki93}, the physically appropriate approximation for the $\gamma$-LQG metric should involve $\min_\pi \sum_{v\in \pi} e^{\frac{\gamma}{d_H(\gamma)} \eta_{N, v}}$, i.e., the parameter in the exponential of GFF is $\gamma/d_H(\gamma)$ instead of $\gamma$. Here $d_H(\gamma)$ is the unknown Hausdorff dimension which is predicted to be $1+\frac{\gamma^2}{4} + \sqrt{(1+ \frac{\gamma^2}{4})^2+ \gamma^2}$. We chose not to emphasize this in the main text, since mathematically, as far as our main result is concerned, this is merely a change of parameter.}; see \cite{MS15, MS15b} and references therein. Our approach is different, in the sense that we aim to understand the random metric of LQG via approximations of natural discrete metrics. In the physics literature \cite{Watabiki93, ANRBW98, AB14}, precise predictions were made on closely related metric exponents, and our Theorem~\ref{theo:main} is consistent with these predictions. 
In addition, we note that one may need to tweak the definition of the Liouville FPP in order to obtain a discrete approximation leading to an invariant scaling limit.  However, we feel that in the level of precision of the present article, it is likely that the fundamental mathematical structures (and thus obstacles) are common for all the candidate discrete approximations.

Furthermore, we expect that Liouville FPP metric is related to the heat kernel estimate for Liouville Brownian motion (LBM), for which the mathematical construction (of the diffusion) was provided in \cite{GRV13, B14} and the heat kernel was constructed in \cite{GRV14}. The LBM is closely related to the geometry of LQG; in \cite{FM09, BGRV14} the KPZ formula was derived from Liouville heat kernel. In \cite{MRVZ14} some nontrivial bounds for LBM heat kernel were established. A very interesting direction is to compute the heat kernel of LBM with high precision. It is plausible that understanding the Liouville FPP metric is of crucial importance in computing the LBM heat kernel.

Finally, in a very recent work \cite{DD16}, it was shown that at high temperatures the appropriately normalized Liouville FPP converges subsequentially in the Gromov-Hausdorff sense to a random metric on the unit square, where all the (conjecturally unique) limiting metrics are homeomorphic to the Euclidean metric. We remark that the proof method in the current paper bears little similarity to that in \cite{DD16}.

\subsection{New challenges for the GFF provided a proof for BRW}\label{sec:new-challenge}

Our proof strategy naturally inherits that of \cite{DG15} which proved an analogue of Theorem~\ref{theo:main} in the context of BRW, and we encourage the reader to flip through \cite{DG15} and in particular \cite[Section 1.2]{DG15} which contains a prototype of the multi-scale analysis carried out in the current paper. In what follows, we emphasize the \emph{substantial} new challenges in the case of GFFs. To make our point, we note that the maximum for branching Brownian motion and BRW (with Gaussian increments) were well-understood (see, e.g., \cite{Bramson78, Bachmann00}) much before a good understanding for the maximum of GFF \cite{BDG01, BZ10,  BDZ14} --- even though there is \emph{universality} among log-correlated Gaussian fields \cite{Madaule15, DRZ15} for the behavior of the maximum. For the FPP problem, we know from \cite{DZ15}  that there is no universality for the weight exponent  and in particularly the weight exponent can be arbitrarily close to 1 if we allow to tune the covariance structure of the field up to a large additive constant. Therefore, in order to prove Theorem~\ref{theo:main}, one has to take into account the very subtle covariance structure of the GFF, rather than simply treat it as an instance of log-correlated Gaussian fields. From a technical point of view, \cite{DZ15} implies that Gaussian comparison theorems such as Slepian's lemma \cite{slepian62} and Sudakov-Fernique inequality \cite{F2} are not expected to be available for the FPP problem --- but the comparison theorems allowed to approximate GFF by a more tractable field in the study of the maximum, which was crucial to \cite{BZ10, BDZ14}. 

In light of the preceding discussion, when carrying out the multi-scale analysis for the Liouville FPP problem, we did not see an alternative rather than precisely decompose the GFF into many scales. There are a number of such decompositions available, and the one using Markov field property of the GFF turns out (as least as it seems to us at the moment) the way easiest to work with. Such decomposition was used extensively in the study of GFFs, for instance in \cite{BDZ14}. In what follows, we will elaborate a number of new subtleties (in review of \cite{DG15}) that have to be taken into account for the Liouville FPP problem. First of all, the Markov field property was used in \cite{BDZ14} to decompose the GFF into a sum of \emph{coarse} and \emph{fine} fields where the fine field is then approximated by the modified branching random walk. In our context, we have to use Markov field property to decompose the field into order of $\log N$ many scales and we are not aware of any legitimate approximation (due to non-universality as of \cite{DZ15}). Second of all, usually squared boxes are employed in the Markov field decomposition as done in \cite{BDZ14}, but in our context in order to fully harness the covariances of the GFF in each scale in the horizontal direction (this corresponds to our optimization strategy which is to switch between light crossings in the top and bottom layers in order to construct a light crossing in a bigger scale) we need to do the Markov field decomposition using \emph{rectangles}. Third of all, each scale  after decomposing the GFF is roughly a collection of harmonic averages on the boundary of boxes/rectangles, and in order for the variances in all scales to be of order 1 it is important not to take harmonic averages for points close to the boundary. In order to address this, in \cite{BDZ14} the authors simply threw away a small fraction of the box and showed \emph{a priori} that its effect is negligible. In the context of FPP problem since we need to construct light crossings that is a connected path, we cannot afford to simply throw away a fraction of the box. As a result, the rectangles employed in our decomposition are not nested. Moreover, when inductively constructing light crossings in big scales from small scales, we need to let the scale to grow as a power of $(2 + \delta)$ (other than a more conventional power of 2) where the ``$\delta$'' is due to the need of filling in the gaps near the boundary of the two big rectangles; see Figure \ref{fig:Tiling}. All of these incur technical challenges. In Section~\ref{sec:prelim}, we lay out the foundation for the multi-scale analysis, which in particular includes the decomposition of the GFF, and a number of variance/covariance estimates for various Gaussian processes that arise from such decomposition. 

Another source of main challenges is the fact that the field in each scale is a smoothly varying field, rather than a constant over a box as in the case of BRW. As a result, in order for an effective optimization in the induction procedure, we will have to know the geometry of the already constructed crossings (in the previous scale) in a resolution that is much more refined then the current scale. This is in some sense equivalent to access the realization of the switching strategy in previous many scales. Among others, this incur an issue of correlation between switching strategies in different scales. In order to address the issue, we introduce two types of strategies in the inductive constructions for light crossings, where Strategy I only serves to decorrelate the switching strategy in different blocks of scales (where each block consists of a large constant order of scales).  The inductive construction and analysis is carried out in Section~\ref{sec:the_main}.

Finally, a crucial ingredient in \cite{DG15} is the asymptotics for the regularized total variation of Brownian motion as shown in \cite{Dunlap}. This arose because when constructing crossing in the current scale, each switching requires some vertical gadget to connect the top and bottom layer of the crossings (in previous scale) and each such vertical gadget has a cost. In \cite{DG15}, it suffices to simply consider the expected cost for each vertical gadget where in particular we average over the height difference between the crossings in the top and bottom. In order for an efficient optimization in the present paper, we have to take into account the actual height difference in the optimization, and this lead to a problem on regularized total variation for Brownian motion with \emph{inhomogeneous} penalties (as opposed to homogeneous penalty as in \cite{Dunlap}). We are, in fact, not able to compute the asymptotic value in the inhomogeneous case.  Instead, we prove a lower bound that is sufficient for our purpose, based on a combination of an idea of \cite{Dunlap} and a delicate application of renewal theory \cite{Stone65}. This is incorporated in Section~\ref{sec:total_variation}. 

\subsection{A heuristic outline of proof}
\label{subsec:heuristic}
In this subsection, we provide a heuristic calculation which on one hand ignores a large part of the subtleties discussed in Subsection~\ref{sec:new-challenge} but on the other hand contains the mathematical indication that the weight exponent for Liouville FPP is strictly less than 1. We will not be completely precise in what follows.

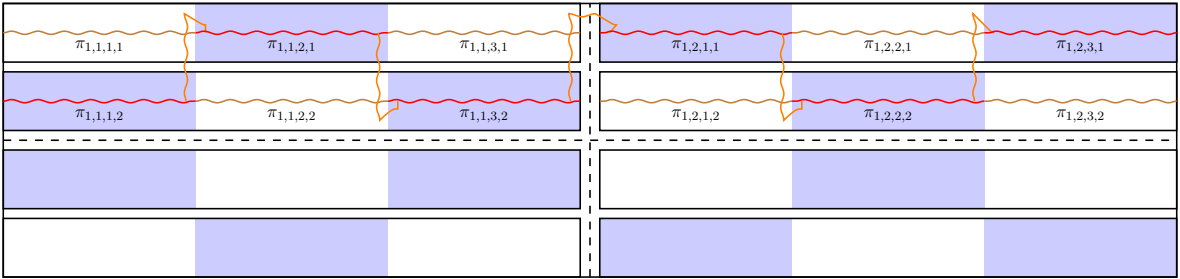
\begin{figure}[!htb]
\centering
\begin{tikzpicture}[semithick, scale = 2.6]
\fill [blue!20]	(-3, -0.3) rectangle (-2.016, 0);
\fill [blue!20]	(-2.016, 0.05) rectangle (-1.033, 0.35);
\fill [blue!20]	(-1.033, -0.3) rectangle (-0.05, 0);

\fill [blue!20]	(3, 0.05) rectangle (2.016, 0.35);
\fill [blue!20]	(2.016, -0.3) rectangle (1.033, 0);
\fill [blue!20]	(1.033, 0.05) rectangle (0.05, 0.35);

\draw (-3, -0.3) rectangle (-0.05, 0);
\draw (-3, 0.05) rectangle (-0.05, 0.35);
\draw (3, -0.3) rectangle (0.05, 0);
\draw (3, 0.05) rectangle (0.05, 0.35);

\fill [blue!20]	(-3, -0.7) rectangle (-2.016, -0.4);
\fill [blue!20]	(-2.016, -1.05) rectangle (-1.033, -0.75);
\fill [blue!20]	(-1.033, -0.7) rectangle (-0.05, -0.4);

\fill [blue!20]	(3, -1.05) rectangle (2.016, -0.75);
\fill [blue!20]	(2.016, -0.7) rectangle (1.033, -0.4);
\fill [blue!20]	(1.033, -1.05) rectangle (0.05, -0.75);
\draw (-3, -1.05) rectangle (-0.05, -0.75);
\draw (-3, -0.7) rectangle (-0.05, -0.4);
\draw (3, -1.05) rectangle (0.05, -0.75);
\draw (3, -0.7) rectangle (0.05, -0.4);

\draw  (-3, -1.05) rectangle (3, 0.35);

\draw [dashed] (-3, -0.35) -- (3, -0.35);
\draw [dashed] (0, -1.05) -- (0, 0.35);

\draw [brown, style={decorate,decoration={snake,amplitude = 0.6}}] (-3, 0.2) -- (-2.016, 0.2);
\node [scale = 0.6, below] at (-2.508, 0.19) {$\pi_{1, 1, 1, 1}$};
\draw [red, style={decorate,decoration={snake,amplitude = 0.6}}] (-3, -0.15) -- (-2.016, -0.15);
\node [scale = 0.6, below] at (-2.508, -0.16) {$\pi_{1, 1, 1, 2}$};

\draw [orange,style={decorate,decoration={snake,amplitude = 0.6}}] (-2.066, -0.15) -- (-2.066, 0.24) -- (-1.966, 0.24) -- (-1.966, 0.2);

\draw [red, style={decorate,decoration={snake,amplitude = 0.6}}] (-2.016, 0.2) -- (-1.033, 0.2);
\node [scale = 0.6, below] at (-1.525, 0.19) {$\pi_{1, 1, 2, 1}$};
\draw [brown, style={decorate,decoration={snake,amplitude = 0.6}}] (-2.016, -0.15) -- (-1.033, -0.15);
\node [scale = 0.6, below] at (-1.525, -0.16) {$\pi_{1, 1, 2, 2}$};

\draw [orange,style={decorate,decoration={snake,amplitude = 0.6}}] (-1.083, 0.2) -- (-1.083, -0.19) -- (-0.983, -0.19) -- (-0.983, -0.15);

\draw [brown, style={decorate,decoration={snake,amplitude = 0.6}}] (-1.033, 0.2) -- (-0.05, 0.2);
\node [scale = 0.6, below] at (-0.5415, 0.19) {$\pi_{1, 1, 3, 1}$};
\draw [red, style={decorate,decoration={snake,amplitude = 0.6}}] (-1.033, -0.15) -- (-0.05, -0.15);
\node [scale = 0.6, below] at (-0.5415, -0.16) {$\pi_{1, 1, 3, 2}$};

\draw [orange,style={decorate,decoration={snake,amplitude = 0.6}}] (-0.1, -0.15) -- (-0.1, 0.24) -- (0.1, 0.24) -- (0.1, 0.2);

\draw [red, style={decorate,decoration={snake,amplitude = 0.6}}] (1.033, 0.2) -- (0.05, 0.2);
\node [scale = 0.6, below] at (0.5415, 0.19) {$\pi_{1, 2, 1, 1}$};
\draw [brown, style={decorate,decoration={snake,amplitude = 0.6}}] (1.033, -0.15) -- (0.05, -0.15);
\node [scale = 0.6, below] at (0.5415, -0.16) {$\pi_{1, 2, 1, 2}$};

\draw [orange,style={decorate,decoration={snake,amplitude = 0.6}}] (0.983, 0.2) -- (0.983, -0.19) -- (1.083, -0.19) -- (1.083, -0.15);

\draw [brown, style={decorate,decoration={snake,amplitude = 0.6}}] (2.016, 0.2) -- (1.033, 0.2);
\node [scale = 0.6, below] at (1.525, 0.19) {$\pi_{1, 2, 2, 1}$};
\draw [red, style={decorate,decoration={snake,amplitude = 0.6}}] (2.016, -0.15) -- (1.033, -0.15);
\node [scale = 0.6, below] at (1.525, -0.16) {$\pi_{1, 2, 2, 2}$};

\draw [orange,style={decorate,decoration={snake,amplitude = 0.6}}] (1.966, -0.15) -- (1.966, 0.24) -- (1.966, 0.24) -- (2.066, 0.2);

\draw [red, style={decorate,decoration={snake,amplitude = 0.6}}] (3, 0.2) -- (2.016, 0.2);
\node [scale = 0.6, below] at (2.508, 0.19) {$\pi_{1, 2, 3, 1}$};
\draw [brown,style={decorate,decoration={snake,amplitude = 0.6}}] (3, -0.15) -- (2.016, -0.15);
\node [scale = 0.6, below] at (2.508, -0.16) {$\pi_{1, 2, 3, 2}$};

\end{tikzpicture}
\caption{{\bf Construction of a crossing through $V_{2N, \Gamma, 1}$.} The broken lines demarcate $V_{2N, \Gamma; i, j}$'s, starting from $V_{2N, \Gamma; 1, 1}$ at the upper-left corner. Each $V_{2N, \Gamma; i, j}$ consists of two rows, the top and the bottom one corresponds to $V_{2N, \Gamma; i, j, 1}$ and $V_{2N, \Gamma; i, j, 2}$ 
respectively. Each $V_{2N, \Gamma; i, j, k}$ is further subdivided into rectangles $V_{2N, \Gamma; i, j, j', k}$'s for $j' \in [3]$ in this figure. $\pi_{1, j, j', k}$. The red lines indicate the segments $\pi_{1, j, j', k}$'s that have been selected by $\A_{n + 1}$. The orange lines indicate the vertical gadgets that join these segments into a crossing for $V_{2N, \Gamma, 1}$.}
\label{fig:short_note}
\end{figure}
Let $\Gamma \approx \alpha / \gamma^2$ for some $1/\gamma \gg \alpha \gg 1$ and $V_{N, \Gamma} = ([0, \Gamma N - 1] 
\times [0, 2N - 1]) \cap \Z^2$ for $N = 2^n$. Thus $V_{N, \Gamma}$ consists of two $\Gamma N \times N$ rectangles placed on top of each other, say $V_{N, \Gamma, 1}$ (the top one) and $V_{N, \Gamma, 2}$ (the bottom one). 
Similarly $V_{2N, \Gamma}$ can be sub-divided into 4 copies of $V_{N, \Gamma}$ (see 
Figure~\ref{fig:short_note}). Call them $V_{2N, 
\Gamma; i, j}$ ($i, j \in [2]$) in the usual order. Each $V_{2N, \Gamma; i, j}$ contains two $\Gamma N \times N$ rectangles, denoted as $V_{N, \Gamma; i, j, 1}$ and $V_{N, \Gamma; i, j, 2}$. Define 
$$\eta_{2N, v} = (\eta_{2N, v} - \E(\eta_{2N, v} | \eta_{2N, \partial V_{2N, \Gamma; i, j}})) + \E(\eta_{2N, v} | \eta_{2N, \partial V_{2N, \Gamma; i, j}}) = \eta_{2N,i,j v} + X_{2N,i,j v}\,.$$
By Markov field property of GFF, the field $\{\eta_{2N,i,j,.}\}$ is a GFF on $V_{N, \Gamma; i, j}$ with Dirichlet boundary condition and is 
independent with $\{X_{2N,i,j, .}\}$. We refer to these two fields respectively as the fine and coarse fields on 
$V_{2N, \Gamma; i, j}$. In order to study the growth of weight for the crossings (i.e. paths connecting two facing boundaries) from scale $n$ to $n+1$, we will employ an algorithm $\A_{n'}$ for every $n' \in [n]$ that builds an ``economic'' crossing through each of $V_{N', \Gamma, 1}$ and $V_{N', \Gamma, 2}$ ($N' = 
2^{n'}$). These algorithms are of inductive nature. Below we give a (incomplete) description of $\A_{n'}$ based on algorithms in previous scales. We use $\A_{n' - 1}$ to build crossings through each $V_{N', \Gamma; i, j, k}$ where we take $\{\eta_{N',i,j,.}\}$ as 
the underlying field. Next we partition the horizontal range of $V_{N', \Gamma; i, j}$ into intervals of length $\beta N'/2$ where we choose $\beta$ such that $1/\gamma \gg \beta \gg 1$. Each such interval, say $I_{j, j'}$ ($j' \leq \Gamma/\beta$), defines a sub-rectangle 
$V_{N', \Gamma; i, j, k}$ of $V_{N', \Gamma; i, j}$. Denote by $\pi_{i, j, j', k}$ the portion of the crossing through $V_{N', \Gamma; i, j, k}$ that lies within $V_{N', \Gamma; i, j, 
j', k}$ (see Figure~\ref{fig:short_note}). For each $(j, j')$, we may opt for either $\pi_{i, j, j', 1}$ or 
$\pi_{i, j, j', 2}$ to move across $I_{j, j'}$. These choices correspond to a sequence of $\{1, 2\}$ valued random variables $\{k_{i, j, j'}\}_{(j, j') \in [2] \times [\Gamma/\beta]}$ (the \emph{switching strategy}) for each $i \in [2]$. Whenever we make a switch at $I_{j, j'}$ (i.e., $k_{i, j, j'} \neq k_{i, j, j' + 1}$), we link $\pi_{i, j, j', 1}$ and $\pi_{i, j, j', 2}$ using a 
``vertical gadget''. We can construct this gadget as a crossing through an appropriately placed, vertically aligned copy of $V_{2^{n''}, \Gamma}$ for some $n'' \leq n' - \log_2\Gamma$ and use $\A_{n''}$ with respect to the corresponding fine field. These operations give us a 
crossing through $V_{N', \Gamma; i}$. Our construction should also ensure that the two crossings are identically distributed with respect to the reflection of $V_{N', \Gamma, i}$.

Let us now focus on $V_{2N, \Gamma}$ and $V_{2N, \Gamma; 
1}$. Clearly, the expected weight of crossing should expand by a factor \emph{close} to 2 (compared to the previous scale). But we are evaluating expectations with respect to $\{\eta_{2N, .}\}$ and hence there is the effect of an additional factor $\e^{\gamma X_{2N, 1, j, 
\cdot}}$. In fact, the analysis of this effect is the 
central issue, which we elaborate in what follows. It is plausible that the total weight of gadgets that have been used between the scales $n - 100\log \Gamma$ and $n$ to build the crossing through $V_{2N, \Gamma; 1, j, 
k}$ is negligible compared to its total weight. Thus, we can ignore these gadgets and only consider the remaining points in the segments $\pi_{1, j, j', k}$'s when analyzing the effect of $\e^{\gamma X_{2N, 1, j, 
\cdot}}$. Denote these new segments as $\tilde \pi_{1, 
j, j', k}$. Since $\gamma$ is small, we have a legitimate approximation $\e^{\gamma X_{2N, 1, j, v}} \approx 1 + \gamma X_{2N, 1, j, v} + 
\tfrac{\gamma^2}{2}\E X_{2N, 1, j, v}^2$. If we choose our segments in a symmetric fashion for all the scales, it is not hard to show that (as in Lemma~\ref{lem:Brownian_scaling}) 
\begin{equation}\label{eq-expansion}
\tfrac{\gamma^2}{2}\E X_{2N, 1, j, \cdot}^2 \approx \tfrac{2}{\pi}\tfrac{\gamma^2}{2}\log 2 = 
\tfrac{\gamma^2}{\pi}\log 2 \approx 0.22\gamma^2\,,
\end{equation}
which amounts to the increment of the weight. The decrement of the weight will come 
from the random variables $X_{2N, 1, j .}$, together with our judicious switching strategy (which naturally would favor smaller random variables when choosing layers). We can imagine that the crossings we built are rather smooth in the coarse resolution. Combined with the smoothness of the field $\{X_{2N, 1, j, .}\}$ (see Lemma~\ref{lem:field_smoothness}), it yields the following approximate expression for the expected total weight of the segments with respect to $\{\eta_{2N, .}\}$:
\begin{equation}
\label{eq:short_note1}
\gamma\frac{d_{n, j'}}{\beta N}\E \Big( \sum_{(j, j') \in [2] \times [\Gamma / \beta]} \sum_{v \in I_{j, j'} \times \{\nu_{j, j', k_{1, j, j'}}\}}X_{2N, 1, j, v} \Big)\,,
\end{equation}
where $\nu_{j, j', k}$ is the (approximate) common height of $\pi_{1, j, j', k}$ on $I_{j, j'}$ and $d_{n, j'}$ is the expected weight of $\tilde \pi_{1, j, j', k}$ with 
respect to the fine field on $V_{2N, \Gamma; 1, j}$. It 
is not hard to imagine that $d_{n, j'}$'s should be roughly equal for all $j'$'s. Then the expression in \eqref{eq:short_note1} evaluates to (approximately)
\begin{equation}
\label{eq:short_note2}
\gamma\frac{d_n}{\Gamma N}\E \Big( \sum_{(j, j') \in [2] \times [\Gamma / \beta]} \sum_{v \in I_{j, j'} \times \{\nu_{j, j', k_{1, j, j'}}\}}X_{2N, 1, j, v} \Big)\,,
\end{equation}
where $d_n$ is the expected weight of crossing through 
$V_{N, \Gamma, i}$. Note that each $X_{2N, 1, j, v}$ is a harmonic average of $\{\eta_{2N, w}: w\in \partial V_{2N, \Gamma; 1, j}\}$. Since $\beta$ and $\Gamma$ are large, we can effectively assume that the harmonic measure is supported on $I_{j, j'} \times \{2N\}$. In addition, we can assume that the sub-rectangle $V_{2N, \Gamma; 1, j, j'} = V_{2N, \Gamma; 1, j, j', 1} \cup V_{2N, \Gamma; 1, j, j', 2}$ is effectively an infinite strip from (a random walk started at) a ``typical'' $v$ in $I_{j, j'} \times {\nu_{j, j', k}}$. 
Thus the probability that a simple random walk starting from a typical vertex $v$ exits $V_{2N, \Gamma; 1, j}$ through $I_{j, j'} \times \{2N\}$ is approximately 
$\tfrac{2N - 1 - \nu_{j, j', k}}{2N - 1}$. Changing perspective, we see that the segment $I_{j, j'} \times \{\nu_{j, j', k}\}$ ``looks similar'' from any typical $w 
\in I_{j, j'} \times \{2N\}$. Hence the sum of coefficients of $\eta_{2N, w}$ (as obtained from all $v \in I_{j, j'} \times \{\nu_{j, j', k}\}$) for any typical $w$ is approximately $\tfrac{2N - 1 - \nu_{j, j', k}}{2N 
- 1}$. Thus, \eqref{eq:short_note2} can be further approximated by
\begin{equation}
\label{eq:short_note3}
\gamma\frac{d_n}{\Gamma N}\E \Big( \sum_{(j, j') \in [2] \times [\Gamma / \beta]} \frac{2N - 1 - \nu_{j, j', k_{1, j, j'}}}{2N - 1}\sum_{v \in I_{j, j'} \times \{2N\}}\eta_{2N, w} \Big)\,.
\end{equation}
Since the gadget joining $\pi_{1, j, j', 1}$ and $\pi_{1, j, j', 2}$ is constructed as a crossing through a rectangle whose longer dimension is $\nu_{j, j', 2} - \nu_{j, j', 1}$, its expected weight (conditioned on the heights $\nu_{j, j', k}$) is bounded approximately by 
$(\nu_{j, j', 2} - \nu_{j, j', 1})\tfrac{d_n}{\Gamma N}$. 
Therefore our net expected decrement from switchings is approximated by 
\begin{equation}
\label{eq:short_note4}
\E \Big(\gamma\frac{d_n}{\Gamma N}\E \Big( \sum_{(j, j') \in [2] \times [\Gamma / \beta]} \frac{2N - 1 - \nu_{j, j', k_{1, j, j'}}}{2N - 1}\sum_{v \in I_{j, j'} \times \{2N\}}\eta_{2N, w} \Big)+ \frac{d_n}{\Gamma N}\sum_{J'}(\nu_{j, j', 2} - \nu_{j, j', 1})\Big)\,,
\end{equation}
where $J'$ is the collection of pairs $(j, j')$ corresponding to the switching locations.
Now, one can show (see Lemma~\ref{lem:general_covar}) that $\var(\sum_{v \in I_{j, j'} \times \{2N\}}\eta_{2N, w}) \approx $ $4\beta N^2$ and that $\{\sum_{v \in I_{j, j'} \times \{2N\}}\eta_{2N, w}\}_{j, j'}$  is 
weakly-correlated. Thus, it is legitimate to replace $\{\sum_{v \in I_{j, j'} \times \{2N\}}\eta_{2N, w}\}_{j, j'}$ with $\{2\sqrt{\beta} NZ_{j, j'}\}_{j, j'}$ where $Z_{j, j'}$ are i.i.d.\ standard Gaussians 
independent with $\nu_{j, j', k}$'s. Combining with the observation that,
$$\E \frac{2N - 1 - \nu_{j, j', k_{1, j, j'}}}{2N - 1} Z_{j, j'} = \frac{1}{2}\E (-1)^{k_{1, j, j'} + 1}\frac{\nu_{j, j', 1} - \nu_{j, j', 2}}{2N - 1}Z_{j, j'}\,,$$
we can then approximate \eqref{eq:short_note4} by
\begin{equation}
\label{eq:short_note5}
\E \Big( \gamma\frac{d_n\sqrt{\beta}}{\Gamma} \sum_{(j, j') \in [2] \times [\Gamma / \beta]} (-1)^{k_{1, j, j'} + 1}\frac{\nu_{j, j', 1} - \nu_{j, j', 2}}{2N - 1}Z_{j, j'} + \frac{d_n}{\Gamma N}\sum_{J'}(\nu_{j, j', 2} - \nu_{j, j', 1})\Big)\,.
\end{equation}
By Theorem~\ref{thm-total-variation}, there exists a switching strategy such that the (conditional) expectation of the expression inside the parentheses in \eqref{eq:short_note5} is (roughly) at most
\begin{equation}\label{eq:short_note6}
\frac{-d_n^2\beta\gamma^2}{\Gamma^2}\sum_{(j, j') \in [2] \times [\Gamma / \beta]}\Big( \frac{\nu_{j, j', 1} - \nu_{j, j', 2}}{2N - 1}\Big)^2\frac{1}{\frac{d_n}{\Gamma N} (\nu_{j, j', 1} - \nu_{j, j', 2})} \approx \frac{-d_n\beta\gamma^2}{2\Gamma}\sum_{(j, j') \in [2] \times [\Gamma / \beta]}\frac{\nu_{j, j', 1} - \nu_{j, j', 2}}{2N - 1}\,.
\end{equation}
Since our switching strategies are symmetric at every scale, we have $\E \tfrac{\nu_{j, j', 1} - \nu_{j, j', 2}}{2N - 1} \approx 
1/2$. Therefore, the expectation of the right hand side in \eqref{eq:short_note6} is close to 
$-0.5d_n\gamma^2$. Combined with \eqref{eq-expansion}, we get 
\begin{equation*}
d_{n+1, \gamma} \precsim d_{n, \gamma} (2 + 0.44 \gamma^2 - 0.5\gamma^2) \leq d_{n, \gamma} (2 - 0.06\gamma^2)\,.
\end{equation*}
This implies that the weight exponent for Liouville FPP is strictly less than 1.

\subsection{Notation convention}
Any \emph{fixed} number in this paper will be implicitly assumed to be independent of $\gamma$ or any other 
variable. 
Let $\delta \ll 1$ be a fixed positive number whose exact value is to be decided and let $\alpha = 
\delta^{-1/4}$. Choose $\Gamma$ as the smallest (integral) power of $2 + \delta$ that is $\geq \alpha / 
\gamma^2$. Thus $\Gamma = (2 + \delta)^{m_\Gamma}$ for some positive integer $m_\Gamma$ and $\alpha \leq \Gamma \gamma^2 < (2 + \delta)\alpha$. We denote the number $(2 + \delta)^m$ as $a_m$ where $m \in \Z$. For a subset $S$ of $\R^d$, let $\lfloor S \rfloor$ denote the set $S \cap \Z^d$. If $S = [1, \ell]$ for some $\ell \in \N$, 
then we denote it simply as $[\ell]$. A \emph{$\R$-interval} is the usual interval considered as 
a subset of the real line. An \emph{integer interval} or simply an interval is the set $\lfloor [\ell, r] 
\rfloor$ where $\ell, r \in \Z$. The right and left endpoints of an interval $I$ are denoted as $r_I$ and 
$p_I$ respectively. The \emph{length} of $I$ is the difference $r_I - p_I$. We refer to the vertices of $\Z^2$ (when considered as a graph) as \emph{points}. If $z \in \Z^2$, then $z_x$ and $z_y$ respectively denote the horizontal and vertical coordinates of $z$. For $\nu \in \Z$ the lines $y = \nu$ and $x = \nu$ are denoted by $\Hl_{\nu}$ and $\Vl_{\nu}$ respectively. 
For any $A \subseteq \Z^2$, we also use $A$ to denote the corresponding induced subgraph of $\Z^2$. \emph{Interior} of a subset $A$ of $\Z^2$, denoted as $\inte(A)$, is defined as the set of all points in $A$ 
whose neighborhood is also contained in $A$. The \emph{boundary} of $A$, denoted as $\partial A$, is the set of all points in $A \setminus \inte(A)$ that have at 
least one neighbor in $\inte(A)$. All the rectangles in this paper will be assumed to have sides parallel to the 
coordinate axes. The left, right, top and bottom boundaries of a rectangle $R$ are denoted as $\partial_{\mathrm{left}}R,\partial_{\mathrm{right}}R,\partial_{\up}R$ and $\partial_{\down}R$ respectively. 
We call the horizontal range of a rectangle as its \emph{base} and the vertical range as its \emph{span}. 
Thus the base and span of a rectangle are both intervals 
in $\Z$. A \emph{left-right crossing} or simply a \emph{crossing} of a rectangle $R$ is a connected subset $A$ of $R$ that intersects both $\partial_{\mathrm{left}}R$ and 
$\partial_{\mathrm{right}}R$. Similarly we can define 
\emph{up-down crossing}. For purely technical purpose we allow $A$'s to be multisets in which case the corresponding sets are required to be connecting sets. 
However we still call them crossings. If $\{X_v\}_{v \in A}$ is a stochastic process indexed by $A \subseteq \Z^2$ and $B \subseteq A$, then $X_B$ denotes the 
collection of random variables $\{X_v\}_{v \in B}$. For (nonnegative) functions $F(.)$ and $G(.)$ we write $F = O(G)$ (or $\Omega(G)$) if there exists an absolute constant $C > 0$ such that $F \leq C G$ (respectively 
$\geq C G$) everywhere in the domain. If the constant $C$ depends on variables $x_1, x_2, \ldots, x_n$, we modify these notations as $O_{x_1,x_2, \ldots, x_n}(G)$ 
and $\Omega_{x_1,x_2, \ldots, x_n}(G)$ respectively. In a similar vein we write $F = o_{x_0 \to c;x_1,\ldots,x_n}(1)$ if $F$ is a $\R$-valued function with arguments $x_0, x_1, \ldots, x_{n'}$ for some $n' \geq n$ and $\lim_{x_0 \to c \in \overline{\mathbb{R}}}\sup_{x_{n+1}, \ldots, x_{n'}}|F(x_0, x_1, \ldots, x_{n'})| = 0$ for any given 
values of the variables $x_1, x_2, \ldots, x_n$. We denote by $o_{x_0 \to c;x_1,\ldots,x_n}(G)$ any function $F$ that satisfies $F = o_{x_0 \to c;x_1,\ldots,x_n}(1)G$.

\subsection{Acknowledgement}

We thank Marek Biskup, Hugo Duminil-Copin, Steve Lalley,  Elchanan Mossel, R{\'e}mi Rhodes, Vincent Vargas, Ofer Zeitouni for helpful discussions. We also thank Alexander Dunlap for many useful discussions and help with the figures in this paper.

\section{Preliminaries}\label{sec:prelim}
This section is devoted to foundational results that are needed for our multi-scale analysis carried out in Section~\ref{sec:the_main}.
\subsection{A self-similar partition of an interval in $\R$}
\label{subsec:self_similar}
Let $\delta > 0$ be chosen such that $a_m = 1 / \delta$ 
for a fixed integer $m \gg 1$. Evidently such a number is unique. Now for $\ell \in \Z, k \in \N$ and $x \in \R$, consider the $\R$-interval $\mathcal I_{\ell, k, x} = x + [0, k a_\ell]$ (so for any given $\ell$ the intervals $\mathcal I_{\ell, k, x}$'s are translates of 
each other). We can see from the definition of 
$\delta$ that $\mathcal I_{\ell, k, x}$ is a union of 
three contiguous intervals with disjoint interiors whose lengths are $k a_{\ell - 1}, k a_{\ell - m - 1}$ and $k 
a_{\ell - 1}$ respectively from left to right. See Figure~\ref{fig:partition} for an illustration. Thus we get a partition of $\mathcal I_{\ell, k, x}$ into three subintervals each of which is a translate of $\mathcal 
I_{\ell', k, 0}$ for some $\ell'$. Hence we can partition each of these subintervals in a similar 
fashion. Suppose we apply this procedure to the subintervals we obtain in each step as long as 
their lengths are bigger than $k a_{\ell - d}$ where $d \in \N$. Denote the resulting partition of $\mathcal I_{\ell, k, 
x}$ by $\mathscr P_{\ell, k, x; d}$. It is not difficult to see that $|\mathscr P_{\ell, k, x; d}| \leq (2 + 
\delta)^{d + m}$. If we discard the ``middle'' segments at each stage of the partitioning, we get a different collection of intervals called $\mathscr P_{\ell, k, x; 
d, \principal}$. Notice that each ($\R$-) interval in $\mathscr P_{\ell, k, x; d, \principal}$ has the same 
length $ka_{\ell - d}$. We denote the union of these intervals as $\mathcal I_{\ell, k, x; d, \principal}$.

\begin{figure}[!htb]
\centering
\begin{tikzpicture}[scale = 3.5]
\draw [red] (0, 0) -- (0.5625, 0);
\draw       (0.5625, 0) -- (0.9375, 0);
\draw [red] (0.9375, 0) -- (1.5, 0);
\draw       (1.5, 0) -- (2.5, 0);
\draw [red] (2.5, 0) -- (3.0625, 0);
\draw       (3.0625, 0) -- (3.4375, 0);
\draw [red] (3.4375, 0) -- (4, 0);

\draw (1.5,-0.06) -- (1.5, 0.06);
\draw (2.5,-0.06) -- (2.5, 0.06);
\draw (0.5625,-0.04) -- (0.5625, 0.04);
\draw (0.9375,-0.04) -- (0.9375, 0.04);
\draw (1.875,-0.04) -- (1.875, 0.04);
\draw (2.125,-0.04) -- (2.125, 0.04);
\draw (3.0625,-0.04) -- (3.0625, 0.04);
\draw (3.4375,-0.04) -- (3.4375, 0.04);

\draw [<->, thin] (0.01, 0.15) -- (0.5525, 0.15);
\node [scale = 0.55, above] at (0.28125, 0.15) {$ ka_{\ell - 2}$};

\draw [<->, thin] (0.5725, 0.15) -- (0.9275, 0.15);
\node [scale = 0.55, above] at (0.75, 0.15) {$k a_{\ell - m - 2}$};

\draw [<->, thin] (0.9475, 0.15) -- (1.49, 0.15);
\node [scale = 0.55, above] at (1.21875, 0.15) {$k a_{\ell - 2}$};

\draw [<->, thin] (0.01, 0.3) -- (1.49, 0.3);
\node [scale = 0.7, above] at (0.75, 0.3) {$k a_{\ell - 1}$};

\draw [<->, thin] (1.51, 0.15) -- (1.865, 0.15);
\node [scale = 0.55, above] at (1.6875, 0.15) {$k a_{\ell - m - 2}$};

\draw [<->, thin] (1.885, 0.15) -- (2.115, 0.15);
\node [scale = 0.55, above] at (2, 0.15) {$k a_{\ell - 2m - 2}$};

\draw [<->, thin] (2.135, 0.15) -- (2.49, 0.15);
\node [scale = 0.55, above] at (2.3125, 0.15) {$k a_{\ell - m - 2}$};

\draw [<->, thin] (1.51, 0.3) -- (2.49, 0.3);
\node [scale = 0.7, above] at (2, 0.3) {$k a_{\ell - m - 1}$};

\draw [<->, thin] (2.51, 0.15) -- (3.0525, 0.15);
\node [scale = 0.55, above] at (2.78125, 0.15) {$k a_{\ell - 2}$};

\draw [<->, thin] (3.0725, 0.15) -- (3.4275, 0.15);
\node [scale = 0.55, above] at (3.25, 0.15) {$k a_{\ell - m - 2}$};

\draw [<->, thin] (3.4475, 0.15) -- (3.99, 0.15);
\node [scale = 0.55, above] at (3.71875, 0.15) {$k a_{\ell - 2}$};

\draw [<->, thin] (2.51, 0.3) -- (3.99, 0.3);
\node [scale = 0.7, above] at (3.25, 0.3) {$k a_{\ell - 1}$};

\draw [<->, thin] (0.01, 0.45) -- (3.99, 0.45);
\node [scale = 0.8, above] at (2, 0.45) {$k a_\ell$};
\end{tikzpicture}
\caption{{\bf Nesting of the intervals $\mathcal I_{\ell', k, x}$'s for three successive levels.} The leftmost point of the interval is $0$. The subintervals colored in red lie in $\mathscr P_{\ell, k, 0; 2, \principal}$.}
\label{fig:partition}
\end{figure}
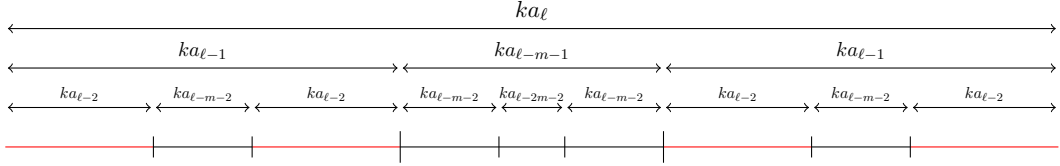
We can use this partitioning scheme to obtain a \emph{self-similar covering} $\mathscr C_{\ell, 
k, x; d}$ of $\lfloor \mathcal I_{\ell, k, x} \rfloor$. 
This covering will be defined in a recursive manner 
starting with $\mathscr C_{\ell, k, x; 1}$. In order to 
avoid cumbersome notations let us assume $x = 0$. As already described in the previous paragraph, $\mathscr P_{\ell, k, 0; 1}$ consists of three subintervals namely $\mathcal I_{\ell-1, k, 0}$, $\mathcal I_{\ell- m -1, k, ka_{\ell - 1}}$ and $\mathcal I_{\ell-1, k, k(a_{\ell-1} + a_{\ell - m - 1})}$, aligned from left to right. 
We first include the intervals $\lfloor \mathcal I_{\ell-1, k, 0}\rfloor$ and $\lfloor \mathcal I_{\ell - m - 1, k, \lceil k a_{\ell-1} \rceil}\rfloor$ in $\mathscr C_{\ell, k, 0; 1}$. 
As to $\mathcal I_{\ell-1, k, k(a_{\ell-1} + a_{\ell - m - 1})}$, notice that there is a unique integer $p$ in the set $\{\lfloor k(a_{\ell-1} + a_{\ell - m - 1})\rfloor, \lceil k(a_{\ell-1} + a_{\ell - m - 1})\rceil\}$ such that the right endpoint of $\mathcal I_{\ell - 1, k, p}$ lies in the interval $[\lfloor ka_n \rfloor ,$ $\lfloor ka_n \rfloor 
+ 1)$. We finish the construction of $\mathscr C_{\ell, k, 0; 1}$ by including $\lfloor \mathcal I_{\ell - 1, k, p} \rfloor$. 
Now suppose that we have defined $\mathscr C_{\ell, k, 0; d-1}$ for some $d \geq 2$ and that each interval in $\mathscr C_{\ell, k, 0; d-1}$ is a copy of $\lfloor \mathcal I_{\ell - d', k, 0} \rfloor$ 
for some $d - 1 \leq d' \leq d + m - 1$. If $d' \geq d$ for such a interval, we simply include it in $\mathscr 
C_{\ell, k, 0; d}$. Otherwise if $d' = d - 1$, we apply the same procedure to the corresponding interval as we did to $\lfloor \mathcal I_{\ell, k, 0} \rfloor$ in the first step and include the new intervals so obtained in 
$\mathscr C_{\ell, k, 0; d}$. If at each stage of the construction described above, we discard the intervals corresponding to the middle segments, we would end up with a particular sub-collection of $\mathscr C_{\ell, k, x; d}$ called $\mathscr C_{\ell, k, x; d, \principal}$. 
Notice that each interval in $\mathscr C_{\ell, k, x; d, \principal}$ has cardinality $\lfloor k a_{\ell - d}\rfloor 
+ 1$.




\subsection{A hierarchical representation of discrete GFF on a family of rectangles}\label{subsec:GFF_represent}
Denote by $\tilde{V}_\ell^\Gamma$ the rectangle $\big([-\lfloor \Gamma a_{\ell - m - 1}\rfloor, \lfloor \Gamma a_\ell \rfloor + \lfloor \Gamma a_{\ell - m - 1}\rfloor] \times \lfloor [-\lfloor a_{\ell - m}\rfloor, \lfloor a_{\ell+1} \rfloor + \lfloor a_{\ell - m}\rfloor]\big) \cap \Z^2$ and by $\tilde{V}_\ell^{\Gamma, z}$ the translation of $\tilde{V}_\ell^\Gamma$ by a point $z \in 
\Z^2$. Let $\{\eta_{n, v}: v\in \tilde{V}_n^\Gamma\}$ be a discrete GFF on $\tilde{V}_n^\Gamma$ with Dirichlet boundary condition.

We can define a multilevel scheme for placing nested rectangles inside $\tilde{V}_n^\Gamma$ using the coverings $\mathscr C_{n, \Gamma, x; r}$'s. 
Figure~\ref{fig:Tiling} gives an illustration for the 
very top level, that is, level $n$. In this figure we have placed four copies of $\tilde{V}_{n-1}^\Gamma$ (called $\tilde V_{n; i, j}^\Gamma$, $i, j \in [2]$) inside $\tilde{V}_{n}^\Gamma$ each of which contains two copies of the rectangle $\lfloor \mathcal I_{n-1, \Gamma, 0} \rfloor \times \lfloor \mathcal I_{n-1, 1, 0}\rfloor$. 
All the placements are carried out in symmetric fashion. 
More precisely lower left corner vertex of any copy of $\lfloor \mathcal I_{n-1, \Gamma, 0} \rfloor \times \lfloor \mathcal I_{n-1, 1, 0}\rfloor$ has the form $(q_\Gamma, q_1)$ where $\lfloor\mathcal I_{n-1,1, q_1}\rfloor$ and $\lfloor\mathcal I_{n - 1,\Gamma, q_\Gamma}\rfloor$ are intervals in $\mathscr C_{n+1,1,0;2}$ and $\mathscr C_{n,\Gamma,0;1}$ 
respectively. For each such $q_\Gamma$ (there are 2 of them) we have two copies of $\tilde{V}_{n-1}^\Gamma$ namely $\tilde{V}_{n-1}^{\Gamma, (\lceil q_\Gamma \rceil, 0)}$ and $\tilde{V}_{n-1}^{\Gamma, (\lceil 
q_\Gamma \rceil, p )}$. Here $p$ is the left endpoint of the rightmost interval in $\mathscr C_{n+1, 1, 0; 1}$.
Also for each interval $\lfloor \mathcal I_{n - m, 1, p} \rfloor$ in $\mathscr C_{n, 1, 0;m, \principal}$ or $\mathscr C_{n, 1, \lceil r_{n+1, 1} - a_{n}\rceil;m, \principal}$, we have placed the rectangle $\tilde{V}_{n - m - 1}^{\Gamma, z}$ with $z_x = \lceil \Gamma (a_{n - 
m - 1})\rceil$ and $z_y = \lceil p \rceil$. We repeat the same placement procedure inside all these rectangles. 
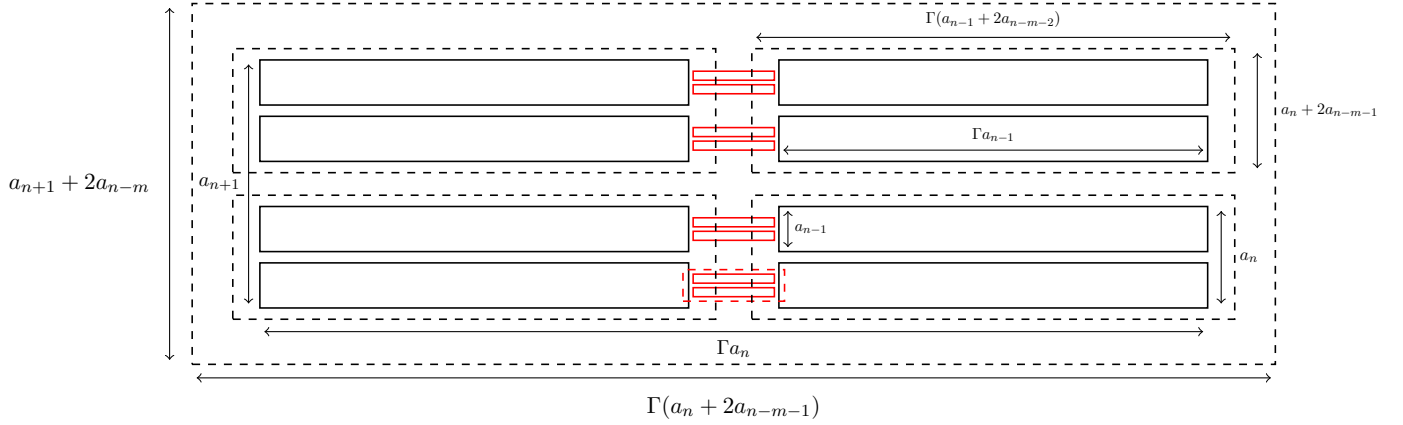
\begin{figure}[!htb]
\centering
\begin{tikzpicture}[semithick, scale = 3]
\draw (-2.1, -0.55) rectangle (-0.2, -0.35);
\draw  (0.2, -0.55) rectangle (2.1, -0.35);

\draw  (-2.1, -0.3) rectangle (-0.2, -0.1);
\draw  (0.2, -0.3) rectangle (2.1, -0.1);

\draw  (-2.1, 0.35) rectangle (-0.2, 0.55);
\draw  (0.2, 0.35) rectangle (2.1, 0.55);

\draw  (-2.1, 0.1) rectangle (-0.2, 0.3);
\draw  (0.2, 0.1) rectangle (2.1, 0.3);

\draw [red] (-0.18, -0.44) rectangle (0.18 ,-0.4);
\draw [red] (-0.18, -0.5) rectangle (0.18 ,-0.46);
\draw [dashed, red] (-0.225, -0.52) rectangle (0.225 ,-0.38);

\draw [red] (-0.18, -0.19) rectangle (0.18 ,-0.15);
\draw [red] (-0.18, -0.25) rectangle (0.18 ,-0.21);

\draw [red] (-0.18, 0.44) rectangle (0.18 ,0.4);
\draw [red] (-0.18, 0.5) rectangle (0.18 ,0.46);

\draw [red] (-0.18, 0.19) rectangle (0.18 ,0.15);
\draw [red] (-0.18, 0.25) rectangle (0.18 ,0.21);

\draw[dashed] (-2.4, -0.8) rectangle (2.4, 0.8);

\draw[dashed] (0.08, 0.05) rectangle (2.22, 0.6);
\draw[dashed] (-0.08, 0.05) rectangle (-2.22, 0.6);
\draw[dashed] (-0.08, -0.05) rectangle (-2.22, -0.6);
\draw[dashed] (0.08, -0.05) rectangle (2.22, -0.6);

\draw [<->, thin] (0.22, 0.15) -- (2.08, 0.15);

\node[scale = 0.55, above] at (1.15, 0.16) {$\Gamma a_{n - 1}$};

\draw [<->, thin] (-2.38, -0.86) -- (2.38, -0.86);

\node[scale = 0.8, below] at (0, -0.9) {$\Gamma (a_n + 2 a_{n - m - 1})$};

\draw [<->, thin] (-2.5, -0.78) -- (-2.5, 0.78);

\node[scale = 0.8, left] at (-2.55, 0) {$a_{n + 1} + 2a_{n - m}$};

\draw [<->, thin] (0.1, 0.65) -- (2.2, 0.65);

\node[scale = 0.55, above] at (1.15, 0.67) {$\Gamma (a_{n - 1} + 2 a_{n - m - 2})$};

\draw [<->, thin] (2.32, 0.07) -- (2.32, 0.58);

\node[scale = 0.55, right] at (2.4, 0.325) {$a_n + 2a_{n - m - 1}$};

\draw [<->, thin] (-2.08, -0.655) -- (2.08, -0.655);
\node[scale = 0.7, below] at (0, -0.66) {$\Gamma a_n$};

\draw [<->, thin] (2.16, -0.53) -- (2.16, -0.12);
\node[scale = 0.6, right] at (2.215, -0.325) {$a_n$};

\draw [<->, thin] (0.24, -0.28) -- (0.24, -0.12);
\node[scale = 0.55, right] at (0.245, -0.2) {$a_{n - 1}$};

\draw [<->, thin] (-2.15, -0.53) -- (-2.15, 0.53);
\node[scale = 0.7, left] at (-2.155, 0) {$a_{n + 1}$};

\end{tikzpicture}
\caption{{\bf The nesting structure of rectangles 
between levels $n - 1$ and $n$.} The number next to an arrow represents the length of the corresponding 
$\R$-interval. The four copies of $\tilde{V}_{n-1}^\Gamma$ have been indicated by black broken boundary lines while the eight copies of $\lfloor \mathcal I_{n-1, \Gamma, 0}\rfloor \times \lfloor \mathcal I_{n-1, 1, 0}\rfloor$ have been indicated by black solid boundary lines. The two copies of $\tilde V_{n-1}^\Gamma$ on the top (bottom) are $\tilde V_{n; 1, 1}^\Gamma$ and $\tilde V_{n; 1, 2}^\Gamma$ (respectively $\tilde V_{n; 2, 1}^\Gamma$ and $\tilde V_{n; 2, 2}^\Gamma$) from left to right. The rectangle with red broken boundary lines is a copy of $\tilde{V}_{n - m - 1}^\Gamma$. The rectangles with red solid boundary lines are copies of $\lfloor \mathcal I_{n- m - 1, \Gamma, 0}\rfloor \times \lfloor \mathcal I_{n-m-1, 1, 0}\rfloor$.}
\label{fig:Tiling}
\end{figure}

An alternative way to describe this placement scheme is
to visualize it as a tree. We begin with $\tilde{V}_{n}^\Gamma$ as the root node. The successors of a rectangle in the tree are the rectangles that are 
placed immediately inside it. Since rectangles of different dimensions co-occur at any stage of the placement, we choose the depths of the successor nodes according to their vertical (or equivalently horizontal) 
side lengths. Thus each node in the tree has two attributes viz. \emph{depth} and \emph{level}. The depth of a node indicates the dimension of the rectangle it 
corresponds to while the level indicates its lineage. We denote this tree as $\T_n$ and henceforth we will use \emph{node} interchangeably with the 
corresponding rectangle. Notice that two nodes from 
different levels may have the same depth. Also two rectangles from different branches may have (very 
slight) overlaps. We refer the reader to Fig~\ref{fig:Tree} for an illustration.
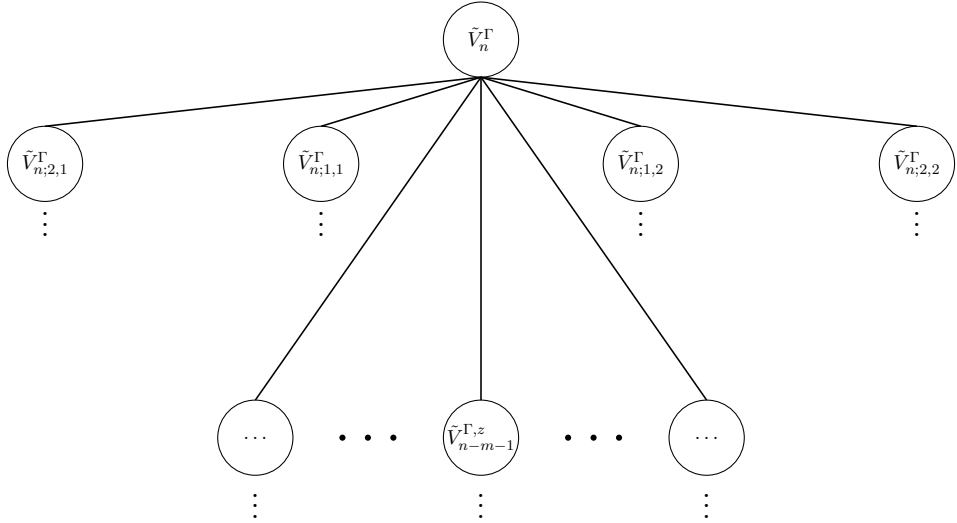
\begin{figure}[!htb]
\centering
\begin{tikzpicture}[scale = 2]
\draw  (0, 0.05) circle [radius = 0.25];
\node [scale = 0.7] at (0, 0.05) {$\tilde{V}_n^\Gamma$};

\draw  (0.9659258*3, - 0.258819*3) circle [radius = 0.25];
\node [scale = 0.7] at (0.9659258*3, - 0.258819*3) {$\tilde{V}_{n;2,2}^\Gamma$};
\node [scale = 1] at (0.9659258*3, - 0.258819*3 - 0.15 - 0.2) {$\vdots$};

\draw  (-3*0.9659258, -3*0.258819) circle [radius = 0.25];
\node [scale = 0.7] at (-0.9659258*3, - 0.258819*3) {$\tilde{V}_{n;2,1}^\Gamma$};
\node [scale = 1] at (-0.9659258*3, - 0.258819*3 - 0.15 - 0.2) {$\vdots$};

\draw (0.9659258*1.1, - 0.258819*3) circle [radius = 0.25];
\node [scale = 0.7] at (0.9659258*1.1, - 0.258819*3) {$\tilde{V}_{n;1,2}^\Gamma$};
\node [scale = 1] at (0.9659258*1.1, - 0.258819*3 - 0.15 - 0.2) {$\vdots$};

\draw (-0.9659258*1.1, -3*0.258819) circle [radius = 0.25];
\node [scale = 0.7] at (-0.9659258*1.1, - 0.258819*3) {$\tilde{V}_{n;1,1}^\Gamma$};
\node [scale = 1] at (-0.9659258*1.1, - 0.258819*3 - 0.15 - 0.2) {$\vdots$};

\draw [semithick] (0, -0.2) -- (0.9659258*3, - 0.258819*3 + 0.25);
\draw [semithick] (0, -0.2) -- (-0.9659258*3, - 0.258819*3 + 0.25);
\draw [semithick] (0, -0.2) -- (-0.9659258*1.1, - 0.258819*3 + 0.25);
\draw [semithick] (0, -0.2) -- (0.9659258*1.1, - 0.258819*3 + 0.25);

\draw  (-0.5*3, -0.8660254*3) circle [radius = 0.25];
\draw [semithick] (0, -0.2) -- (-0.5*3, -0.8660254*3 + 0.25);
\node [scale = 0.7] at (-0.5*3, -0.8660254*3) {$\hdots$};

\draw  (0.5*3, -0.8660254*3) circle [radius = 0.25];
\draw [semithick] (0, -0.2) -- (0.5*3, -0.8660254*3 + 0.25);
\node [scale = 0.7] at (0.5*3, -0.8660254*3) {$\hdots$};

\draw  (0, -0.8660254*3) circle [radius = 0.25];
\draw [semithick] (0, -0.2) -- (0, -0.8660254*3 + 0.25);
\node [scale = 0.7] at (0, -0.8660254*3) {$\tilde{V}_{n - m - 1}^{\Gamma, z}$};

\node [scale = 2] at (-0.5*1.5, -0.8660254*3) {$\hdots$};
\node [scale = 2] at (0.5*1.5, -0.8660254*3) {$\hdots$};

\node [scale = 1] at (-0.5*3, -0.8660254*3 - 0.4) {$\vdots$};
\node [scale = 1] at (0, -0.8660254*3 - 0.4) {$\vdots$};
\node [scale = 1] at (0.5*3, -0.8660254*3 - 0.4) {$\vdots$};
\end{tikzpicture}
\caption{{\bf The tree representation of the placement scheme.} Only two topmost levels have been shown. The four rectangles that are nearer to the root are copies of $\tilde V_{n-1}^\Gamma$ and the ones that are farther down are copies of $\tilde V_{n-1}^\Gamma$.}
\label{fig:Tree}
\end{figure}

We next choose a convention to describe the level and depth of nodes in $\T_n$. We enumerate the levels of 
nodes downwards starting with $n$ for the root node. As for the depth of a node $B$, we define it to be $d$ if $B = \tilde{V}_d^{\Gamma, z}$ for some $z \in \Z^2$. 
Denote by $\B_\ell$ the collection of all rectangles at 
level $\ell$. Except for $\ell = n$, $\B_\ell$ can be further divided into two sub-collections, namely $\B_{\ell; \principal}$ and $\B_{\ell; \midd}$, which consist of rectangles at depth $\ell$ and $\ell - m$ respectively. 
Let $B$ be a member of $\B_\ell$ with depth $d$. 
Denote the four rectangles of depth $d - 1$ 
that are placed inside $B$ by $\{B_{i, j}\}_{i \in 
[2], j \in [2]}$ in the usual order (see 
Figure~\ref{fig:Tiling}). Each $B_{i, j}$ contains two copies of $\lfloor \mathcal I_{d-1, \Gamma, 0} \rfloor \times \lfloor \mathcal I_{d-1, 1, 0}\rfloor$, denoted as $B_{i, j, 1}$ (the upper one) and $B_{i, j, 2}$ (the 
bottom one). Similarly denote by $B_{1}$ and $B_{2}$, the two copies of the rectangle $\lfloor \mathcal I_{d, \Gamma, 0} \rfloor \times \lfloor \mathcal I_{d, 1, 0}\rfloor$ placed inside $B$. 
There are also $2^{m-1}$ copies of $\tilde{V}^\Gamma_{d - m - 1}$ adjacent to each $B_{i, 1}$ and we refer to 
that collection as $B_{\midd, i}$. Let us state another important notation convention in this regard:  Often the notation for a particular rectangle $B$ 
might already involve some subscripts. Then the new subscripts will be appended to the end of the existing subscripts followed by a semicolon. 
For instance,  if $B = \tilde{V}_{\ell}^\Gamma$, the notation for $B_{i, j}$ would be $\tilde{V}_{\ell; i, j}^\Gamma$.

We now decompose the GFF along each branch of the tree. Since a point may belong to more than one rectangle at the same level, we may get different series 
of fields for the same point. For integers $\ell' < \ell \leq n$ and a rectangle $B \in \B_{\ell'}$, define the field $\{X_{n, B, \ell, v}\}_{v \in B}$ by
$$X_{n, B, \ell, v} = \E(\eta_{n, v} \mid \eta_{n, \partial B}) - \E(\eta_{n, v} \mid \eta_{n, \partial B^{\ancest, \ell}})\,,$$
where $B^{\ancest, \ell}$ is the ancestor of $B$ at level $\ell$. Also define another field $\{\eta_{n, B, v}\}_{v \in B}$ as
$$\eta_{n, B, v} = \eta_{n, v} - \E(\eta_{n, v} \mid \eta_{n, \partial B})\,.$$
If $B$ is $B'_i$ for some $B' \in \B_{\ell'}$ and $i' \in [2]$, then we use the same notations to denote the fields $\{X_{n, B', \ell, .}\}$ and $\{\eta_{n, B', 
.}\}$ when restricted to $B$. We often refer to the fields $\{X_{n, B, \ell, .}\}$ as \emph{coarse fields} and the field $\{\eta_{n, B, .}\}$ as the \emph{fine 
field} on $B$. In the particular case when $B = \tilde{V}^\Gamma_{\ell'}$, we modify these notations as $\{X_{n, \ell', \ell, .}\}$ and $\{\eta_{n, \ell', .}\}$ 
respectively. The following observation is based on the Markov property of the GFF.
\begin{observation}
\label{observ:indep}
The processes $\{X_{n, B, \ell, v}\}_{v \in B}$ and $\{X_{n, B^{\ancest, \ell}, \ell'', v}\}_{v \in B^{\ancest, \ell}}$ are independent Gaussian processes 
for all $\ell'' > \ell$. Furthermore the field $\{\eta_{n, B, v}\}_{v \in B}$ is a GFF on $B$ with Dirichlet boundary condition that is independent with $\{X_{n, B, \ell, v}\}_{v \in B}$ for all $\ell > 
\ell'$. Notice, however, that $\{\eta_{n, B_1, v}\}_{v \in B_1}$ and $\{X_{n, B_2, \ell, v}\}_{v \in B_2}$ may not be independent for different $B_1, B_2 \in \B_{\ell'}$.
\end{observation}
One can similarly define the fine field $\{\eta_{n, B, v}\}_{v \in B}$ on any rectangle $B$ that is a subset of 
$\tilde V_n^\Gamma$. The coarse field $\{X_{n, B, v}\}_{v \in B}$ in this case would be $\E (\eta_{n, v} | \eta_{n, \partial B})$. These notions will be useful later for constructing crossings through rectangles that are not in $\B_\ell$.

In the remaining subsections we prove a few results that will be used when we prove our main theorem 
in section~\ref{sec:the_main}. The reader might opt to skip these subsections and move directly to section~\ref{sec:total_variation} (and come back to them when they are used).
\subsection{Some inequalities involving extreme values of stochastic processes}
We record a few standard results in this subsection.
\begin{lemma}\cite[Theorem 7.1, Equation (7.4)]{L01}
\label{lem:Borell_ineq}
Consider a centered Gaussian process $\{X_v: v \in A\}$, with $A$ finite, and set $\sigma^2 = 
\sup_{v \in A}\E X_v^2$. Let $X_A^* = \max_{v \in A}X_v$. Then, for $x > 0$,
$$\P(| X_A^*  - \E X_A^*| \geq x) \leq 2 \e^{-x^2 / 2 \sigma^2}\,.$$
\end{lemma}

\begin{lemma}\cite[Theorem 4.1]{A90}
\label{lem:basic_chaining}
Let $(S, d)$ be a (finite) metric space such that $\max_{s, t \in S} d(s, t) = 1$. Suppose that there exist positive numbers $\beta$ and $C_1$ such that $N_\epsilon(S, d) \leq C_1\epsilon^{-\beta}$ for all $\epsilon \in (0, 1]$ where $N_\epsilon(S, d)$ is the $\epsilon$-covering number of $(S, d)$. If $\{X_s\}_{s \in S}$ is a centered Gaussian process satisfying
$$\sqrt{\E(X_s - X_{s'})^2} \leq C_2d(s, s')^\alpha, \mbox{  for all } s, s' \in S \mbox{ and some } \alpha, C_2 > 0\,,$$
then
$$\E \max_{s \in A} X_s \leq C_2(\sqrt{\beta\log 2} + \sqrt{\log (C_1 + 1)})C_\alpha\,,$$
where $C_\alpha = \sum_{n \geq 0}\sqrt{n + 1}2^{-n\alpha}$.
\end{lemma}

As a consequence of Lemma~\ref{lem:basic_chaining} we get the following lemma which we will use repeatedly.
\begin{lemma} 
\label{lem:max_expectation}
Let $B_1, B_2, \ldots, B_N$ be squares of side lengths $b_1, b_2, \ldots, b_N$ respectively and $B = \cup_{j 
\in [N]}B_j$. Suppose that $\{X_v\}_{v \in B}$ is a centered Gaussian processes satisfying
$$\E(X_u - X_v)^2 \leq |u -v| / b_j, \mbox{  whenever } u, v \in B_j \mbox{ for some } j \in [N]\,.$$
Then there exists an absolute constant $C' > 0$ such that
$$\E \mbox{$\max_{v \in B}$} X_v \leq C'\sqrt{\log N}(1 + \mbox{$\max_{v \in B}$}\sqrt{\E X_v^2}) + C'\,.$$
\end{lemma}
\subsection{Some results on simple random walk in $\Z^2$}
\label{subsec:random_walk}
In this subsection we will present some results on simple random walk in $\Z^2$. 
First we need some notations. Denote by $\{S_t\}_{t \geq 0}$ a simple random walk in $\Z^2$ and by $\P^v$ the measure 
corresponding to the random walk starting from $v$. Let $A \subseteq \Z^2$ and $\tau_A = \min \{t \geq 0: S_t 
\notin \inte(A)\}$. For $x \in \inte(A), y \in \partial A$, define the Poisson kernel $H_{A}(x, y)$ as $\P^x 
(S_{\tau_{A}} = y)$. The simple random walk Green's function $G_A(x, y)$ is defined as $\E^x (\sum_{t = 
0}^{\tau_A - 1}\mathbf{1}_{\{S_t = y\}})$. For positive integers $M$ and $N$ we denote the rectangle $([0, M] \times [0, N]) \cap \Z^2$ as $R_{M, N}$, whereas for $z \in \Z^2$ we denote by $R_{M, N}^z$ the rectangle $z + R_{M, N}$.

In the next a few lemmas we will heavily use the following exact expression for $H_{R_{M, N}}(.,.)$.
\begin{proposition}\cite[Proposition 8.1.5]{Lawler10}
\label{prop:poisson_ker}
For $(x, y) \in \inte(R_{M, N})$ and $y_1 \in [N-1]$,
\begin{eqnarray*}
H_{R_{M, N}}((x, y), (0, y_1)) &=& H_{R_{M, N}}((M - x, y), (M, y_1))\\ 
&=&\frac{2}{N}\sum_{j = 1}^{N - 1} 
\frac{\sinh(r(\tfrac{j\pi}{N})(M - x))}{\sinh (r(\tfrac{j\pi}{N})M))}\sin 
\Big(\frac{j\pi y}{N}\Big)\sin \Big(\frac{j\pi y_1}{N}\Big)\,,
\end{eqnarray*}
where $r(t)$ is the even function $\cosh^{-1}(2 - \cos t)$. 
\end{proposition}
The function $r(t)$ is almost linear for small values of $t$ as shown below.
\begin{lemma}
\label{lem:r_t}
$r(t) = t + O(t^3)$ on $[0, 1]$ and $r(t) \geq t/4$ on $[0, \pi]$.
\end{lemma}
\begin{proof}
The first assertion follows from the fact that $\cosh^{-1}(1 + x) = \sqrt{2x} + O(x^{3/2})$ as $x \to  
0+$ (see \cite[Chapter 8]{Lawler10}). For the lower bound consider the function $f(t) = 2 - \cos t - \cosh 
\tfrac{t}{4}$. Then $f(0) = 0$ and
$$f'(t) = 2 + \sin t - \frac{\sinh \tfrac{t}{4}}{4} \geq 2 - \frac{\mathrm e}{8} \geq 1$$
on $[0, \pi]$. Thus $(2 - \cos t) \geq \cosh \tfrac{t}{4}$ on $[0, \pi]$. Combining this with the 
fact that $\cosh t$ is increasing for $t \geq 0$ yields the lemma.
\end{proof}
As a first application of Proposition~\ref{prop:poisson_ker}, we derive an upper bound on the probability that a simple random walk starting from a point inside $R_{\Upsilon N, N}$ exits 
it through one of the vertical boundaries. Here $\Upsilon$ is a positive number.
\begin{lemma}
\label{lem:kernel_bound1}
Let $v$ be a point in $\inte(R_{\Upsilon N, N})$ and $v_x' = \tfrac{v_x}{N}$. Then
$$\sum_{z \in \partial_{\mathrm{left}} R}H_{R_{\Upsilon N, N}} (v, z)\leq O(1)\e^{-\tfrac{\pi v_x'}{4}}\,.$$
\end{lemma}
\begin{proof}
From Proposition~\ref{prop:poisson_ker} we get,
\begin{eqnarray*}
\sum_{z \in \partial_{\mathrm{left}}}H_{R_{\Upsilon N, N}}(v, z) &=& 2N^{-1}\sum_{j = 1}^{N - 1} 
\frac{\sinh(r(\tfrac{j\pi}{N})(\Upsilon N - v_x))}{\sinh (r(\tfrac{j\pi}{N})\Upsilon N))}\sin 
\Big(\frac{j\pi v_y}{N}\Big)\sum_{k \in [N -1]}\sin \Big(\frac{j\pi k}{N}\Big)\\
& \leq & 2N^{-1}\sum_{j = 1}^{N - 1} \frac{\sinh(r(\tfrac{j\pi}{N})(\Upsilon N - v_x))}{\sinh (r(\tfrac{j\pi}{N})\Upsilon N))}\frac{1}{\sin \Big(\frac{j\pi }{2 N}\Big)}\\
& \leq & O(1)\sum_{j = 1}^{N - 1} \frac{\e^{-\frac{j\pi v_x'}{4}}}{1 - \e^{-\frac{j\pi \Upsilon}{2}}} \mbox{\hspace{0.2 cm} (from Lemma~\ref{lem:r_t} and the fact that $\inf\limits_{x \in [0, \pi/2]}\frac{\sin x}{x} > 0$)}\\
& \leq & \frac{O(1)\e^{-\frac{\pi v_x'}{4}}}{(1 - \e^{-\pi \Upsilon / 2})(1 - \e^{-\pi v_x'/4})}\,.
\end{eqnarray*}
This inequality gives us the bound $O(1) \e^{-\frac{\pi v_x'}{4}}$ whenever $v_x' \geq 0.1 N$ (say). 
Otherwise if any one of $v_x'$ or $\Upsilon$ is smaller than $0.1$, we get the bound trivially as $\e^{-\frac{\pi v_x'}{4}} = \Omega(1)$ in that case.
\end{proof}
When $v$ is very near one of the horizontal boundaries or the right boundary, the bound in Lemma~\ref{lem:kernel_bound1} can be considerably improved as shown by our next result.
\begin{lemma}
\label{lem:kernel_bound2}
Let $v = (v_x, v_y)$ be a point in $\inte(R_{\Upsilon N, N})$ and $(v_x', v_y') = (\tfrac{v_x}{N}, \tfrac{v_y}{N})$. 
Then
$$\sum_{z \in \partial_{\mathrm{left}} R} H_{R_{\Upsilon N, N}} (v, z)\leq O(v_y')O((\Upsilon - v_x')\wedge 1\Big)\frac{\e^{-\tfrac{\pi v_x'}{4}}}{(1 - \e^{-\pi v_x' / 4})^2(1 - \e^{-\Upsilon \pi / 2})}\,.$$
\end{lemma}
\begin{proof}
The proof is again a straightforward consequence of Proposition~\ref{prop:poisson_ker}. Here we have
\begin{eqnarray*}
&& N \sum_{z \in \partial_{\mathrm{left}}}H_{R_{\Upsilon N, N}}(v, z) = 2\sum_{j = 1}^{N - 1} 
\frac{\sinh(r(\tfrac{j\pi}{N})(\Upsilon N - v_x))}{\sinh (r(\tfrac{j\pi}{N})\Upsilon N))}\sin 
\Big(\frac{j\pi}{N}\Big)\sum_{k \in [N - 1]}\sin \Big(\frac{j\pi k}{N}\Big)\\
&\leq& 2\sum_{j = 1}^{N - 1} 
\frac{\sinh(r(\tfrac{j\pi}{N})(\Upsilon N - v_x))}{\sinh (r(\tfrac{j\pi}{N})\Upsilon N))}|\sin 
\Big(\frac{j\pi v_y}{N}\Big)|\frac{\cos \Big(\frac{j\pi}{2N}\Big)}{\sin \Big( \frac{j\pi}{2N}\Big)} \leq 4v_y\sum_{j = 1}^{N - 1} 
\frac{\sinh(r(\tfrac{j\pi}{N})(\Upsilon N - v_x))}{\sinh (r(\tfrac{j\pi}{N})\Upsilon N)}\cos \Big(\frac{j\pi}{2N}\Big)^2\,,
\end{eqnarray*}
where we used the fact that $|\sin(kt)| \leq k\sin t$ 
for $t \in [0, \pi]$ and $k \in \N$. Since
$$\sinh(r(\tfrac{j\pi}{N})(\Upsilon N - v_x)) \leq \e^{-\frac{j\pi v_x'}{4}}\big(2r(\tfrac{j\pi}{N})(\Upsilon N - v_x) \wedge 1 \big)\,,$$
and $r(t) = O(t)$ on $[0, \pi]$, we get
\begin{equation*}
N \sum_{z \in \partial_{\mathrm{left}}}H_{R_{\Upsilon N, N}}(v, z) \leq \frac{O(v_y)O(\frac{\Upsilon N - v_x}{N} \wedge 1)\e^{-\frac{\pi v_x'}{4}}}{(1 - \e^{-\pi \Upsilon / 2})(1 - \e^{-\pi v_x'/4})^2}\,. \qedhere
\end{equation*} 
\end{proof}

The following bound on the sum of Poisson kernels along a horizontal segment will be useful.
\begin{lemma}
\label{lem:kernel_bound3}
Let $y, y_1\in [N - 1]$ and $\delta' \in (0,1)$. Then
$$\sum_{\Upsilon N \delta' \leq x \leq \Upsilon N - 1} H_{R_{\Upsilon N, N}} \big((x, y), (0, y_1)\big)= \frac{O(1)\e^{-\tfrac{\Upsilon \pi \delta'}{4}}}{(1 - \e^{-\Upsilon \pi \delta' / 4})^2}\,.$$
\end{lemma}
\begin{proof}
Again from Proposition~\ref{prop:poisson_ker} and Lemma~\ref{lem:r_t} we get,
\begin{eqnarray*}
\sum_{\Upsilon N \delta' \leq x \leq \Upsilon N - 1}H_{R_{\Upsilon N, N}} \big((x, y), (0, 
y_1)\big) & \leq & O(N^{-1})\sum_{j = 1}^{N - 1} \sum_{\Upsilon N \delta' \leq x \leq \Upsilon N - 1}\frac{\e^{-\frac{j\pi x}{4N}}}{1 - \e^{-\frac{j\pi \Upsilon}{2}}}\\
&\leq& O(N^{-1})\sum_{j = 1}^{N - 1} \frac{\e^{-\frac{j\Upsilon \pi\delta'}{4}}}{(1 - \e^{-\frac{j \Upsilon \pi}{2}})(1 - \e^{-\frac{j\pi}{4N}})}\\
&\leq& O(N^{-1}) \frac{\e^{-\frac{\Upsilon \pi\delta'}{4}}}{(1 - \e^{-\Upsilon \pi/2})(1 - \e^{-\Upsilon \pi \delta' / 4})(1 -\e^{-\pi/4N})}\\
&\leq &\frac{O(1)\e^{-\tfrac{\Upsilon \pi \delta'}{4}}}{(1 - \e^{-\Upsilon \pi \delta' / 4})^2}\,. \quad\qquad\qquad\qquad\qquad\qquad\qquad\qquad\qquad\qedhere
\end{eqnarray*}
\end{proof}
Another useful lemma is the next simple relation between Poisson kernel and Green's function.
\begin{lemma} \cite[Lemma 6.3.6]{Lawler10}
\label{lem:reversibility}
Let $v \in \inte(R_{\Upsilon N, N})$ and $z \in \partial R_{\Upsilon N, N}$. Then
$$H_{R_{\Upsilon N, N}}(v, z) = \frac{1}{4}G_{R_{\Upsilon N, N}}(z_{R_{\Upsilon N, N}}, v)\,,$$
where $z_{R_{\Upsilon N, N}}$ is the unique neighbor of $z$ that lies in $\inte(R_{\Upsilon, N})$.
\end{lemma}

When $\Upsilon$ is large, a simple random walk starting from a ``typical'' point inside 
$R_{\Upsilon N, N}$ would most likely exit the rectangle before getting any close to the 
vertical boundaries. In our next result we use this simple intuition to show that the Green's 
function $G_{R_{\Upsilon N, N}}(v, v)$ at a typical point $v$ does not depend much on 
$\Upsilon$. We introduce a new notation in this connection. For a point $v$ inside $R_{\Upsilon N, 
N}$ and a positive integer $\Upsilon'$, let $R_{N, \Upsilon, \Upsilon'; v}$ denote the rectangle 
defined by the points $((v_x - \Upsilon' N)^+, 0)$ and $((v_x + \Upsilon' N) \wedge \Upsilon N, 
N)$.    
\begin{lemma}
\label{lem:greens_overshoot}
Let $\Upsilon, \Upsilon' \geq 1$. Then for any $v \in \inte(R_{\Upsilon N, N})$,
$$G_{R_{\Upsilon N, N}}(v, v) \leq G_{R_{N, \Upsilon, \Upsilon'; v}}(v, v) + O(1)\e^{-\frac{\Upsilon'\pi}{4}}\,.$$
\end{lemma}
\begin{proof}
Denote the numbers $0, (v_x - 2\Upsilon' N)^+$ and $(v_x - \Upsilon' N)^+$ by $x_{-3}, x_{-2}$ and 
$x_{-1}$; and the numbers $(v_x + \Upsilon' N)\wedge \Upsilon N, (v_x + 2\Upsilon' N)\wedge 
\Upsilon N$ and $\Upsilon N$ by $x_{0}, x_{1}$ and $x_{2}$ respectively. Let $\tau_1 < \tau_2 < 
\ldots$ denote the successive time points at which the simple random walk $\{S_t\}$ visits the 
lines $\Vl_{x_i}$'s. It is easy to see that
$$G_{R_{\Upsilon N, N}}(v, v) = G_{R_{N, \Upsilon, \Upsilon'; v}}(v, v) + \E \Big (\sum_{j \geq 1}\mathbf{1}_{\{\tau_{j}^* < \tau_{R_{\Upsilon N, N}}\}}\sum_{t = 
\tau_j^*}^{\tau_{j+1}^* - 1}\mathbf{1}_{\{S_t = v\}}\Big)\,,$$
where $\tau_j^* = \tau_{R_{\Upsilon N, N}} \wedge 
\tau_j$. In order to estimate the expectation of $\sum_{t = \tau_j^*}^{\tau_{j+1}^* - 1}\mathbf{1}_{\{S_t = v\}}$, we will use the following expression for $G_{R_{M, N}}(u, w)$ ($u, w \in \inte(R_{M, N})$) from \cite[Theorem 4.6.2]{Lawler10}:
\begin{equation}
\label{eq:Greens_expression}
G_{R_{M, N}}(u, w) = \sum_{z \in \partial R_{M, N}}H_{R_{M, N}}(u, z)a(z - w) - a(u - w)\,,
\end{equation}
where $a(x)$ is the potential kernel for two dimensional simple random walk. An approximation for $a(x)$ is given in \cite[Theorem 4.4.4]{Lawler10} as follows
\begin{equation}
\label{eq:poten_kernel}
a(x) = \frac{2}{\pi}\log|x| + \frac{2\overline{\gamma} + \log 8}{\pi} + O(|x|^{-2})\,,
\end{equation}
where $a(0) = 0$ and $\overline{\gamma}$ is the 
Euler-Mascheroni constant. The last two expressions and the choice of the numbers $x_i$'s ensures that $\E^{S_{\tau_j^*}} (\sum_{t = \tau_j^*}^{\tau_{j+1}^* - 1}\mathbf{1}_{\{S_t = v\}}) = O(1)$ uniformly for all 
values of $S_{\tau_j^*}$. Also from Lemma~\ref{lem:kernel_bound1} we have $\P^v (\tau_j^* < \tau_{R_{\Upsilon N, N}}) \leq 
O(1)\e^{-\frac{\Upsilon'\pi j}{4}}$. The lemma now follows from these facts and strong Markov property.
\end{proof}
Lemma~\ref{lem:greens_overshoot} implies the following upper bound on the Green's function for a rectangle.
\begin{lemma}
\label{lem:free_field_var}
For all $v \in \inte(R_{\Upsilon N, N})$, we have $G_{R_{\Upsilon N, N}}(v, v) \leq \tfrac{2}
{\pi}\log N + O(1)$.
\end{lemma}
\begin{proof}
From \eqref{eq:Greens_expression} and \eqref{eq:poten_kernel} we get that $G_{R_{N, N}}(v, v) 
\leq \tfrac{2}{\pi}\log N$ for all $v \in R_{N, N}$. The bound now follows from this observation and Lemma~\ref{lem:greens_overshoot}.
\end{proof}
We can similarly obtain an upper bound on the two points Green's function inside a rectangle.
\begin{lemma}
\label{lem:free_field_cov}
Let $u, v \in \inte(R_{\Upsilon N, N})$ such that $|u_x - v_x| \geq 0.1N$. 
Then we have
$$G_{R_{\Upsilon N, N}} (u, v) \leq O(1)\e^{-\tfrac{\pi|u_x - v_x|}{8N}}\,.$$
\end{lemma}
\begin{proof}
Assume that $u_x < v_x$. Let $R_{\Upsilon N, N; \mathrm{left}}$ and $R_{\Upsilon N, N; \mathrm{right}}$ be the two sub-rectangles of $R_{\Upsilon N, N}$ that are formed by the vertical line $\mathbb L$ passing through the middle of $u_x$ and $v_x$. Due to strong Markov property we have
$$G_{R_{\Upsilon N, N}} (u, v) = \sum_{z \in \mathbb L}H_{R_{\Upsilon N, N;\mathrm{left}}}(u, z)G_{R_{\Upsilon N, N}}(z, v)\,.$$
From an argument similar to the one used to prove Lemma~\ref{lem:greens_overshoot}, we can deduce
$$G(z, v) \leq O(1)\log\big(O(1) + \tfrac{O(1)N}{|u_x - v_x|}\big)\,,$$
for all $z \in \mathbb L$. Also since $|u_x - v_x| = \Omega(N)$ we get from Lemma~\ref{lem:kernel_bound1} that 
$$\sum_{z \in \mathbb L}H_{R_{\Upsilon N, N;\mathrm{left}}}(u, z) \leq O(1)\e^{-\tfrac{\pi|u_x - v_x|}{8N}}\,.$$
The last two displays together yield the lemma.
\end{proof}
We will conclude this subsection with some limit 
results. Let $\Upsilon > 1$ and $R_{2\Upsilon}$ denote the $\R^2$-rectangle $[0, 2\Upsilon] \times [0, 1]$. Now define a function $h_{R_{2\Upsilon}}(w, z): \inte(R_{2\Upsilon}) \times \partial R_{2\Upsilon} \to \R^+$ as:
\begin{eqnarray*}
h_{R_{2\Upsilon}}(w, (0, y)) &=& h_{R_{2\Upsilon}}((2\Upsilon - w_x, w_y), (\Upsilon, y))\\ &=& 2\sum_{j = 1}^{\infty} \frac{\sinh(j\pi(2\Upsilon - w_x))}{\sinh (2j\pi\Upsilon)}\sin(j\pi w_y)\sin(j\pi y)\,
\end{eqnarray*}
and,
\begin{eqnarray*}
h_{R_{2\Upsilon}}(w, (x, 0)) &=& h_{R_{2\Upsilon}}((w_x, 1 - w_y), (x, 1))\\
&=& \frac{1}{\Upsilon}\sum_{j = 1}^{\infty} \frac{\sinh(\tfrac{j\pi}{2\Upsilon}(1 - w_y))}{\sinh (\tfrac{j\pi}{2\Upsilon})}\sin \Big(j\pi \big(1 - \frac{w_x}{2\Upsilon}\big)\Big)\sin \Big(j\pi \big(1 - \frac{x}{2\Upsilon}\big)\Big)\,.
\end{eqnarray*}
Here $w_x$ and $w_y$ are the horizontal and vertical 
coordinates of $w$ respectively. A quick comparison with the expression in Proposition~\ref{prop:poisson_ker} suggests that $h_{R_{2\Upsilon}}(w, z) \approx N 
H_{R_{2\Upsilon N, N}}(Nw, Nz)$. Our next lemma gives a quantitative bound on the error of approximation.
\begin{lemma}
\label{lem:convergence}
Let $w, (x, y) \in \inte(R_{2\Upsilon}) \cap \tfrac{1}{N}\Z^2$. Then
\begin{equation*}
|N H_{R_{2\Upsilon N, N}}(Nw, (0, Ny)) - h_{R_{2\Upsilon}}(w, (0, y))| \leq \frac{O_{\Upsilon}(1)}{N^2w_x^6}\sin(\pi y)\sin(\pi w_y)\,,
\end{equation*}
and,
\begin{equation*}
|N H_{R_{2\Upsilon N, N}}(Nw, (Nx, 0)) - h_{R_{2\Upsilon}}(w, (x, 0))| \leq \frac{O_{\Upsilon}(1)}{N^2w_y^6}\sin \Big(j\pi \big(1 - \frac{w_x}{2\Upsilon}\big)\Big)\sin \Big(j\pi \big(1 - \frac{x}{2\Upsilon}\big)\Big)\,.
\end{equation*}
\end{lemma}
\begin{proof}
We will follow the approach adopted in the proof of 
\cite[Proposition~8.1.4]{Lawler10}. Let us first split $NH_{R_{2\Upsilon N, N}}(Nw , (0, Ny))$ into two parts:
\begin{equation*}
NH_{R_{2\Upsilon N, N}}(Nw , (0, Ny)) = 2\sum_{j = 1}^{N-1}\frac{\sinh (Nr(\tfrac{j\pi}{N})(2\Upsilon - w_x))}{\sinh(2N r(\tfrac{j\pi}{N})\Upsilon)}\sin (j\pi y)\sin (j \pi w_y) = 2\big(\sum_{1; H} + \sum_{2; H}\big)\,,
\end{equation*}
where $\sum_{1, H}$ and $\sum_{2, H}$ contain the terms corresponding to $k < N^{2/3}$ and $k \geq N^{2/3}$ 
respectively. In a similar way we can write,
\begin{equation*}
h_{R_{2\Upsilon}}(w , (0, y)) = 2\big(\sum_{1; h} + \sum_{2; h}\big)\,,
\end{equation*}
Using Lemma~\ref{lem:r_t}, and the fact that $|\sin(kt)| \leq k \sin t$ for $k \in \N$ and $t \in (0, \pi)$ we get
\begin{equation}
\label{eq:convergence1}
\big|\sum_{2;H}\big| \leq \sin(\pi y)\sin(\pi w_y)\sum_{j \geq N^{2/3}}j^2\e^{-\tfrac{j\pi w_x}{4}} \leq \frac{\sin(\pi y)\sin(\pi w_y)}{N^2}\sum_{j \geq N^{2/3}}j^5\e^{-\tfrac{j\pi w_x}{4}}\,.
\end{equation}
Similarly,
\begin{equation}
\label{eq:convergence2}
\big|\sum_{2;h}\big|  \leq \frac{\sin(\pi y)\sin(\pi w_y)}{N^2}\sum_{j \geq N^{2/3}}j^5\e^{-j\pi w_x}\,.
\end{equation}
When $k < N^{2/3}$, Lemma~\ref{lem:r_t} implies that for $0 \leq x \leq 2\Upsilon$,
\begin{equation*}
\sinh (Nr(\tfrac{j\pi}{N})x) = \sinh(j\pi x)\big(1 + O_\Gamma(\tfrac{j^3}{N^2})\big)\,.
\end{equation*} 
Thus
\begin{equation*}
\big|\sum_{1;H} - \sum_{1, h}\big| \leq O_{\Gamma}(1)\frac{\sin(\pi y)\sin(\pi w_y)}{N^2}\sum_{j < N^{2/3}}j^5\e^{-j\pi w_x}\,.
\end{equation*}
Together with \eqref{eq:convergence1} and \eqref{eq:convergence2} this gives us,
$$|N H_{R_{2\Upsilon N, N}}(Nw, (0, Ny)) - h_{R_{2\Upsilon}}(w, (0, y))| \leq \frac{O_{\Upsilon}(1)}{N^2w_x^6}\sin(\pi y)\sin(\pi w_y)\,.$$
The bound on $|N H_{R_{2\Upsilon N, N}}(Nw, (Nx, 0)) - h_{R_{2\Upsilon}}(w, (x, 0))|$ can be derived in a similar way.
\end{proof}
Using Lemma~\ref{lem:convergence} and properties of the function $h_{R_{2\Upsilon}}$, we can obtain an asymptotic expression for the average value of Green's function as follows:
\begin{lemma}
\label{lem:Brownian_scaling}
Let $\theta \in (0,1)$. Then for any subset $\mathcal I$ of $[\theta, 1 - \theta]$ that is a union of finitely many disjoint intervals,
$$\frac{1}{|\{\Upsilon N\}\times N \mathcal I \cap \Z^2|}\sum_{v \in \{\Upsilon N\}\times N\mathcal I \cap \Z^2}G_{R_{2\Upsilon N, N}}(v, v) = \frac{2}{\pi} \log N + C_{\mathcal I} + O_{\mathcal I, \theta, \Upsilon}(N^{-1})\,,$$
where
$C_{\mathcal I} = \tfrac{2}{\pi |I|}\int_{I}\int_{\partial R_{2\Upsilon}}\log |(\Upsilon, y) - w|h_{R_{2\Upsilon}}((\Upsilon, y), w) dw dy + \tfrac{2\overline{\gamma} + \log 8}{\pi}$.
\end{lemma}
\begin{proof}
Recall the expression for $G_{R_{2\Upsilon N, N}}(v, v)$ from \eqref{eq:Greens_expression}:
\begin{equation}
\label{eq:Brownian_scaling1}
G_{R_{2\Upsilon N, N}}(v, v) = \sum_{z \in \partial R_{2\Upsilon N, N}}H_{R_{2\Upsilon N, N}}(v, z)a(z - v)\,.
\end{equation}
Now notice that the functions $h_{R_{2\Upsilon}}(w, z)$ and $\log |w - z|$ are Lipschitz separately in each variable for $w \in \{\Upsilon\} \times [\theta, 1 - \theta]$ and $z$ in any one of the four boundary 
segments of $R_{2\Upsilon}$. The lemma now follows from this fact along with Lemma~\ref{lem:convergence}, \eqref{eq:Brownian_scaling1} and \eqref{eq:poten_kernel}.
\end{proof}
\subsection{Some properties of discrete Gaussian free field on a rectangle}
In the current subsection we will derive some properties 
of the discrete GFF defined on rectangles. For our first result we assume that $\Upsilon_1, \Upsilon_2 \geq 1$. 
Let $\{\chi_{v}\}_{v \in R_{\Upsilon_1 N, N}}$ be a GFF 
with Dirichlet boundary conditions. For $\theta \in (0, 1)$, call a point $v$ in the rectangle $R = ([a, b] \times [c, d]) \cap \Z^2$ as \emph{$\theta$-isolated} from $\partial R$ if $d(z_x, \{a, b\}) \geq \theta|a - 
b|$ and $d(z_y, \{c, d\}) \geq \theta|c - d|$. The set of all points that are $\theta$-isolated from $\partial R_{\Upsilon_1 N, N}$ is a rectangle, say, $\partial 
R_{\Upsilon_1 N, N, \theta}$. Let $R_{\Upsilon_2 K, K}^w$ be another rectangle contained in $R_{\Upsilon_1N, N, \theta}$ and define $R_{\Upsilon_2 K, K, \theta}^w$ 
similarly. Here $K$ (and hence $N$) is big enough so 
that $K\theta \geq 2$. Now define two additional fields on $R_{\Upsilon_2 K, K}^w$ as follows:
\begin{equation}
\label{eq:field_decomp}
\chi_v^c = \E(\chi_v | \chi_{\partial R_{\Upsilon_2 K, K}^w}),\mbox{\mbox{\hspace{0.3cm}}}\chi_v^f 
= \chi_v - \chi_v^c\,.
\end{equation}
Notice that $\chi_v^f$ is distributed as a GFF on $R_{\Upsilon_2 K, K}^w$ with Dirichlet boundary 
conditions, and that the Gaussian fields $\{\chi_v^f\}_{v \in R_{\Upsilon_2 K, K}^w}$ and 
$\{\chi_v^c\}_{v \in R_{\Upsilon_2 K, K}^w}$ are independent of each other.
\begin{lemma}
\label{lem:field_smoothness}
Let $u, v \in R_{\Upsilon_2 K, K, \theta}^w$ such that $||u - v||_\infty \leq (1 - 2\theta) K$.
$$\E(\chi_u^c - \chi_v^c)^2 \leq  O(1/\theta^3)\Big(\frac{|u - v|}{K}\Big)^2\,.$$
\end{lemma}
\begin{proof}
Independence of the fields $\{\chi_v^f\}$ and $\{\chi_v^c\}$ imply that for any $u$ and $v$ in 
$R_{\Upsilon_2 K, K}^w$,
\begin{equation}
\label{eq:field_smoothness1}
\E (\chi_u^c - \chi_v^c)^2 = \E (\chi_u - \chi_v)^2 - \E (\chi_u^f - \chi_v^f)^2\,.
\end{equation}
From \eqref{eq:Greens_expression} and a routine algebra we get
\begin{equation}
\label{eq:field_smoothness2}
\E (\chi_u - \chi_v)^2 = 2a(u - v) + \sum_{z \in \partial R_{\Upsilon_1 N, N}}
\Big(H_{R_{\Upsilon_1 N, N}}(u, z) - H_{R_{\Upsilon_1 N, N}}(v, z)\Big)\big(a(z - u) - a(z - v)\big)\,,
\end{equation}
and
\begin{equation}
\label{eq:field_smoothness3}
\E (\chi_u^f - \chi_v^f)^2 = 2a(u - v) + \sum_{z \in \partial R_{\Upsilon_2 K, 
K}^w}\Big(H_{R_{\Upsilon_2 K, K}^w}(u, z) - H_{R_{\Upsilon_2 K, K}^w}(v, z)\Big)\big(a(z - u) - a(z - v)\big)\,.
\end{equation}
Since $\Upsilon_1 \geq 1$ and $u, v \in R_{\Upsilon_1 N, N, \theta}$, we have from \eqref{eq:poten_kernel} 
that
\begin{equation}
\label{eq:field_smoothness5}
\mbox{$\max_{z \in \partial R_{\Upsilon_1 N, N}}$}|a(z - u) - a(z - v)| \leq \frac{4|u - v|}{\theta N} + \frac{O(1)}{\theta^2N^2}\,.
\end{equation}
Since $||u - v||_\infty \leq (1 - 2\theta)K$, we can define a square $R_{(1 - 2\theta)K}$ of side length $(1 - 2\theta)K$ within $R^w_{\Upsilon_2K, K, \theta}$ that 
contains both $u$ and $v$. Let $R_K$ and $R_N$ denote the squares of side length $K$ and $N$ respectively 
placed symmetrically around $R_{(1 - 2\theta)K}$. It is clear that $R_K \subseteq R^w_{\Upsilon_2K, K}$ and 
$R_N \subseteq R^w_{\Upsilon_1N, N}$. Now applying difference estimates for the harmonic function $H_{R_{\Upsilon_1 N, N}}(v, z)$ (in $v$) on $int(R_{\Upsilon_1 N, N})$ (see e.g. \cite[Theorem 6.3.8]{Lawler10}), we obtain
\begin{equation}
\label{eq:field_smoothness6}
|H_{R_{\Upsilon_1 N, N}}(u, z) - H_{R_{\Upsilon_1 N, N}}(v, z)| \leq c \frac{|u - v|}{\theta N} 
\max_{w \in R_{(1 - 2\theta)K}}H_{R_{\Upsilon_1 N, N}}(w, z) \,,
\end{equation}
where $c$ is an absolute constant. Observe that
\begin{equation}
\label{eq:field_smoothness7}
H_{R_{\Upsilon_1N, N}}(w, z) = \sum_{w' \in \partial R_N}H_{R_K}(w, w')H_{R_{\Upsilon_1N, N}}(w', z)\,,
\end{equation}
for all $w \in R_{(1 - 2\theta)K}$. As each point inside $R_{(1 - 2\theta)K}$ lies at least $\theta N$ away from $\partial R_N$, we get from Proposition~\ref{prop:poisson_ker} and Lemma~\ref{lem:r_t}:
\begin{equation}
\label{eq:field_smoothness8}
H_{R_N}(w, w') \leq O(N^{-1})\sum_{j \in [N-1]}\e^{-\frac{j\pi\theta}{4}} = \frac{O(N^{-1})}{1 - \e^{-\frac{\theta \pi}{4}}} = O((N\theta)^{-1})\,,
\end{equation}
for all $w \in R_{(1 - 2\theta)K}$ and $w' \in \partial 
R_N$. The last four displays together imply that
\begin{equation}
\label{eq:field_smoothness9}
|\sum_{z \in \partial R_{\Upsilon_1 N, N}}
\Big(H_{R_{\Upsilon_1 N, N}}(u, z) - H_{R_{\Upsilon_1 N, N}}(v, z)\Big)\big(a(z - u) - a(z - v)\big)| \leq O(\theta^{-3})\Big(\frac{|u - v|}{N}\Big)^2\,.
\end{equation}
In a similar way we get that
\begin{equation}
\label{eq:field_smoothness10}
|\sum_{z \in \partial R^w_{\Upsilon_2 K, K}}
\Big(H_{R^w_{\Upsilon_2 K, K}}(u, z) - H_{R^w_{\Upsilon_2 K, K}}(v, z)\Big)\big(a(z - u) - a(z - v)\big)| \leq O(\theta^{-3})\Big(\frac{|u - v|}{K}\Big)^2\,.
\end{equation}
The lemma now follows from plugging in the expressions \eqref{eq:field_smoothness2} and \eqref{eq:field_smoothness3} into \eqref{eq:field_smoothness1}, and using the last two bounds.
\end{proof}
In a similar vein we obtain the following continuity result for Green's functions.
\begin{lemma}
\label{lem:Green_function_continuity}
Let $\Upsilon \geq 1$ and $u, v \in R_{\Upsilon N, N, \theta}$ such that $||u - v||_\infty \leq (1 - 2\theta)N$. Then
$$|G_{R_{\Upsilon N, N}} (u, u) - G_{R_{\Upsilon N, N}} (v, v)| \leq  O(\log \Upsilon /\theta^2)\frac{|u - v|}{N}\,.$$
\end{lemma}
\begin{proof}
We begin with the expression of $G_{R_{\Upsilon N, N}}(u, u)$ from \eqref{eq:Greens_expression},
\begin{equation*}
G_{R_{\Upsilon N, N}} (u, u) = \sum_{z \in \partial R_{\Upsilon N, N}} H_{R_{\Upsilon N, N}}(u, z) a(z - u)\,.
\end{equation*}
Since $u, v \in R_{\Upsilon N, N, \theta}$, from \eqref{eq:poten_kernel} we get,
\begin{eqnarray*}
|G_{R_{\Upsilon N, N}} (u, u) - G_{R_{\Upsilon N, N}} (v, v)| &\leq& O(1)\sum_{z \in \partial R_{\Upsilon N, N}} |H_{R_{\Upsilon N, N}}(u, z) - H_{R_{\Upsilon N, N}}(v, z)| \log \frac{|z - v|}{\theta N} \\ 
&& + O(1)\sum_{z \in \partial R_{\Upsilon N, N}} H_{R_{\Upsilon N, N}}(u, z) \frac{|u - v|}{\theta N} + \frac{O(1)}{\theta^2 N^2}\\
&\leq& O(\log \Upsilon)\sum_{z \in \partial R_{\Upsilon N, N}} |H_{R_{\Upsilon N, N}}(u, z) - H_{R_{\Upsilon N, N}}(v, z)|  +  O(1)\frac{|u - v|}{\theta N}
\end{eqnarray*}
From a computation similar to the one that led to \eqref{eq:field_smoothness9} in the proof of previous lemma, we get
$$\sum_{z \in \partial R_{\Upsilon N, N}} |H_{R_{\Upsilon N, N}}(u, z) - H_{R_{\Upsilon N, N}}(v, z)| = O(1)\frac{|u - v|}{\theta^2 N}\,.$$
The last two displays together yield,
\begin{equation*}
|G_{R_{\Upsilon N, N}} (u, u) - G_{R_{\Upsilon N, N}} (v, v)| = O(\log \Upsilon)\frac{|u - v|}{\theta^2 N}\,. \qedhere
\end{equation*}
\end{proof}

In order to prove our main theorem we need smoothness results similar to the one given in Lemma~\ref{lem:field_smoothness} for certain other types 
of fields. To this end let $\{\chi_v\}_{v \in R_{\Upsilon N, N}}$ be a GFF with Dirichlet boundary 
conditions. For any (nonempty) subinterval $I$ of $[1, \Upsilon N - 1]$ and $\nu \in (0, 1) \cap N^{-1}\Z$, define a random variable $Z_{I, \nu}^\Sigma$ as
$$Z_{I, \nu}^\Sigma = \sum_{v \in I \times \{\nu N\}} \chi_v\,.$$
The following lemma gives bounds on variances and covariances of these variables.
\begin{lemma}
\label{lem:general_covar}
Suppose that $|I| \geq 2000N$. Then for all $\nu \in (0, 1)\cap N^{-1}\Z$ we have,
$$4\Big(|I| - 201N\log \frac{|I|}{N}\Big)\nu(1 - 
\nu)N \leq \var Z_{I, \nu}^\Sigma \leq 4|I|\nu(1 - \nu)N\,.$$
Also for any two disjoint subintervals $I_1$ and $I_2$ of $[1, \Upsilon N - 1]$ and $\nu_1, \nu_2\in (0, 1)\cap N^{-1}\Z$ we have
$$0 \leq \cov(Z_{I_1, \nu_1}^\Sigma, Z_{I_2, \nu_2}^\Sigma) \leq O(1)(|I_2|\wedge N)\nu_2(1 - \nu_2)\e^{-\frac{\pi d(I_1, I_2)}{4N}}N\,.$$
\end{lemma}
\begin{proof}
For any $A  \subseteq R_{\Upsilon N, N}$, define $G_{R_{\Upsilon N, N}}(v, A)$ to be $\sum_{w \in 
A}G_{R_{\Upsilon N, N}}(v, w)$. Notice that
\begin{equation}
\label{eq:general_covar1}
\var Z_{I, \nu}^\Sigma = \sum_{v \in I \times \{\nu N\}}G_{R_{\Upsilon N, N}}(v, I \times \{\nu N\})\,.
\end{equation}
From definition of Green's function (see the first paragraph in Subsection~\ref{subsec:random_walk}) we then get
\begin{equation}
\label{eq:general_covar2}
G_{R_{\Upsilon N, N}}(v, I \times \{\nu N\}) \leq \E^v\big(\sum_{0 \leq t \leq \tau_{0, N}}\mathbf 1_{\{S_{t; y} = \nu N\}}\big)\,,
\end{equation}
where $S_{t; y}$ is the vertical coordinate of $S_t$ and $\tau_{0, N}$ is the first time $S_t$ hits the lines 
$\Hl_0$ or $\Hl_{N}$. But the law of $\{S_{t; y}\}$ is that of a one-dimensional lazy random walk starting 
from $v_y$. More precisely $\{S_{t; y}\}$ is a $\Z$-valued Markov chain starting from $v_y$ with transition probabilities $\{p_{a, b}\}_{a, b \in \Z}$ given by
$$p_{a, b} = \begin{cases}
1/2 &\mbox{ if } a = b\,, \\
1/4 &\mbox{ if } |a - b| = 1\,, \\
0 &\mbox{ otherwise}\,.
\end{cases}$$
Likewise with simple random walk in two dimension, the lazy random walk Green's function $G^{1, \star}_{[a, b] \cap \Z}(x_1, x_2)$ is defined as
$$G^{1, \star}_{[a, b] \cap \Z}(y_1, y_2) = \E^{1, \star, y_1}\big(\sum_{0 \leq t \leq \tau_{a, b}}\mathbf 1_{\{S_{t; y} = y_2\}}\big)\,,$$
where $\E^{1, \star, y_1}$ is with respect to the law of $S_{t; y}$ starting from $y_1 \in \Z$ and $\tau_{a, b}$ 
is the first time $S_{t; y}$ hits $\Hl_a$ or $\Hl_b$. From a straightforward calculation involving effective resistance it follows that
$$G^{1, \star}_{[a, b]\cap \Z}(y, y) = 4\frac{(b -y)(y - a)}{b - a}\,,$$
whenever $a < b$. Hence from \eqref{eq:general_covar2} we get
\begin{equation*}
G_{R_{\Upsilon N, N}}(v, I \times \{\nu N\}) \leq G^{1, \star}_{[0, N]\cap \Z}(\nu N, \nu N) = 4\nu (1 - \nu) N\,,
\end{equation*}
for all $v \in I \times \{\nu N\}$. Plugging this into \eqref{eq:general_covar2} gives the upper bound on 
variance. For the lower bound we will show that the approximation of $G_{R_{\Upsilon N, N}}(v, I \times \{\nu N\})$ with $G^{1, \star}_{[0, N]\cap \Z}(\nu N, \nu N)$ is good except when $v$ lies very close to the 
endpoints of $I\times \{\nu N\}$. To this end denote by $I_{\mathrm{center}}$ the subset of $I$ consisting of points that are at least $100 N\log \tfrac{|I|}{N}$ 
away from the endpoints of $I$. Also denote by $R_I$ the 
rectangle $I \times ([0, N] \cap \Z)$. Then for any $v \in I_{\mathrm{center}} \times \{\nu N\}$, we have
\begin{equation*}
G^{1, \star}_{[0, N]\cap \Z}(\nu N, \nu N) - G_{R_{\Upsilon N, N}}(v, I \times \{\nu N\}) \leq \sum_{z \in \partial_{\mathrm{left}}R_I \cup \partial_{\mathrm{right}}R_I}H_{R_I}(v, z)G^{1, \star}_{[0, N]\cap \Z}(z_y, \nu N)\,.
\end{equation*}
But $G^{1, \star}_{[0, N]\cap \Z}(z_y, \nu N) \leq G^{1, \star}_{[0, N]\cap \Z}(\nu N, \nu N) = 4\nu (1 - \nu) 
N$. Thus
\begin{equation*}
G^{1, \star}_{[0, N]\cap \Z}(\nu N, \nu N) - G_{R_{\Upsilon N, N}}(v, I \times \{\nu N\}) \leq 4\nu (1 - \nu) N\sum_{z \in \partial_{\mathrm{left}}R_I \cup \partial_{\mathrm{right}}R_I}H_{R_I}(v, z)\,.
\end{equation*}
From Lemma~\ref{lem:kernel_bound1} we now get
\begin{align}
\label{eq:general_covar3}
G^{1, \star}_{[0, N]\cap \Z}(\nu N, \nu N) - G_{R_{\Upsilon N, N}}(v, I \times \{\nu N\}) &\leq 4\nu (1 - \nu) NO(\e^{-25\pi \log (\frac{|I|}{N})}) \nonumber \\
& = \nu (1 - \nu) N O((|I|/N)^{-70})\,,
\end{align}
for all $v \in I_{\mathrm{center}} \times \{\nu N\}$. 
Plugging this bound into \eqref{eq:general_covar1} and using the non-negativity of Green's functions we deduce
\begin{equation*}
\var Z_{I, \nu}^\Sigma \geq 4\Big(|I| - 201N\log \frac{|I|}{N}\Big)\nu(1 - \nu)N\,.
\end{equation*}
The lower bound on covariance is trivial. For the upper bound let us assume, without loss of generality, that 
$I_1$ lies to the left of $I_2$. 
Denote the interval $[0, r_{I_1}] \cap \Z$ as $I_1^\star$ and the rectangle $I_1^\star \times ([0, N] \cap \Z)$ as $R_{I_1^\star}$ (see Figure~\ref{fig:general_covar}).
\begin{figure}[!htb]
\centering
\begin{tikzpicture}[semithick, scale = 3]
\draw (-1.7, 0.3) -- (-0.2, 0.3);
\node [scale = 0.65, below] at (-0.95, 0.3) {$I_1 \times \{\nu_1 N\}$};


\draw (0.2, 0) -- (1.7, 0);

\node [scale = 0.65, below] at (0.85, 0) {$I_2 \times \{\nu_2 N\}$};

\draw [dashed] (-2.2, -0.2) rectangle (-0.2, 0.7);
\node [scale = 0.65, below] at (-1.2, -0.2) {$I_1^\star \times \{0\}$};

\end{tikzpicture}
\caption{{\bf The interval $I_1^\star$ and the rectangle $R_{I_1^\star}$.} The rectangle with broken boundary lines is $R_{I_1^\star}$.}
\label{fig:general_covar}
\end{figure}
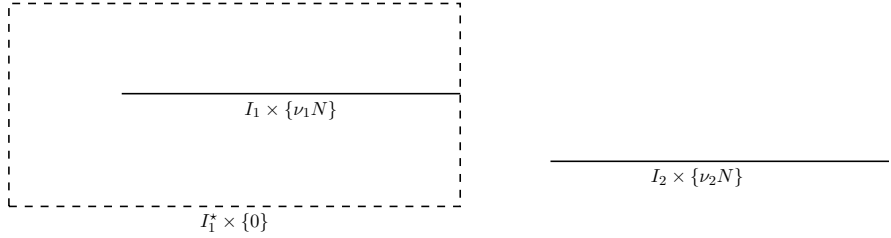
It follows from our previous discussion that for all $v \in R_{\Upsilon N, N}$,
$$G_{R_{\Upsilon N, N}}(v, I_2 \times \{\nu_2 N \}) \leq 4\nu_2(1 - \nu_2)N\,.$$
Now if $v \in I_1 \times \{\nu_1 N\}$, using Markov property we get
\begin{equation*}
G_{R_{\Upsilon N, N}}(v, I_2 \times \{\nu_2 N \}) \leq  \sum_{z \in \partial_{\mathrm{right}}R_{I_1^\star}}H_{R_{I_1}^\star}(v, z)G_{R_{\Upsilon N, N}}(z, I_2 \times \{\nu_2 N \})\,.
\end{equation*}
Hence from Lemma~\ref{lem:kernel_bound1} it follows that
\begin{equation}
\label{eq:general_covar4}
G_{R_{\Upsilon N, N}}(v, I_2 \times \{\nu_2 N \}) \leq O(1)\nu_2(1 - \nu_2)N\e^{-\frac{\pi d(v_x, I_2)}{4N}}\,.
\end{equation}
Consequently 
\begin{equation}
\label{eq:general_covar5}
\cov(Z_{I_1, \nu_1}^\Sigma, Z_{I_2, \nu_2}^\Sigma) = \sum_{v \in I_1 \times \{\nu_1 N\}}G_{R_{\Upsilon N, N}}(v, I_2 \times \{\nu_2 N \}) \leq \sum_{v \in I_1 \times \{\nu_1 N\}}\e^{-\frac{\pi d(v_x, I_2)}{4N}}O(\nu_2(1 - \nu_2))N\,.
\end{equation}
Summing this geometric series and using the fact that $1 - \e^{-x} = x (1 + o_{x \to 0}(1))$, we get the desired bound.
\end{proof}

Our next result shows that the field conditioned on any $Z_{I, \nu}^\Sigma$ is rather smooth.
\begin{lemma}
\label{lem:smoothness_conditional}
Let $\nu \in (0.4, 0.6) \cap N^{-1}\Z$ and $I$ be a sub-interval of $[0.1 N, \Upsilon N - 0.1 N]$. Also 
suppose that $|I| \geq 2000 N$. 
Then for all pairs of points $u, v \in \inte(R_{\Upsilon N, N})$ that lie at least $0.1 N$ away from $\partial_{\mathrm{left}}R_{\Upsilon N, N} \cup \partial_{\mathrm{right}}R_{\Upsilon N, N}$, we have
\begin{equation}
\label{eq:smoothness_conditional0}
\var \big(\E(\chi_u - \chi_v | Z_{I, \nu}^\Sigma)\big) = O(1)\Big(\frac{|u - v|}{N}\Big)^2 \vee \Big(\frac{|u - v|}{N}\Big)\,.
\end{equation}
In addition, if $u, v$ lie on the same side of the rectangle $I \times ([0, N] \cap \Z)$ such that distance of $u$ and $v$ from $I \times [N - 1]$ is at least $\Upsilon'N \geq 0.1 N$, then 
\begin{equation}
\label{eq:smoothness_conditional1}
\var \big(\E(\chi_u - \chi_v | Z_{I, \nu}^\Sigma)\big) = O(1)\e^{-\frac{\pi\Upsilon'}{2}}\Big(\frac{|u - v|}{N}\Big)^2 \vee \Big(\frac{|u - v|}{N}\Big)\,.
\end{equation}
\end{lemma}
\begin{proof}
Since
\begin{equation}
\label{eq:smoothness_condition1}
\E(\chi_u - \chi_v | Z_{I, \nu}^\Sigma) = \frac{G_{R_{\Upsilon N, N}}(u, I \times \{\nu N\}) - G_{R_{\Upsilon N, N}}(v, I \times \{\nu N\})}{\var Z_{I, \nu}^\Sigma}Z_{I, \nu}^\Sigma\,
\end{equation}
and we already have a good bound on $\var Z_{I, \nu}^\Sigma$ from Lemma~\ref{lem:general_covar}, all we need is to bound the difference $G_{R_{\Upsilon N, 
N}}(u, I \times \{\nu N\}) - G_{R_{\Upsilon N, N}}(v, I 
\times \{\nu N\})$. Let us begin with \eqref{eq:smoothness_conditional0} which we address by 
splitting into several cases. To this end first assume $u_y, v_y \in [0.1 N, 0.9 N]$ so that $d(u, \partial 
R_{\Upsilon N, N}) \wedge d(v, \partial R_{\Upsilon N, N}) \geq 0.1N$. From \cite[Theorem 4.6.2]
{Lawler10} we get the following expression for $G_{R_{\Upsilon N, N}}(u, I \times \{\nu N\}) - G_{R_{\Upsilon N, N}}(v, I \times \{\nu N\})$:
\begin{equation}
\label{eq:smoothness_condition2}
\sum_{w \in I \times \{\nu N\}}\sum_{z \in \partial R_{\Upsilon N, N}}H_{R_{\Upsilon N, N}}(w, 
z)\big(a(z - u) - a(z - v)\big) - \sum_{w \in I \times \{\nu N\}}\big(a(u - w) - a(v - w)\big)\,.
\end{equation}
Notice that we can rewrite the first summation in \eqref{eq:smoothness_condition2} as
$$\sum_{z \in \partial R_{\Upsilon N, N}}\big(a(z - u) - a(z - v)\big)\sum_{w \in I \times \{\nu 
N\}}H_{R_{\Upsilon N, N}}(w, z)\,.$$
By Lemma~\ref{lem:reversibility} this equals 
$$\frac{1}{4}\sum_{z \in \partial R_{\Upsilon N, N}}\big(a(z - u) - a(z - v)\big)G_{R_{\Upsilon N, N}}(z_{R_{\Upsilon N, N}}, I \times \{\nu N\})\,.$$
Since the distance of $u, v$ from $\partial R_{\Upsilon N, N}$ is at least $0.1N$, it follows from 
\eqref{eq:poten_kernel} that,
$$|a(z - u) - a(z - v)| \leq O(1) \frac{|u - v|}{|z - v| \wedge |z - u|}\,,$$
for all $z \in \partial R_{\Upsilon N, N}$. Also, we can bound $G_{R_{\Upsilon N, N}}(z_{R_{\Upsilon N, N}}, I \times \{\nu N\})$ as (similar to the proof in Lemma~\ref{lem:general_covar})
$$G_{R_{\Upsilon N, N}}(z_{R_{\Upsilon N, N}}, I \times \{\nu N\}) \leq \P^{z_{R_{\Upsilon N, N}}}(S_t\mbox{ hits } I \times \{\nu N\}\mbox{ before }\tau_{R_{\Upsilon N, N}}) G^{1, \star}_{[0, N] \cap \Z}(\nu N, \nu N)\,.$$
Now by Lemma~\ref{lem:kernel_bound2} we have,
$$\P^{z_{R_{\Upsilon N, N}}}(S_t\mbox{ hits } I \times \{\nu N\}\mbox{ before }\tau_{R_{\Upsilon N, N}}) \leq  \begin{cases}
O(1/N) \mbox{\hspace{0.3cm}if } z_{R_{\Upsilon N, N};x} \in [p_I - 10N, r_I + 10N]\,, \\
\e^{-\tfrac{\pi d(z_{R_{\Upsilon N, N};x}, I)}{4}}O(N^{-1}) \mbox{\hspace{0.3cm}otherwise}\,,
\end{cases}$$
where $z_{R_{\Upsilon N, N};x}$ is the horizontal coordinate of $z_{R_{\Upsilon N, N}}$. Since $G^{1, \star}_{[0, N] \cap \Z}(\nu N, \nu N) = O(N)$ and $|I| \geq 2000N$, the last two displays give us 
\begin{equation}
\label{eq:smoothness_condition3*}
|\sum_{w \in I \times \{\nu N\}}\sum_{z \in \partial R_{\Upsilon N, N}, z_x \in I}H_{R_{\Upsilon N, N}}(w, z)\big(a(z - u) - a(z - v)\big)| = O(|u - v|)\log (|I|N^{-1})\,,
\end{equation}
and
\begin{equation}
\label{eq:smoothness_condition3**}
|\sum_{w \in I \times \{\nu N\}}\sum_{z \in \partial R_{\Upsilon N, N}, z_x \notin I}H_{R_{\Upsilon N, N}}(w, z)\big(a(z - u) - a(z - v)\big)| = O(|u - v|)\,.
\end{equation}
Thus,
\begin{equation}
\label{eq:smoothness_condition3}
|\sum_{w \in I \times \{\nu N\}}\sum_{z \in \partial R_{\Upsilon N, N}}H_{R_{\Upsilon N, N}}(w, z)\big(a(z - u) - a(z - v)\big)| = O(|u - v|)\log (|I|N^{-1})\,.
\end{equation}
We need to be more careful for bounding the sum $\sum_{w \in I \times \{\nu N\}}\big(a(u - w) - a(v - w)\big)$ as $u, v$ may potentially lie very close to 
$w$. To this end notice that 
\begin{equation}
\label{eq:smoothness_condition4}
|a(w - u) - a(w - v)| \leq \log \big(1 + \tfrac{|u - v|}{|w - v| \wedge |w - u|}\big) + (|w - u| 
\wedge |w - v|)^{-2}\,
\end{equation}
for all $w \neq u, v$. Since $w$'s lie along a segment, it follows from a straightforward calculation that
\begin{equation}
\label{eq:smoothness_condition5}
\sum_{w \in I \times \{\nu N\}}(|w - u| \wedge |w - v|)^{-2} = O(1) = O(|u - v|)\,,
\end{equation}
where we assume $u \neq v$ as that is the only 
interesting case. Next we split the range of summation and write
$$\sum_{w \in I \times \{\nu N\} \setminus \{u\}} \log \big(1 + \tfrac{|u - v|}{|w - u|}\big) = \Sigma_1 + \Sigma_2\,.$$
Here $\Sigma_1, \Sigma_2$ contain the terms corresponding to $|w_x - v_x| \leq |u - v|$ and $> |u 
- v|$ respectively. We can further partition the range of $\Sigma_1$ as $\cup_{n \geq 0} D_n$ where $D_n$ consists of points in $I \times \{\nu N\}$ satisfying 
$2^{-n - 1}|u - v| < |w_x - v_x| \leq 2^{-n}|u - v|$. It 
is clear that $|D_n| \leq 2^{-n-1}|u - v|$ and $\log \big(1 + \tfrac{|u - v|}{|w - u|}\big)$ is $O(n)$ for $w 
\in D_n$. Thus $\Sigma_1 = O(|u - v|)$. Further, note that 
$$\Sigma_2 \leq |u - v|\sum_{w \in I \times \{\nu N\} \setminus \{u\}: |w_x - v_x| > |u - v|}\frac{1}{|w_x - v_x|} \leq O(|u - v|) \Big(1 + \log \big (1 + \frac{|I|}{|u - v|}\big)\Big)\,,$$
where the last inequality follows from the observation 
that $|w_x - v_x|$ can be at most $|I| + |u - v|$. Putting all these together we get
\begin{equation}
\label{eq:smoothness_condition6}
\sum_{w \in I \times \{\nu N\} \setminus \{u\}} \log \big(1 + \tfrac{|u - v|}{|w - u|}\big) = O(|u - v|) \Big(1 + \log \big (1 + \frac{|I|}{|u - v|}\big)\Big)\,.
\end{equation}
The same bound holds if we interchange $u$ and $v$ in the preceding inequality. Finally when $w = u$ or $v$, we have
\begin{equation}
\label{eq:smoothness_condition7}
|a(w - u) - a(w - v)| = a(u - v) = O(\log |u - v|) + O(1) = O(|u - v|)\,.
\end{equation}
Therefore, we get from \eqref{eq:smoothness_condition5}, \eqref{eq:smoothness_condition6} and \eqref{eq:smoothness_condition7} that
\begin{equation}
\label{eq:smoothness_conditional5}
\sum_{w \in I \times \{\nu N\}}\big(a(u - w) - a(v - w)\big) = O(|u - v|) \Big(1 + \log \big (1 + \frac{|I|}{|u - v|}\big)\Big)\,.
\end{equation}
Combined with \eqref{eq:smoothness_condition3} this yields that
\begin{equation}
\label{eq:smoothness_conditional6}
\mid \sum_{w \in I \times \{\nu N\}} \big(G_{R_{\Upsilon N, N}}(u, w) - G_{R_{\Upsilon N, N}}(v, 
w)\big)\mid = O(|u - v|) \Big(1 + \log \big (1 + \frac{|I|}{|u - v|}\big)\Big)\,.
\end{equation}
Using \eqref{eq:smoothness_conditional6} and Lemma~\ref{lem:general_covar} for the corresponding terms in \eqref{eq:smoothness_condition1} we get that
\begin{equation}
\label{eq:smoothness_conditional7}
\var(\big(\E(\chi_u - \chi_v | Z_{I, \nu}^\Sigma)\big)) \leq O\big(\frac{\log(|I|N^{-1})^2}{|I|N^{-1}}\big)\frac{|u - v|^2}{N^2} + O(1)\frac{|u - v|}{N}\frac{|u - v|}{|I|}\Big(\log \big (1 + \frac{|I|}{|u - v|}\big)\Big)^2\,.
\end{equation}
The second term is bounded by $O(\tfrac{|u - v|}{N})$ as 
$\sup_{x \geq 0}\tfrac{(\log(1+x))^2}{x} < \infty$, which gives that
$$\var \big(\E(\chi_u - \chi_v | Z_{I, \nu}^\Sigma)\big) = O(1)\Big ( \frac{|u - v|}{N}\Big)^2 \wedge \Big ( \frac{|u - v|}{N}\Big)\,.$$
We will use a different way to bound when $u_y \wedge v_y \geq 0.9N$ or $u_y \vee v_y \leq 0.1N$. For these two cases we will argue that $|G_{R_{\Upsilon N, N}}(u, I \times \{\nu N\}) - G_{R_{\Upsilon N, N}}(u, I \times 
\{\nu N\})| = O(|u - v|)$. It suffices to prove the 
statement when $u, v$ are adjacent and $u_y \wedge v_y 
\leq 0.1N$. To this end we observe that $G_{R_{\Upsilon N, N}}(u, I \times \{\nu N\})$ is a harmonic function in $u$ on $\inte(R_{\Upsilon N, N}) \setminus I \times 
\{\nu N\}$. Also,
$$G_{R_{\Upsilon N, N}}(u, I \times \{\nu N\}) \leq  \P^{z_{R_{\Upsilon N, N}}}(S_t\mbox{ hits } I \times \{\nu N\}\mbox{ before }\tau_{R_{\Upsilon N, N}}) G^{1, \star}_{[0, N] \cap \Z}(\nu N, \nu N) = O(u_y)\,,$$
Now applying difference estimates (see \cite[Theorem 6.3.8]{Lawler10}) to this function we get,
$$|G_{R_{\Upsilon N, N}}(u, I \times \{\nu N\}) - G_{R_{\Upsilon N, N}}(u, I \times 
\{\nu N\})| = O(|u - v|/u_y)O(u_y) = O(|u-v|)\,.$$
Hence for all such pairs $u, v$ we have,
\begin{equation}
\label{eq:smoothness_conditional8}
\var\big(\E(\chi_u - \chi_v | Z_{I, \nu}^\Sigma)\big) \leq O(1) \frac{|u-v|^2}{\var Z_i}\leq O(1)\Big(\frac{|u - v|}{N}\Big)^2\,.
\end{equation}
Now suppose that exactly one of $u_y$ and $v_y$ lies in 
$[0.1 N, 0.9 N]$, say $u_y$. Then each of the 3 pairs of points $(u, (u_x, v_y^*)), ((u_x, v_y^*), (u_x, v_y))$ and $((u_x, v_y), v)$ can be covered by one of the two 
cases we considered. Here $v_y^* = \lfloor 0.1 N \rfloor$ or $\lceil 0.9 N \rceil$ accordingly as $v_y < 
0.1 N$ or $> 0.9 N$ respectively. Similar argument holds if $u_y \geq 0.9N$ and $v_y \leq 0.1N$ (or vice versa). This completes the proof of \eqref{eq:smoothness_conditional0}. 

Proof of \eqref{eq:smoothness_conditional1} is similar to \eqref{eq:smoothness_conditional8}, where one can use Lemma~\ref{lem:kernel_bound2} to bound the probability $\P^{z_{R_{\Upsilon N, N}}}(S_t\mbox{ hits } I \times \{\nu N\}\mbox{ before }\tau_{R_{\Upsilon N, N}})$.
\end{proof}
\begin{remark}
\label{remark:symmetrize}
Since $\var \tfrac{Z_{I, \nu_1}^\Sigma + Z_{I, \nu_2}^\Sigma}{2} = \Theta(1)\var Z_{I, \nu_1}^\Sigma = \Theta(N)$, the proof of Lemma~\ref{lem:smoothness_conditional} would work even if we replaced $Z_{I, \nu}^\Sigma$ with $\tfrac{Z_{I, \nu_1}^\Sigma + Z_{I, \nu_2}^\Sigma}{2}$ for some 
$\nu_1, \nu_2 \in (0.4, 0.6)$. This is the version that we will use for proving Theorem~\ref{theo:main}.
\end{remark}
\subsection{Conditional expectation with respect to weakly correlated Gaussians}
Lemma~\ref{lem:smoothness_conditional} implies that $\var(\E(\chi_u - \chi_v | Z_{I, \nu}^\Sigma))$ decays geometrically with the (horizontal) distance of $u, v$ 
from the segment $I \times \{\nu N\}$. In addition, Lemma~\ref{lem:general_covar} tells us that $\cov(Z_{I_1, \nu}^\Sigma, Z_{I_2, \nu}^\Sigma)$ decays 
geometrically with $d(I_1, I_2)$. Now consider a sequence of disjoint intervals $I_1, I_2, \ldots, I_n$ satisfying the conditions of Lemma~\ref{lem:smoothness_conditional} and suppose that 
$u_x, v_x$ lie in one of these intervals say $I_1$. Then it is expected that $\var(\E(\chi_u - \chi_v | \{Z_{I_j, \nu}^\Sigma\}_{j \in [n]}))$ should not differ much from 
$\var(\E(\chi_u - \chi_v | Z_{I_1, \nu}^\Sigma))$. We 
validate this intuition in the next lemma. Since conditional expectations are orthogonal projections in Gaussian linear space, we phrase our result in terms of vectors in a general Hilbert space.

\begin{lemma}
\label{lem:Gram_Schimidt}
Let $x_1, x_2, \ldots, x_n$ be $n$ vectors of unit norm in a Hilbert space $\mathcal H$ such that $|(x_i, x_{i'})| \leq A_1\rho^{|i - i'|}$ for some $0 < \rho < 
0.25$ and $A_1 < \tfrac{0.1}{\rho}$. Here $(.,.)$ is the inner product in $\mathcal H$. Denote by $\hat{x_i}$ the orthogonal projection of $x_i$ onto the space spanned by the vectors $x_1, \ldots, x_{i - 1}$ and by $\epsilon_i$ 
the residue $x_i - \hat{x_i}$. Now suppose $y$ is a vector such that $|(y, x_i)| \leq A_2\rho^{(n - i - 1)}$ for all $i \in [n-2]$ and $|(y, x_i)|\leq A_3$ for $i \in \{n - 1, 
n\}$. Then $|(y, \epsilon_i)| \leq 3A_2\rho^{n - i - 1}$ for all $i \in [n-2]$ and $|(y, \epsilon_i)| \leq 2A_3 + 8A_1A_2\rho^2$ 
for $i \in \{n-1, n\}$. Furthermore for all $i, i' \in [n]$ we have $1 \geq |\epsilon_i|^2 \geq 1 - \frac{4A_1^2\rho^2}{1 - \rho^2}$ and $|(\hat x_i, \hat x_{i'})| \leq \tfrac{4A_1^2\rho^{|i - i'| + 2}}{1 - \rho^2}$.
\end{lemma}
\begin{proof}
First we apply a standard Gram-Schimidt orhtogonalization procedure to the vectors $x_1, x_2,$ 
$\ldots, x_n$ to obtain the following series of identities:
\begin{equation}\label{eq:Gram_Schimidt1}
x_k = \sum_{i=1}^{k-1} a_{k, i} \epsilon_i + \epsilon_k\,, \mbox{ for } 1\leq k\leq n\,.
\end{equation}
We claim that $|a_{i, i'}| \leq 2A_1\rho^{i - i'}$ for 
all $n \geq i > i' \geq 1$. We will prove this claim by 
induction. Order the pairs $(i, i')$ lexicographically. 
The statement is vacuously true for $i = 1$. Now assume that $|a_{i, i'}| \leq 2A_1\rho^{i - i'}$ for all $(i, 
i') \leq (k, k' - 1)$ where $k' \in [k-1]$ and we 
interpret the pair $(k, 0)$ as $(k-1, k-2)$. Then from the display \eqref{eq:Gram_Schimidt1} we have 
\begin{equation}
\label{eq:Gram_Schimidt2}
|\epsilon_i|^2 \geq 1 - \frac{4A_1^2\rho^2}{1 - \rho^2}\,,
\end{equation}
whenever $i < k$. On the other hand from the expansion of $(x_k, x_{k'})$ yields the following:
\begin{equation}
\label{eq:Gram_Schimidt3}
|a_{k, k'}||\epsilon_{k'}|^2 \leq |(x_k, x_{k'})| + 4A_1^2\sum_{1 \leq i' \leq k' - 1} \rho^{k + k' - 2i'}\,.
\end{equation}
Plugging \eqref{eq:Gram_Schimidt2} and the upper bound of $|(x_k, x_{k'})|$ into \eqref{eq:Gram_Schimidt3} we get
\begin{equation*}
a_{k, k'} \leq \rho^{k - k'}\frac{A_1 + 4\frac{A_1^2\rho^2}{1 - \rho^2}}{1 - 4\frac{A_1^2\rho^2}{1 - \rho^2}} \leq 2A_1\rho^{k - k'}\,,
\end{equation*}
where the last inequality follows from the restrictions on $\rho$ and $A_1$. Thus \eqref{eq:Gram_Schimidt2} 
holds for all $i \in [n]$. As to $(\hat x_i, \hat x_{i'})$, notice that this is equal to $\sum_{k \in [i' - 1]}a_{i, k}a_{i', k}|\epsilon_k|^2$ when $i' \leq i$. 
Therefore,
$$|(\hat x_i, \hat x_{i'})| \leq \sum_{k \in [i' - 1]}|a_{i, k}a_{i', k}| \leq 4A_1^2\rho^{i - i'}\sum_{k \geq 1}\rho^{2k} \leq \frac{4A_1^2\rho^{|i - i'| + 2}}{1 - \rho^2} \,.$$

Our argument for the bounds on $|(y, \epsilon_i)|$'s is also inductive and uses the bounds on $|a_{i, i'}|$'s 
that we have already proved. Let us first write $y$ as
\begin{equation}
\label{eq:Gram_Schimidt4}
y = b_1\epsilon_1 + b_2 \epsilon_2 + \ldots + b_n \epsilon_n + y_{\mathrm{res}}\,,
\end{equation}
where $y_{\mathrm{res}}$ is orthogonal to 
$\{\epsilon_i\}_{i \in [n]}$. Then $(y, \epsilon_i) = b_i |\epsilon_i|^2$ and $|(y, \epsilon_1)| = |(y, x_1)| 
\leq  A_2\rho^{(|\ell - 1| - 1)^+}$. Now assume that $|(y, \epsilon_i)| \leq 3 A_2 \rho^{n - i - 1}$ for 
all $i < k \leq n - 2$. From \eqref{eq:Gram_Schimidt1} and \eqref{eq:Gram_Schimidt4} we get that
$$|b_k||\epsilon_k|^2 \leq \sum_{k' < k}|a_{k, k'}||b_{k'}||\epsilon_{k'}|^2 + |(y, x_k)| \leq A_2\rho^{n - k - 1}\Big(1 + 6A_1\sum_{k' \in [k - 1]}\rho^{2k'}\Big)\,.$$
Since $\rho < 0.25$ and $A_1\rho < 0.1$, it now follows from a routine computation that $|b_k||\epsilon_k|^2 
\leq 3A_2 \rho^{n - k - 1}$. 
Thus by induction $|(y, \epsilon_i)| \leq 3 A_2 \rho^{n 
- i - 1}$ for all $i \leq n - 2$. Similarly, we have
$$|b_{n-1}||\epsilon_{n-1}|^2 \leq A_3 + 6A_1A_2\sum_{k' \in [n - 2]}\rho^{2k'} \leq A_3 + 7A_1A_2\rho^2\,.$$
Finally when $k = n$, we get
\begin{eqnarray*}
|b_n||\epsilon_n|^2 &\leq&  \sum_{k' < n-1}|a_{n, k'}||b_{k'}||\epsilon_{k'}|^2 + |a_{n, n-1}||b_{n-1}||\epsilon_{n-1}|^2 + |(y, x_n)|\\
&\leq & 7A_1A_2\rho^3 + (A_3 + 7A_1A_2 \rho^2)A_1\rho + A_3 \leq 2A_3 + 8A_1A_2\rho^2\,.
\end{eqnarray*}
This completes the proof of the lemma.
\end{proof}
As a corollary we get the following slightly more general version.
\begin{lemma}
\label{lem:cond_general}
Consider the vectors $x_1, x_2, \ldots, x_n$ as in the 
statement of Lemma~\ref{lem:Gram_Schimidt}. Assume that 
$0 < \rho < 1/16$ and $A_1 < \tfrac{0.1}{\rho^{1/2}}$. Then if $y$ is a vector such that $|(y, x_i)| \leq A_2\rho^{(|i - \ell| - 1)^+}$ for some $\ell \in [n]$ and all $i \in [n]$, we have $|\hat y| = O(1)A_2$.
\end{lemma}
\begin{proof}
We assume, without any loss of generality, that $\ell 
\geq n/2$. Define a new sequence of vectors $\{z_i\}_{i \in [n]}$ as follows:
$$z_{n - i +1} = \begin{cases}
x_{\ell - j} &\mbox{ if } i = 2j + 1\mbox{ for some }0\leq j \leq n - \ell - 1\,, \\
x_{\ell +  j} &\mbox{ if } i = 2j\mbox{ for } j \in [n - \ell]\,, \\
x_{\ell - (n - \ell + j - 1)} &\mbox{ if } i = 2(n - \ell) + j\mbox{ for } j \geq 1 \,.
\end{cases}$$
In plain words we fold the interval $[1, n]$ around 
$\ell$ and re-index the vectors $x_i$'s accordingly. It is not difficult to check that $|(z_i, z_i')| \leq 
A_1\rho^{|i - i'|/2}$. On the other hand $|(y, z_i)| \leq A_2\rho^{-0.5}\rho^{(n - i - 1)}/2$ for $i \in 
[n-2]$ and $\leq A_2$ otherwise. Thus $y, z_1, z_2, \ldots, z_n$ satisfy the conditions of 
Lemma~\ref{lem:Gram_Schimidt}. Let $\eta_1, \eta_2, \ldots, \eta_n$ be the sequence of vectors that we obtain by applying Gram-Schimidt to $z_1, z_2, \ldots, z_n$. Then from \eqref{eq:Gram_Schimidt4} we have
$$\hat y = \sum_{i \in [n]}\frac{(y, \eta_i)}{|\epsilon_i|^2}\epsilon_i\,.$$
Hence applying Lemma~\ref{lem:Gram_Schimidt} we get
\begin{equation*}
|\hat y|^2 = \sum_{i \in [n]} \frac{|(y, \eta_i)|^2}{|\epsilon_i|^2} \leq O(1)A_2^2\,.\qedhere
\end{equation*} 
\end{proof}

\section{Regularized total variation of Brownian motion}
\label{sec:total_variation}
The notion of regularized total variation of Brownian motion, studied in \cite{Dunlap}, was a crucial ingredient in the analysis of first passage percolation on the exponential of branching random walk \cite{DG15}. 
When the underlying media is Gaussian free field it is necessary to extend this notion to \emph{inhomogeneous} penalties in order to obtain a more efficient 
optimization. This is the main goal of this section.

For a number $\lambda > 0$, the author of \cite{Dunlap} introduced the regularized total variation for a continuous function $f: [a, b] \mapsto \mathbb R$, defined by 
\begin{equation}\label{eq-def-regularized-total-variation}
\Phi_{\lambda, [a, b]}(f) = \sup_{k} \sup_{a = t_0 <t_1<\ldots < t_k < t_{k+1} = b} \Big\{\sum_{i=1}^{k+1} |f(t_i) - f (t_{i-1})| - \lambda k\Big\}\,.
\end{equation}
The main result of \cite{Dunlap} is when $f$ is given by the sample path of a standard Brownian motion $\{B_t\}_{0\leq t\leq 1}$, which states that 
\begin{equation}\label{eq-Dunlap}
\lambda \leq \E \Phi_{\lambda, [0, 1]} (B) \leq \frac{1}{\lambda} + \lambda\,.
\end{equation}
In this section we attempt a partial generalization of \eqref{eq-Dunlap} when $\lambda$ is a step function with a fixed upper bound on the number of 
steps. We will only provide a lower bound, which is all we need for the proof of 
Theorem~\ref{theo:main} (we do not expect our bound to be sharp in general). 

For a step function $\lambda: [0, 1] \mapsto (0, \infty)$ and a partition $\mathcal  P = (t_0, t_1, \ldots, t_{k+1})$ of $[0,1]$, define
\begin{equation}\label{eq-def-regularized-total-variation-generalized}
\Phi_{\lambda, \mathcal P}(f) = \sum_{i=1}^{k+1} |f(t_i) - f (t_{i-1})| - \sum_{i=1}^k\lambda(t_i)\,.
\end{equation}
We require a few more notations for the formulation of 
our result. Denote by $N_{\lambda, \star}$ the number of steps of $\lambda$ and by $\lambda_\infty$ the maximum value of $\lambda$. Also define $\lambda_*$ to be the solution to the equation
$$\int_{[0,1]} \Big(\frac{\lambda_*}{\lambda(s)}\Big)^2 ds = 1\,.$$
We are now ready to state the main theorem of this section.
\begin{theorem}\label{thm-total-variation}
For any $\epsilon$ and there exists $\delta = \delta(\epsilon) \in (0, 1)$ such that the following holds.
Let $\lambda$ be an arbitrary step function with $\lambda_* \leq \delta$ and $N_{\lambda,\star} \leq 
\lambda_*^{-M}$ for some $M > 0$. Then there exists a (random) partition $\mathcal Q^* = (q_0^*, q_1^*, \ldots, q_{k+1}^*)$ of $[0, 1]$ such that $k \leq 2/\lambda_*^2$ and
$$\E \Phi_{\lambda, \mathcal Q^*} (B) \geq (1-\epsilon) \int_{[0,1]} \frac{1}{\lambda(t)} d t - O_M(\lambda_\infty \lambda_*^{-1.5})\,.$$
\end{theorem}
Our proof of Theorem~\ref{thm-total-variation} builds 
upon the proof in \cite{Dunlap}. We remark that, while much of the work in \cite{Dunlap} was devoted to characterize the optimizer $(t_1, \ldots, t_k)$ for \eqref{eq-def-regularized-total-variation} (thus obtaining the upper bound in \eqref{eq-Dunlap}), the proof for the lower bound in \eqref{eq-Dunlap} was relatively short. 

Define
$$F(t) = \int_{[0,t]} \Big(\frac{\lambda_*}{\lambda(s)}\Big)^2 ds\,.$$
From definition of $\lambda_*$ it follows that $F(1) = 
1$. Further, we define a penalty function $\tilde \lambda : [0, 1] \mapsto (0, \infty)$ by $\tilde 
\lambda(t) = \lambda(F^{-1}(t))$. Clearly $\tilde \lambda$ is also a step function with same number of 
steps $N_{\lambda, \star}$. Let $\{W_s\}_{0\leq s\leq 
1}$ be another standard Brownian motion and $\mathcal P = (t_0, t_1, \ldots, t_{k+1})$ be a (random) partition of $[0,1]$. Now define
\begin{equation}\label{eq-def-tilde-Phi}
\tilde \Phi_{\tilde \lambda, \mathcal P}(W) = \sum_{i=1}^{k+1} \Big|\int_{[t_{i-1}, t_i]} \frac{\tilde \lambda(s)}{\lambda_*} dW_s\Big| - \sum_{i=1}^k \tilde \lambda(t_i)\,.
\end{equation}
\begin{lemma}\label{lem-W-B-total-variation}
We have that $\E \Phi_{\lambda, F^{-1}\mathcal P} (B) = 
\E \tilde \Phi_{\tilde \lambda, \mathcal P}(W)$, where $F^{-1} \mathcal P$ is the partition $(F^{-1}(t_0), F^{-1}(t_1)$ $, \ldots, F^{-1}(t_{k+1}))$.
\end{lemma}
\begin{proof}
Define $\{\tilde W_t\}_{0\leq t\leq 1}$ by 
$$\tilde W_t = \int_{[0, F(t)]} \frac{\tilde \lambda(s)}{\lambda_*} dW_s\,.$$
We see that $\{\tilde W_t\}_{0\leq t\leq 1}$ is a standard Brownian motion, and in particular has the same law as $(B_t)_{0\leq t\leq 1}$. Therefore, we obtain that
\begin{equation*}
\begin{split}
\E \tilde \Phi_{\tilde \lambda, \mathcal P}(W) &= \E \Big\{\sum_{i=1}^{k+1} \Big|\int_{[t_{i-1}, t_i]} \frac{\tilde \lambda(s)}{\lambda_*} dW_s\Big| - \sum_{i=1}^k  \tilde \lambda(t_i)\Big\}\\ 
&= \E \Big\{\sum_{i=1}^{k+1} |B(F^{-1}(t_i)) - B (F^{-1}(t_{i-1}))| - \sum_{i=1}^k\lambda(F^{-1}(t_i))\Big\} = \E \Phi_{\lambda, F^{-1}\mathcal P} (B) \,. \quad\quad\quad\quad\quad\qedhere
\end{split}
\end{equation*}
\end{proof}
In light of Lemma~\ref{lem-W-B-total-variation} we first construct a partition $\mathcal P$ that would yield a large value of $\Phi_{\lambda_*, \mathcal P}$ as 
in \eqref{eq-def-regularized-total-variation}. Then we plug this partition into \eqref{eq-def-tilde-Phi} to 
deduce the desired lower bound. The construction of the partition follows \cite{Dunlap}; the analysis on the lower bound in the second step uses renewal theory. 

For an interval $[a, b]$, we say it contains a $\lambda_*$-uptick (respectively, $\lambda_*$-downtick) if there exists $a\leq s < t\leq b$ so that $W_t - W_s = \lambda_*$ (respectively, $W_t - W_s = -\lambda_*$). Let $\tau_0 = 0$, and define recursively for all $i\geq 1$
\begin{align*}
\tau_{2i-1} &= \inf\{t\geq \tau_{2i-2}: [\tau_{2i-2}, \tau_{2i-1}] \mbox{ contains a $\lambda_*$-uptick}\}\,,\\
\tau_{2i} &= \inf\{t\geq \tau_{2i-1}: [\tau_{2i-1}, \tau_{2i}] \mbox{ contains a $\lambda_*$-downtick}\}\,.
\end{align*}
Let $N_{\tau} = \sup\{j: \tau_j \leq 1\}$. We see that $\tau_1, \tau_2, \ldots$ are stopping times of the 
Brownian motion $W$. It is obvious from strong Markov property that the random variables $\tau_1, \tau_2 - \tau_1, \tau_3 - \tau_2, \ldots$ are identically and 
independently distributed (i.i.d.). The following result is standard.
\begin{lemma}
\label{lem:tau_1}
$\tau_1$ admits a density $f_{\tau_1} \in C^\infty([0, \infty))$, $\E \tau_1 = \lambda_*^2$ and $\E \e^{\theta \tau_1} < \infty$ for a suitable $\theta > 0$.
\end{lemma}
\begin{proof}
Notice that $\tau_1 = \inf\{t \geq 0: W_t - m_t = 
\lambda_*\}$, where $m_t = \min_{s \in [0, t]} W_s$. 
Since the process $W - m$ is identically distributed as the reflected Brownian motion (see e.g. \cite[Theorem 2.34]{Peres_Morters}), $\tau_1$ has the same distribution as $\tau^* = \inf\{t \geq 0: |W_t| = \pm 
\lambda_*\}$. Consequently $\E \tau_1 = \lambda_*^2$, and $\E \e^{\theta \tau_1} = \sec \sqrt{2\theta}\lambda_*$ for $0 \leq \theta \leq \tfrac{\pi^2}{32\lambda_*^2}$ and $\mathrm{sech} \sqrt{2\theta}\lambda_*$ for $\theta 
\leq 0$. The existence of a smooth density should follow from applying Fourier inversion to the 
characteristic function of $\tau_1$. 
\end{proof}

In addition, we define for $i \geq 1$ 
$$\xi_{2i-1} = \arg\min_{\tau_{2i-2} \leq t\leq \tau_{2i-1}}W_t\,, \mbox{ and } \xi_{2i} = \arg\max_{\tau_{2i-1} \leq t\leq \tau_{2i}}W_t\,.$$
We remark that $\xi_1, \xi_2, \ldots$ are \emph{not} stopping times. Observe that the event $\{\xi_1 = s\}$ is equivalent to
$$ \{W_s = m_s, \max_{t \in [0, s]} (W_t - m_t) < \lambda_*\} \cap \{\exists t > \xi_1 \mbox{ such that } W_t = W_{\xi_1} + \lambda_*, \min_{\xi_1 \leq s\leq t} W_s = W_{\xi_1} \}\,.$$
As such, we see that $\xi_1$ is independent of $\tau_1 - 
\xi_1$. For a rigorous proof we can use the random walk approximation of $W$ obtained by sampling it at 
regularly spaced time points. The heuristic claim presented above obviously holds for such a walk. 
Consequently one can take limit as the common gap between successive time points approaches 0 which would imply the independence of $\xi_1$ and $\tau_1 - \xi_1$ due to continuity of Brownian motion paths.

Generalizing this argument and using the fact that $\tau_i$'s are stopping times, we get the next lemma.
\begin{lemma}
\label{lem:xi*_indep}
$\xi_2-\xi_1, \xi_3 - \xi_2, \xi_4 - \xi_3, \ldots,$ is a sequence of i.i.d.\ random variables with the common distribution as same as that of $\tau_1$.
\end{lemma}
\begin{proof}
Strong Markov property and the discussion in the last paragraph imply that the random variables $\xi_1, \tau_1 - \xi_1, \xi_2 - \tau_1, \tau_2 - \xi_2 + \tau_1, 
\ldots,$ are all independent. Also by definition of $\xi_j$'s it follows that the random variables $\xi_1, \xi_2 - \tau_1, \ldots,$ are identically distributed. 
Since distribution of $\tau_1$ is the convolution of the distributions of $\xi_1$ and $\tau_1 - \xi_1$, the lemma follows.
\end{proof}
We will hereafter denote the common density function of $\tau_1, \xi_2 - \xi_1, \xi_3 - \xi_2, \ldots$ by $f$. 
Although $\xi_1, \xi_2, \ldots$ are not stopping times, we nevertheless have an analogue of strong Markov property for the processes $\{W_t - W_{\xi_j}\}_{\xi_j \leq t \leq \xi_{j+1}}$ as shown by our next lemma.
\begin{lemma}
\label{lem:conditional_law}
The pairs $(\{W_t - W_{\xi_j}\}_{\xi_j \leq t \leq \xi_{j + 1}}, \xi_{j + 1} - \xi_j)$'s are independent for $j \geq 0$ (here $\xi_0 = 0$).
\end{lemma}
\begin{proof}
This follows from a slightly more general statement that the pairs $(\{W_t\}_{0\leq t \leq \xi_1}, \xi_1), (\{W_t - W_{\xi_1}\}_{\xi_1 \leq t \leq \tau_1}, \tau_1 - \xi_1), (\{W_t - W_{\tau_1}\}_{\tau_1\leq t \leq \xi_2}, 
\tau_1 - \xi_2), \ldots,$ are independent. For the proof of this fact we again appeal to the random walk 
intuition. It is easy to see that an analogous statement is true for any random walk approximation of $W$ that we mentioned earlier while discussing the independence of 
$\xi_1$ and $\tau_1 - \xi_1$. Thus we can prove the full statement by passing to limits as the common gap between successive time points approaches 0.
\end{proof}
We will use only a subset of $\{F^{-1}(\xi_1), F^{-1}(\xi_2), \ldots\}$ and possibly an extra time 
point to define the partition $\mathcal P$. Let 
$\Delta_j = (-1)^{j+1} (W_{\xi_{j+1}} - W_{\xi_j})$. We see that 
\begin{equation}\label{eq-expectation-Delta-1}
\E \Delta_j = 2 \lambda_* + (-1)^{j+1}\E (W_{\tau_{j+1}} - W_{\tau_{j}}) = 2\lambda_*\,.
\end{equation}
For $t \geq 0$, let $I_t = \sup\{j: \xi_j\leq t\}$, and let $\xi^*(t) = \xi_{I_t+1} - \xi_{I_t}$ and $A(t) = t - 
\xi_{I_t}$. Write $\Delta(t) = \Delta_{I_t}$. Recall at 
this point that $\tilde \lambda$ is a step function. So there is a (non-random) partition $\mathcal S \equiv \mathcal S_\lambda = (s_0, s_1, \ldots, s_{N_{\lambda,\star}})$ of $[0, 1]$ such that $\tilde \lambda$ is constant on each $(s_{i - 
1}, s_i)$. Moreover, by increasing $N_{\lambda, \star}$ by at most a factor of $2$, we can make the $\ell_\infty$ norm of 
$\mathcal S$ is at most $\lambda_*^M$. We can also assume, 
without loss of generality, that $M \geq 20$. From description of the random times $\xi_i$'s, it is clear that the total gain  $\sum_{j \geq 1}(-1)^{j + 1}\int_{[\xi_{j-1} \vee s_{i-1}, \xi_j \wedge s_i]}\tfrac{\tilde \lambda (s)}{\lambda_*}dW_s$ from 
$(s_{i - 1}, s_i)$ is at least
$$\tilde G_i = (-1)^{I_{s_{i - 1}} + 1}\frac{\tilde \lambda_i}{\lambda_*}(W_{s_i} - W_{s_{i-1}})\mathbf 1_{\{R(s_{i - 1}) \geq s_{i - 1}^*\}}\,,$$
where $R(t) = t - A(t)$ and $s_{i - 1}^* = s_i - s_{i - 
1}$. Now define a new collection of functions $G_t: [0, t] \mapsto \R$ as
\begin{equation}
G_t(s) = \E (W_{\xi_1 + s} - W_{\xi_1} \mid \xi_2 - \xi_1 = t)\,, \mbox{ for } 0\leq s \leq t\,.
\end{equation}
In view of Lemma~\ref{lem:conditional_law} we can write
\begin{equation}
\label{eq:expectation1}
\E \tilde G_i = \frac{\tilde \lambda_i}{\lambda_*}\E \Big(G_{\xi^*(s_{i - 1})}(A(s_{i-1}) + s_{i-1}^*) - G_{\xi^*(s_{i - 1})}(A(s_{i-1}))\Big)\mathbf 1_{\{R(s_{i-1}) \geq s^*_{i-1}\}}\,.
\end{equation}
We will analyze this expectation in several steps. A good starting point would be the joint distribution of total 
lifetime $\xi^*(t)$ and age $A(t)$. For the moment being let us forget that our renewal process is actually 
delayed by $\xi_1$, which we will take care of later. It is a well-known fact from renewal theory (see, e.g., \cite[Page 132, Theorem 7]{SR01}) that for any $t > 0$
\begin{equation}
\label{eq-age-lifetime-convergence}
(\xi^*(t), A(t)) \mbox{ converge uniformly as } \delta \to 0 \mbox{ in distribution to } (Z, ZU), 
\end{equation}
where $Z$ has density function $f_Z(z) = \frac{f(z) z}{\lambda_*^2}$, $U$ is a uniform in $[0, 1]$ 
independent of $Z$. But since we are dealing with a ``large'' number of ``tiny'' intervals we need a density 
version of this convergence if possible. So let us first 
look at the joint density of $(\xi^*(t), A(t))$. Some 
new notations will be useful for this purpose. Denote by $f^{(n)}$ the $n$-fold convolution of $f$ with itself 
and by $h$ the sum $\sum_{n \geq 1}f^{(n)}$. Thus $h$ can be thought of as the density for the renewal measure $H = \sum_{n \geq 0} F^{(n)}$ where $F^{(n)}$ is the 
distribution function for $\xi_{n+1} - \xi_1$. Also denote $\xi^*_j = \xi_{j + 1} - \xi_j$. It is straightforward that for $0 < a < t \wedge z$,
\begin{eqnarray}
\label{eq:jt_density}
\P(\xi^*(t) \leq z, A(t) \leq a) &=& \sum_{n \geq 1}\P(t \geq \xi_n \geq t - (a \wedge \xi_n^*), \xi^*_n \leq z)\nonumber\\
&=& \sum_{n \geq 1}\int\limits_{t - (a \wedge z') \leq x \leq t, z' \leq z}f^{(n)}(x)f(z')dx dz' \nonumber\\
&=& \int\limits_{0 \leq a' \leq a, z' \leq z}\sum_{n \geq 1}f^{(n)}(t - a')f(z')\mathbf 1_{\{a' \leq z'\}} da'dz' \nonumber\\
&=& \int\limits_{0 \leq a' \leq a, z' \leq z}h(t - a')f(z')\mathbf 1_{\{a' \leq z'\}} da'dz'\,.
\end{eqnarray}
Blackwell's renewal theorem states that for any given $t > 0$ and $\lambda_*$ sufficiently small, $H(t + a) - 
H(t)$ is approximately $\tfrac{a}{\lambda_*^2}$. The density version of this statement i.e. $h \approx \tfrac{1}{\lambda_*^2}$ when $t \gg \lambda_*^2$ is 
discussed in \cite{Stone65}. Below we state the main theorem in that paper in a form that suits our purpose.
\begin{theorem}\cite[Page 3]{Stone65}
\label{thm:renewal_density}
Let $\mu_H$ be the renewal measure corresponding to the 
distribution function $F$. Then there exist measures $\mu_H', \mu_H''$ on $([0, \infty), \mathcal B(0, \infty))$ such that the following conditions hold:
\begin{enumerate}[(a)]
\item $\mu_H = \mu_H' + \mu_H''$.
\item $\mu_H'$ is absolutely continuous and has a density $h'(x)$ satisfying $h'(x) = \tfrac{1}{\lambda_*^2}(1 + o_{x \to \infty}(\e^{-rx/\lambda_*^2}))$. Here $r$ is a positive, absolute constant.
\item $\mu_H''((x, \infty)) = o_{x \to \infty}(\e^{-rx/\lambda_*^2})$, where $r$ is the same constant as in part (b).
\end{enumerate}
\end{theorem}
We will use Theorem~\ref{thm:renewal_density} to 
estimate \eqref{eq:expectation1}. But before that we need to restrict $s_{i-1}$ and $\xi^*(s_{i-1})$ to a 
suitable range. To this end let $s_{i-1} \geq 
\tfrac{10M}{r_1}\lambda_*^2\log 
\tfrac{1}{\lambda_*}$ where $r_1 = r \wedge 1$. Also assume $\delta$ to be small enough so that $\E \Delta_1 \mathbf 1_{\{\lambda_*^{20} \leq \xi_1^* \leq \tfrac{100}{r_1}\lambda_*^2\log\tfrac{1}{\lambda_*}\}}\geq (1 - \epsilon)\E \Delta_1$. Let $\Omega(s_{i-1})$ be the event that $\{s_{i-1}^* \leq \xi^*(s_{i-1}) \leq \tfrac{5M}{r_1}\lambda_*^2\log\tfrac{1}{\lambda_*}\}$. Thus
\begin{equation}
\label{eq:expectation2}
\E \tilde G_i\mathbf 1_{\Omega(s_{i-1})} = \E \tilde G_{i, 2} -\E \tilde G_{i, 1}\,,
\end{equation}
where 
\begin{align*}
\tilde G_{i, 1} &= \frac{\tilde \lambda_i}{\lambda_*}\int\limits_{\substack{0 < a < z-s_{i-1}^*,\\ s_{i-1}^*\leq z \leq \tfrac{5M}{r_1} \lambda_*^2\log \tfrac{1}{\lambda_*}}}G_z(a)h(s_{i -1} - a)f(z)da dz\,, \mbox{ and}\\
\tilde G_{i, 2} &= \frac{\tilde \lambda_i}{\lambda_*}\int\limits_{\substack{s_{i-1}^* < a < z,\\ s_{i-1}^*\leq z \leq \tfrac{5M}{r_1}\lambda_*^2\log \tfrac{1}{\lambda_*}}}G_z(a)h(s_i - a)f(z)da dz\,.
\end{align*}


We can further divide the range of integration in the definition of $\tilde G_{i, 1}$ (or $\tilde G_{i, 2}$) into two parts based on whether $a \leq s_{i-1}^*$ or not (respectively whether $a \geq z - s_{i-1}^*$ 
or not). Denote the corresponding random variables as $\tilde G_{i, 1, \side}, \tilde G_{i, 1, \midd}$ and $\tilde G_{i, 2, \side}, \tilde G_{i, 2, 
\midd}$ respectively. We will first show that $\E (\tilde G_{i, 2, \midd} - \tilde G_{i, 1, \midd})$ is small. To this end, note
\begin{equation}
\label{eq:expectation_neglect1}
\E (\tilde G_{i, 2, \midd} - \tilde G_{i, 1, \midd}) = \frac{\tilde \lambda_i}{\lambda_*}\int\limits_{\substack{s_{i-1}^* < a < z-s_{i-1}^*,\\ s_{i-1}^* \leq z \leq \tfrac{5M}{r_1} \lambda_*^2\log \tfrac{1}{\lambda_*}}}\E G_{z, \Delta(s_{i - 1})}(a)(h(s_i - a) - h(s_{i -1} - a))f(z)da dz
\end{equation}
From Theorem~\ref{thm:renewal_density}, we get
\begin{eqnarray}
\label{eq:expectation_neglect2}
|\E (\tilde G_{i, 2, \midd} - \tilde G_{i, 1, \midd})| &\leq& \frac{\tilde \lambda_i}{\lambda_*}\e^{-r(s_i/\lambda_*^2 - \tfrac{5M}{r_1}\log\tfrac{1}{\lambda_*})}\int\limits_{[s_{i-1}^*, 5M\log \tfrac{1}{\lambda_*}]}\E (\Delta_1 | \xi_1^* = z)f(z) dz \nonumber \\
&\leq& O(\tilde\lambda_i)\lambda_*^{-5M}\e^{-rs_i/\lambda_*^2}\,.
\end{eqnarray}
Next we will estimate contribution from the ``main term'' i.e. $\E(\tilde G_{i, 2, \side} - \tilde G_{i, 1, \side})$. 
Notice that 
\begin{eqnarray}
\label{eq:expectation3}
\E \tilde G_{i, 2, \side} &=& \frac{\tilde \lambda_i}{\lambda_*}\E (-1)^{I_{s_{i-1}} + 1}(W_{s_i} - W_{\xi_{I_{s_{i-1}}}})\mathbf 1_{\{R(s_i) \leq s_{i-1}^*\} \cap \Omega(s_{i-1})} \nonumber \\
&=& \frac{\tilde \lambda_i}{\lambda_*}\E \big(\Delta(s_{i-1}) + (-1)^{I_{s_{i-1}} + 1}(W_{s_i} - W_{\xi_{I_{s_i} + 1}})\big)\mathbf 1_{\{R(s_i) \leq s_{i-1}^*\} \cap \Omega(s_{i-1})}\,.
\end{eqnarray}
Similarly
\begin{equation}
\label{eq:expectation4}
\E \tilde G_{i, 1, \side} = \frac{\tilde \lambda_i}{\lambda_*}\E (-1)^{I_{s_{i-1}} + 1}(W_{s_{i-1}} - W_{\xi_{I_{s_{i-1}}}})\mathbf 1_{\{A(s_{i-1}) \leq s_{i-1}^*\} \cap \Omega(s_{i-1})}
\end{equation}
The last two displays clearly point at three distinct components of $\E(\tilde G_{i, 2, \side} - \tilde G_{i, 
1, \side})$ among which the following is most significant:
\begin{equation}
\label{eq:expectation5}
\mathrm I = \frac{\tilde\lambda_i}{\lambda_*}\int\limits_{\substack{z-s_{i-1}^* \leq a < z,\\ s_{i-1}^* \leq z \leq \tfrac{5M}{r_1} \lambda_*^2 \log \tfrac{1}{\lambda_*}}}\E (\Delta_1 | \xi_1^* = z)h(s_i - a)f(z)da dz\,.
\end{equation}
Using Theorem~\ref{thm:renewal_density} we get
\begin{eqnarray}
\label{eq:expectation6}
\mathrm I &\geq& \frac{\tilde \lambda_i}{\lambda_*}\Big(\frac{s_{i-1}^*}{\lambda_*^2} - \e^{-r(s_i/\lambda_*^2 - \tfrac{5M}{r_1}\log \tfrac{1}{\lambda_*})}\Big) \int\limits_{[s_{i-1}^*, \tfrac{5M}{r_1} \lambda_*^2\log \tfrac{1}{\lambda_*}]}\E (\Delta_1 | \xi_1^* = z)f(z)dz \nonumber\\
&\geq& (1 - \epsilon)\frac{2\tilde \lambda_i s_{i-1}^*}{\lambda_*^2} - O(\tilde \lambda_i)\lambda_*^{-5M}\e^{-rs_i/\lambda_*^2}\,,
\end{eqnarray}
The remaining components are $\E \tilde G_{i, 2, \side} 
- \mathrm I$ and $\E \tilde G_{i, 1, \side}$. Notice that these expectations involve differences in the value of $W$ sampled at very closely located time points and 
consequently are expected to be very small. In fact the H\"{o}lder continuity of Brownian motion paths implies that $|W_s - W_t| \leq C_{W, \eta}|s - t|^{1/2 - \eta}$ 
for any $\eta \in (0, 1/2)$. It is also standard that the constant $C_{W, \eta}$ has good moment 
behavior. However we still give a proof of this fact in Lemma~\ref{lem:Holder_constant_main} for sake of 
completeness. As a preliminary we need the following result.
\begin{lemma}
\label{lem:Holder_constant_prelim}
Let $Z_1, Z_2, \ldots, Z_{N}$ be i.i.d.\ standard Gaussian variables and $S_K = \sum_{j \leq K}Z_j$ for $1 \leq K 
\leq N \equiv 2^n$. Define $C_{n, \eta} = \max_{M \in [N], K \in [N - M] \cup \{0\}} \tfrac{|S_{K + M} - 
S_K|}{M^{1/2}(N/M)^\eta}$ for some small, positive 
$\eta$. Then $\E \e^{C_{n, \eta}} = O_\eta(1)$.
\end{lemma}
\begin{proof}
Define a new family of random variables $\{T_{K, M}\}_{M \in [N], K \in [N - M]\cup \{0\}}$ as $T_{K, M} = S_{K + M} - S_{K}$ and let $T_m^*$ be the maximum of $\{T_{K, M}\}_{2^m \leq M < 2^{m + 1}, K \in [N - M] \cup 
\{0\}}$. It is clear that $\var (T_{K, M} - T_{K', M}) = 2(|K - K'| \wedge M)$ while $\var (T_{K, M} - 
T_{K, M'}) = 2|M - M'|$. We will use a standard chaining 
argument to show that $\E {T_m^*}^+ = O(1)$. To this end, let us denote by $[N]_j$ the set of all integers in 
$[N - 1] \cup \{0\}$ that are multiples of $2^{n-j}$. 
Then we have,
\begin{eqnarray*}
\E {T_m^*}^+ &\leq& \E \max_{2^m \leq M < 2^{m+1}}T_{0, M}^+ + \sum_{n - m \leq j < n}\E \max_{K \in [N]_j, 2^m \leq M < 2^{m+1}}(T_{K + 2^{n - j - 1}, M} - T_{K, M})^+\\ 
&& + \E \max_{K \in [N]_{n-m}, 2^m \leq M < 2^{m+1}}(T_{K, M} - T_{0, M})^+\\
&\leq& O(2^{m/2}) + \sum_{n - m \leq j < n}\E \max_{K \in [N]_j, 2^m \leq M < 2^{m+1}}(T_{K + M, 2^{n-j-1}} - T_{K + 2^m, 2^{n-j-1}})^+ \\
&& + \sum_{n - m \leq j < n}\E \max_{K \in [N]_j}(T_{K + 2^{n - j - 1}, 2^m} - T_{K, 2^m})^+ + \E \max_{K \in [N]_{n-m}, 2^m \leq M < 2^{m+1}}(T_{K, M} - T_{0, M})^+\\
&\leq& O(2^{m/2}) + O(1)\sum_{0 \leq j < m}\big(\sum _{0 \leq j' < j}(\sqrt{n + m - j'- j - 2}+1)2^{j'/2} + (\sqrt{n + m - 2j - 1} + 1)2^{j/2}\big) \\
&& + O(1)\sum_{0 \leq j < m}(\sqrt{n - j - 1} + 1)2^{j/2} + O(1) \sum_{0 \leq j' < m}(\sqrt{n - j' - 1}+1)2^{j'/2} \\
&\leq& O(2^{m/2})(\sqrt{n - m} \vee 1)\,,
\end{eqnarray*}
where we have repeatedly used Lemma~\ref{lem:basic_chaining} in the second and third 
step. Denoting $C_{+, n, \eta} = \max_{M \in [N], K \in [N - M] \cup \{0\}} \tfrac{S_{K + M} - S_K}{M^{1/2}(N/M)^\eta}$, the last display gives us
\begin{align}
\label{eq:Holder_constant_prelim1}
\E C_{+, n, \eta} \leq O(1)\sum_{j \leq n} 2^{-\eta j}j^{3/2} + O(1) = O_{\eta}(1)\,.
\end{align}
Similarly $\E C_{-, n, \eta} = -O_\eta(1)$ where $C_{-, n, \eta} = \min_{M \in [N], K \in [N - M] \cup \{0\}} \tfrac{S_{K + M} - S_K}{M^{1/2}(N/M)^\eta}$. Now observe that
$$C_{n, \eta} \leq \Big(C_{+, n, \eta} - \frac{Z_1}{N^\eta}\Big) + \Big(\frac{Z_1}{N^\eta} - C_{-, n, \eta}\Big) + \frac{|Z_1|}{N^\eta}\,.$$ 
An immediate consequence of this is $\E C_{n, \eta} 
= O_\eta(1)$. Also since $\var \tfrac{S_{K + M} - S_K}{M^{1/2}(N/M)^\eta} \leq 1$ for all $K, M$, it follows from Lemma~\ref{lem:Borell_ineq} that $\E \e^{C_{n, \eta}} = O_\eta(1)$, thus completing the proof of the lemma.
\end{proof}
\begin{lemma}
\label{lem:Holder_constant_main}
Let $\{B_t\}_{t \geq 0}$ be a standard Brownian motion and $C_{B, \eta} = \sup_{\theta \in (0, 1), t \in [0, 1 - \theta]}\tfrac{|B_{t + \theta} - B_t|}{\theta^{1/2 - 
\eta}}$ where $\eta \in (0, 1/2)$. Then $\E \e^{C_{B, \eta}} = O_\eta(1)$.
\end{lemma}
\begin{proof}
For each $n \geq 1$, define 
$$C_{B, n, \eta} = \max_{M \in [2^n], K \in [2^n - M] \cup \{0\}}\frac{\big|B_{\frac{K + M}{2^n}} -  B_{\frac{K}{2^n}}\big|}{\big(\frac{M}{2^n}\big)^{1/2 - \eta}}\,.$$
It is easy to see that $C_{B, n, \eta}$ is distributed identically as $C_{n, \eta}$ in 
Lemma~\ref{lem:Holder_constant_prelim}. Also the sequence $\{C_{B, n, \eta}\}$ is \emph{nondecreasing} and converges almost surely to $C_{B, \eta}$ due to 
continuity of Brownian motion paths. Hence from Lemma~\ref{lem:Holder_constant_prelim}, we can conclude $\E \e^{C_{B, \eta}} = O_\eta(1)$.
\end{proof}
Applying Lemma~\ref{lem:Holder_constant_main} (for $\eta = 3/8$) and H\"{o}lder's inequality to the expression for $\E \tilde G_{i, 1, \side}$ in \eqref{eq:expectation4}, we get that
$$|\E \tilde G_{i, 1, \side}| \leq O(1)\frac{\tilde \lambda_i}{\lambda_*}{s_{i-1}^*}^{3/8}\big(\P(\{A(s_{i-1}) \leq s_{i-1}^*\} \cap \Omega(s_{i-1}))\big)^{2/3}\,.$$
We can then use Theorem~\ref{thm:renewal_density} and \eqref{eq:jt_density} to obtain
\begin{eqnarray}
\label{eq:expectation7}
|\E \tilde G_{i, 1, \side}| &\leq& O(1)\frac{\tilde \lambda_i}{\lambda_*}{s_{i-1}^*}^{3/8}\Big(\int\limits_{\substack{0 < a \leq s_{i-1}^*,\\ s_{i-1}^* \leq z \leq \tfrac{5M}{r_1}\lambda_*^2\log \tfrac{1}{\lambda_*}}}h(s_{i-1} - a)f(z)da dz \Big)^{2/3}\nonumber \\
&\leq & O(1)\frac{\tilde \lambda_i}{\lambda_*}{s_{i-1}^*}^{3/8}\Big(\frac{s_{i-1}^*}{\lambda_*^2} + \lambda_*^{-5M}\e^{-rs_i/\lambda_*^2}\Big)^{2/3}\,.
\end{eqnarray}
Similarly one can bound $|\E \tilde G_{i, 2, \side} - 
I|$. So what remains is the term $\E \tilde G_i 
\mathbf 1_{\Omega(s_{i-1})^c}$. By H\"{o}lder's inequality and Lemma~\ref{lem:Holder_constant_main},
$$|\E \tilde G_i \mathbf 1_{\Omega(s_{i-1})^c}| \leq O(1)\frac{\tilde \lambda_i}{\lambda_*}{s_{i-1}^*}^{3/8}\big(\P(\{\xi^*(s_{i-1}) > \tfrac{5M}{r_1}\lambda_*^2\log \tfrac{1}{\lambda_*}\})\big)^{2/3}\,.$$
Also from \eqref{eq:jt_density}, Theorem~\ref{thm:renewal_density} and the bound on the moment generating function of $\xi_1^*$ from Lemma~\ref{lem:tau_1}, we have
\begin{eqnarray*}
\P(\{\xi^*(s_{i-1}) \leq \tfrac{5M}{r_1}\lambda_*^2\log \tfrac{1}{\lambda_*}\}) &\geq& \int\limits_{[0, \tfrac{5M}{r_1}\lambda_*^2\log \tfrac{1}{\lambda_*}]}\frac{z}{\lambda_*^2}f(z)dz - \lambda_*^{-5M}\e^{-rs_{i-1}/\lambda_*^2}\\
&\geq& 1 - O(\lambda_*^{1.5M}) - \lambda_*^{-M}\e^{-rs_{i-1}/\lambda_*^2}\,.
\end{eqnarray*}
Hence
\begin{equation}
\label{eq:expectation8}
|\E \tilde G_i \mathbf 1_{\Omega(s_{i-1})^c}| \leq O(1)\frac{\tilde \lambda_i}{\lambda_*}{s_{i-1}^*}^{3/8}\big(O(\lambda_*^{1.5M}) + \lambda_*^{-5M}\e^{-rs_{i-1}/\lambda_*^2}\big)^{2/3}\,.
\end{equation}
Collecting everything together we get
\begin{equation}
\label{eq:expectation9}
\E \tilde G_i \geq (1 - \epsilon)\frac{2\tilde \lambda_is_{i-1}^*}{\lambda_*^2} - O(\tilde \lambda_i)\lambda_*^{\frac{25M}{24} - \frac{7}{3}}\,.
\end{equation}
Now we will define a random partition $\mathcal P$ of $[0, 1]$ which takes into account the delaying 
effect by $\xi_1$. Let $i'$ be the minimum integer such that $s_{i'} \geq \tfrac{13M}{r_1}\lambda_*^2\log 
\tfrac{1}{\lambda_*}$. On the event $\{\xi_1 \leq \tfrac{3M}{r_1}\lambda_*^2\log\tfrac{1}{\lambda_*}\}$, let $\mathcal P$ consist of all the time points $\xi_2, \xi_3, \ldots$ that lie between $s_{i'}$ and $1$ 
along with $0, 1$ and $s_{i'}$. On the other hand if $\xi_1 > \tfrac{3M}{r_1}\lambda_*^2\log \tfrac{1}{\lambda_*}$, we simply define $\mathcal P$ to 
be $(0, 1)$. From Lemma~\ref{lem:conditional_law} and the discussions leading to \eqref{eq:expectation9} we get for $i \geq i' + 1$,
$$\E (\tilde G_i | \{\xi_1 \leq \tfrac{3M}{r_1}\lambda_*^2\log \tfrac{1}{\lambda_*}\}) \geq (1 - \epsilon)\frac{2\tilde \lambda_is_{i-1}^*}{\lambda_*^2} - O(\tilde \lambda_i)\lambda_*^{\frac{25M}{24} - \frac{7}{3}}\,.$$
Also as a consequence of Lemma~\ref{lem:tau_1} we have
$$\P(\xi_1 \leq \tfrac{3M}{r_1}\lambda_*^2\log \tfrac{1}{\lambda_*}) \geq 1 - \lambda_*^{0.9M}\,.$$
Thus for all $i \geq i' + 1$,
$$\E \tilde G_i \mathbf 1_{\{\xi_1 \leq \tfrac{3M}{r_1}\lambda_*^2\log \tfrac{1}{\lambda_*}\}} \geq (1 - \epsilon - \lambda_*^{0.9M})\frac{2\tilde \lambda_is_{i-1}^*}{\lambda_*^2} - O(\tilde \lambda_i)\lambda_*^{\frac{25M}{24} - \frac{7}{3}}\,.$$
Summing over $i$ and using the fact that $|\mathcal S| = O(\lambda_*^{-M})$ we get,
\begin{eqnarray}
\label{eq:expectation10}
\sum_{i'+1}^{N_{\lambda, \star}}\E \tilde G_i\mathbf 1_{\{\xi_1 \leq \tfrac{3M}{r_1}\lambda_*^2\log \tfrac{1}{\lambda_*}\}} &\geq&  (1 - \epsilon - \lambda_*^{0.9M})\int\limits_{[\tfrac{13M}{r_1}\lambda_*^2\log \tfrac{1}{\lambda_*}, 1]}\frac{2\tilde \lambda(s)}{\lambda_*^2}ds - O(\lambda_\infty\lambda_*^{\frac{M}{24} - \frac{7}{3}})\nonumber \\
&\geq& (1 - \epsilon - \lambda_*^{18})\int\limits_{[\tfrac{13M}{r_1}\lambda_*^2\log \tfrac{1}{\lambda_*}, 1]}\frac{2\tilde \lambda(s)}{\lambda_*^2}ds - O(\lambda_\infty\lambda_*^{-1.5})
\,,
\end{eqnarray}
where in the last inequality we used $M \geq 20$. Now notice that 
$$\int_{[\tfrac{13M}{r_1}\lambda_*^2\log \tfrac{1}{\lambda_*}, 1]}\frac{2\tilde \lambda(s)}{\lambda_*^2}ds = \int_{[F^{-1}\big(\tfrac{13M}{r_1}\lambda_*^2\log \tfrac{1}{\lambda_*}\big), 1]}\frac{1}{\lambda(s)}ds\,.$$
From Cauchy-Schwartz inequality we get,
$$\int\limits_{[0, F^{-1}\big(\tfrac{13M}{r_1}\lambda_*^2\log \tfrac{1}{\lambda_*}\big)]}\frac{1}{\lambda(s)}ds \leq \frac{\sqrt{F^{-1}\big(\tfrac{13M}{r_1}\lambda_*^2\log\tfrac{1}{\lambda_*}}\big)}{\lambda_*}\sqrt{\int_{[0, F^{-1}\big(\tfrac{13M}{r_1}\lambda_*^2\log \tfrac{1}{\lambda_*}\big)]}\frac{\lambda_*^2}{\lambda(s)^2}ds}\,.$$
The second factor on the right hand side of the last display is $\sqrt{\tfrac{13M}{r_1}\lambda_*^2\log \tfrac{1}{\lambda_*}}$ by definition of $F$ while the 
first factor can be at most $\tfrac{1}{\lambda_*}$. Hence
\begin{equation}
\label{eq:expectation11}
\sum_{i'+1}^{N_{\lambda, \star}}\E \tilde G_i\mathbf 1_{\{\xi_1 \leq \tfrac{3M}{r_1}\lambda_*^2\log \tfrac{1}{\lambda_*}\}} \geq  (1 - \epsilon - \lambda_*^{18})\int\limits_{[0, 1]}\frac{2}{\lambda(s)}ds - \sqrt{\tfrac{13M}{r_1}\log 
\tfrac{1}{\lambda_*}} - O(\lambda_\infty\lambda_*^{-1.5})\,.
\end{equation}
Finally we want to bound the number of switches without 
compromising too much in terms of the net gain. For this purpose let us try to estimate the amount that we may loose in terms of expected gain if we stop switching 
beyond $\xi_{2/\lambda_*^2}$. From Lemma~\ref{lem:conditional_law} and \eqref{eq-expectation-Delta-1} we can see that this expected loss is bounded by $2\lambda_\infty\P(\tau_{2/\lambda_*^2} \leq 1)(\E I_{\xi_1^*}+ 1)$, where $I_1^* = \sup 
\{i: \xi_i - \xi_1 \leq 1\}$. Using Markov's inequality and the expression for $\E \e^{\theta\tau_1}$ provided in the proof of Lemma~\ref{lem:tau_1}, we get $\E 
I_1^* = O(1/\lambda_*^2)$. Also it is straightforward that $\P(\tau_{2/\lambda_*^2} \leq 
1) = O(\lambda_*^2)$. Now let $\mathcal P^* \equiv (t_0^*, t_1^*, \ldots, t_{k+1}^*)$ denote the partition obtained after throwing off all the $\xi_j$'s from $\mathcal P$ for $j > 2/\lambda_*^2$. 
From the discussion we had so far and the fact that $\lambda_*$ is small, it follows
\begin{equation}
\label{eq:expectation12}
\E \Big\{\sum_{i = 1}^{k + 1}\big|\int_{[t_{i-1}^*, t_i^*]}\frac{\tilde \lambda(s)}{\lambda_*}dW_s\big|\Big\} \geq (1 - \epsilon - \lambda_*^{18})\int\limits_{[0, 1]}\frac{2}{\lambda(s)}ds - O_M(\lambda_\infty\lambda_*^{-1.5})\,.
\end{equation}

Having computed the expected gain we now turn to the 
cost incurred from penalties. By Theorem~\ref{thm:renewal_density} again, we deduce that 
$$\E \Big(\sum_{j:\xi_j \in (s_{i-1}, s_i)}\tilde \lambda(\xi_j)\Big) = \tilde \lambda_i \int_{s_{i-1} < t \leq s_i}h(t) dt \leq \tilde \lambda_i \Big( \frac{s_{i-1}^*}{\lambda_*^2} + \e^{-rs_{i-1}/\lambda_*^2}\Big)\,.$$
Summing over $i$, this gives
\begin{equation}
\label{eq:expectation13}
\E \Big(\sum_{j = 1}^k\tilde \lambda(t_j^*)\Big) \leq \int\limits_{[0, 1]}\frac{\tilde \lambda(s)}{\lambda_*^2}ds + O(\lambda_\infty)\lambda_*^{12M} = \int\limits_{[0, 1]}\frac{1}{\lambda(s)}ds + O(\lambda_\infty)\lambda_*^{12M}\,.
\end{equation}

Therefore, we conclude that for an arbitrarily fixed $\epsilon > 0$ and sufficiently small $\delta > 0$
$$\E \tilde \Phi_{\tilde \lambda, \mathcal P^*}(W)  \geq (1-2\epsilon - 2\lambda_*^{18}) \int\limits_{[0, 1]}\frac{1}{\lambda(s)}ds - O(\lambda_\infty\lambda_*^{-1.5}) \geq  (1-3\epsilon) \int\limits_{[0, 1]}\frac{1}{\lambda(s)}ds - O_M(\lambda_\infty\lambda_*^{-1.5})\,.$$
Combined with Lemma~\ref{lem-W-B-total-variation}, this completes the proof of Theorem~\ref{thm-total-variation}.
\begin{remark}
\label{remark:total-variation-general-time}
Using the scaling property of Brownian motion we can easily adapt the bound in Theorem~\ref{thm-total-variation} when the underlying process is $\{B_t\}_{0 \leq t \leq T}$ for some $T > 0$. 
In this case the penalty function $\lambda$ is defined on $[0, T]$ and the functional $\lambda_*$ satisfies the equation 
$$\int_{[0, T]}\Big(\frac{\lambda_*}{\lambda(s)}\Big)^2 ds = 1\,.$$
Keeping other conditions and definitions same as in Theorem~\ref{thm-total-variation}, we can then find a partition $\mathcal Q^* = (q_0^*, q_1^*, \ldots, q_{k+1}^*)$ of $[0, T]$ such that $k \leq 2/\lambda_*^2$ and
$$\E \Phi_{\lambda, \mathcal Q^*} (B) \geq (1-\epsilon) \int_{[0,T]} \frac{1}{\lambda(t)} d t - O_M(\lambda_\infty \lambda_*^{-1.5})\,.$$
\end{remark}

\section{Multi-scale analysis on light crossings}
\label{sec:the_main}
We are now ready to carry out the multi-scale analysis. 
In Subsection~\ref{subsec:descrip} we will describe our recursive construction for light crossings, for which we introduce Strategy~I (the base and the easy case) and 
Strategy II (the hard case). Strategy II is the essential construction, where Strategy I guarantees that the multiple scales can be decomposed into a number of blocks (each of which consists of constant number of 
scales) that are independent of each other.  In Subsection~\ref{subsec:induct_hypo}, we formulate a list of induction hypotheses, most of which are of auxiliary nature (i.e., in order to facilitate the verification of 
the central hypothesis \eqref{hypo:expected_weight}). In Subsection~\ref{subsec:easy}, we prove the induction hypothesis for Strategy I; in Subsections~\ref{subsec:induct_step_hard}, \ref{subsec:error_terms} and \ref{subsec:hard-remaining}, we give precise constructions for Strategy II and verify all the induction hypotheses for this case (here we will extensively use results obtained in the previous two sections).

\subsection{The recursive construction of crossings}
\label{subsec:descrip}
We will build crossings recursively starting from straight line crossing at the bottommost level. 
Our goal is to give an algorithm $\A_\ell$ that will produce two ``identically distributed'' crossings, one through each of $\tilde{V}_{\ell; 1}^\Gamma$ and 
$\tilde{V}_{\ell;2}^\Gamma$. We denote these two crossings by $\cross^{*, \ell}_{1}$ and $\cross^{*, 
\ell}_{2}$ respectively. The algorithm will take as its inputs the free field on $\tilde{V}_\ell^\Gamma$ and possibly some additional random variables independent 
with $\{\eta_{n, .}\}$. In general we can apply $\A_\ell$ to any $B \in \B_{\ell; \principal}$ with the corresponding fine field as the underlying field and generate a crossing $\cross^{*, B}_i$ through $B_i$ for 
$i \in [2]$. By abuse of terminology we will often refer a crossing through $B_1$ or $B_2$ as a crossing through 
$B$. In the next paragraph we discuss the general idea 
behind $\A_\ell$. We discussed a similar idea in a different setup in the beginning of Subsection~\ref{subsec:heuristic}.

Suppose that at the beginning of step $\ell$, we have built crossings for all levels $\ell' < 
\ell$. In particular we have crossings through $\tilde V_{\ell; i, j, k}^\Gamma$ for $i, j, k \in [2]$ and through each $\tilde V^{\Gamma, z}_{\ell - m - 1}$ in 
$\tilde V^\Gamma_{\ell; \midd, i}$ for $i \in [2]$. Thus for any interval $I$ contained in the base of $\tilde V_{\ell; i, j, 1}^\Gamma$ (and hence $\tilde V_{\ell; i, j, 2}^\Gamma$), we can use the crossing through either $\tilde V_{\ell; i, j, 1}^\Gamma$ or $\tilde V_{\ell; i, 
j, 2}^\Gamma$ to travel across $I$. Similar thing is true if $I$ is an subinterval of the base of some (hence all) $\tilde V^{\Gamma, z}_{\ell - m - 1}$ in $\tilde 
V^\Gamma_{\ell; \midd, i}$. Roughly speaking, we can break up the horizontal range of $\tilde V_{\ell; i}^\Gamma$ into several intervals and choose the ``horizontal piece'' for $\cross_i^{*, \ell}$ along each such interval from one of the available crossings. 
Such \emph{switchings} enable us to select economic pieces for building $\cross_i^{*, \ell}$. However, we also incur additional weight from the crossings (or paths) that are needed to link these pieces. We refer to this type of crossings as \emph{gadgets}. Therefore the gain (here, by gain we mean \emph{decrement} of the weight since our goal is to bound the weight from above) obtained from switchings must exceed the weight of 
gadgets in order to build efficient crossings. In Subsection~\ref{subsec:induct_step_hard} we show that this can be achieved by solving a regularized total variation problem that was discussed in section~\ref{sec:total_variation}. But, as already mentioned in Subsection~\ref{sec:new-challenge}, 
this requires the fine structures of crossings through $\tilde V_{\ell; i, j, k}^\Gamma$ which makes the switching locations across all levels dependent. 
This dependence causes some technical problems and in order to avoid that we periodically use a naive construction to decorrelate our choices. Thus our strategy will vary with the location of $\ell$ between two successive integer multiples of the period which we choose as $200m_\Gamma$.

We will elaborate more on our periodic scheme of 
construction. But before that we need to do some 
preparatory work. The tree representation $\T_n$ introduced in Subsection~\ref{subsec:GFF_represent} will 
be useful here. Let $a$ be the integer $\lfloor \ell / 200m_\Gamma\rfloor$. Go down each branch descending from $\tilde{V}_\ell^\Gamma$ in $\T_n$ until the first time a 
node of depth $\leq 200am_\Gamma - 1$ appears. Denote the family of nodes thus obtained by $\descend_\ell = \descend_{\ell, 200\lfloor \ell / 200m_\Gamma\rfloor 
m_\Gamma}$. The rectangles $B_1$ and $B_2$, where $B \in \descend_\ell$, define another family of rectangles which we denote as $\widetilde{\descend}_\ell = \widetilde{\descend}_{\ell, 200\lfloor \ell / 
200m_\Gamma\rfloor m_\Gamma}$. So at the beginning of step~$\ell$ we have a crossing $\cross^{*, B}$ through each $B$ in $\widetilde \descend_\ell$.
In \cite{DG15} we introduced a notion of coarsening of paths to describe the multi-level 
construction of crossings. In short, coarsening of a path is the sequence of squares in a dyadic partition of 
$\Z^2$ that the path visits along its way. Choice of dyadic partition was natural given the hierarchical 
definition of the underlying field. In the present paper, we use the rectangles in $\widetilde \descend_\ell$ to define a 
similar notion. Some features of the rectangles in $\widetilde{\descend}_\ell$ are worth mentioning here. 
Bases of the rectangles in $\widetilde{\descend}_\ell$ form the collection $\mathscr C_{\ell, \Gamma, 0; \ell\%200(m_\Gamma) + 1}$, where the notation $a 
\% b$ means $a$ mod $b$. The collection of spans of the rectangles on a given base $I \in \mathscr C_{\ell, \Gamma, 0; \ell\%200m_\Gamma + 1}$ is $\mathscr C_{\ell + 1, 1, 0; d, \principal}$ for some $d$ between $\ell \% 200m_\Gamma + 2$ and $\ell \% 
200m_\Gamma + m + 2$. Thus the spans of any two rectangles in $\widetilde \descend_\ell$ are either disjoint or one is a subset of the other, whereas their bases could either be same or have at most two points in 
common. Since the intervals in $\mathscr C_{n, k, 0; d}$ are naturally aligned from left to right, we can make sense of an aligned sequence of rectangles in 
$\widetilde \descend_\ell$. We say a sequence of intervals in $\mathscr C_{\ell, \Gamma, 0; \ell\%200m_\Gamma + 1}$ is \emph{right aligned} if they are aligned from left to right and every consecutive pair 
is adjacent, and we say a sequence of rectangles $\{R_1, R_2, \ldots, R_b\}$ in $\widetilde \descend_\ell$ is \emph{oriented} if the bases of its elements form a 
right aligned sequence of intervals. Now consider a descendant $B$ of $\tilde V_\ell^\Gamma$ at some level 
$\ell' \geq a200m_\Gamma$. 
We say that an oriented sequence of rectangles $S = \{R_1, R_2, \ldots, R_b\}$ is a \emph{skeleton} for the crossing $\cross_i^{*, B}$ if the following conditions are satisfied:\\
(a) For each $I$ in the collection $\mathscr C_{\ell', \Gamma, p_{B, \mathrm{base}}; \ell' \% 200m_\Gamma + 1}$, there is a unique rectangle in $S$ that is based 
on $I$. Here $p_{B, \mathrm{base}}$ is the left endpoint of the base of $S$. Thus the elements of $S$ are naturally indexed by intervals in $\mathscr C_{\ell', \Gamma, p_{B, \mathrm{base}}; \ell' \% 200m_\Gamma + 1}$.\\
(b) All the rectangles are subsets of $B_i$.\\
(c) There are connected multisets $g_{1, 2}, g_{2, 3}, \ldots g_{b-1, b}$ (which we also refer as gadgets) such that $g_{b', b'+1}$ connects the pair $(\cross^{*, R_{b'}}, \cross^{*, R_{b'+1}})$ and that $\cross_i^{*, B}$ is the union of $\cross^{*, R_j}$'s and $g_{b', b'+1}$'s (as multisets).

Using this skeleton $S$ we can define a notion of coarsening of $\cross_i^{*, B}$ for any level $\ell''$ 
between $200am_\Gamma$ and $\ell'$. To this end note that as a consequence of the definition of oriented sequences, the span of each rectangle $R_{b'}$ in $S$ is a subset of a unique interval $I_{b', \ell''}$ in $\mathscr C_{\ell' + 1, 1, 0; \ell' - \ell'' + 1, \principal}$. We call the sequence $\{I_{1, \ell''}, I_{2, \ell''}, \ldots, I_{b, \ell''}\}$ as the \emph{$\ell''$-coarsening} of $\cross_i^{*, B}$ or 
rather the skeleton $S$ to be precise. Similar to $S$, the elements of $\ell''$-coarsening of $\cross_i^{*, B}$ are also indexed by $\mathscr C_{\ell', \Gamma, p_{B, 
\mathrm{base}}; \ell' \% 200m_\Gamma + 1}$. We refer the reader to Figure~\ref{fig:skeleton} for an illustration.
\begin{figure}[!htb]
\centering
\begin{tikzpicture}[semithick, scale = 4]

\foreach \y/\z in {0/0.15, 0.45/0.6, 0.65/0.8}
{
  \draw (-2, \y) rectangle (-1.35, \z);
}
\foreach \y/\z in {0/0.15, 0.2/0.35, 0.45/0.6}
{
  \draw (2, \y) rectangle (1.35, \z);
}

\foreach \y/\z in {0/0.15, 0.45/0.6, 0.65/0.8}
{

  \draw (-0.325, \y) rectangle (0.325, \z);

}
\foreach \y/\z in {0/0.15, 0.2/0.35, 0.65/0.8}
{
  \draw (0.375, \y) rectangle (1.025, \z);
}

\foreach \y/\z in {0.2/0.35, 0.45/0.6, 0.65/0.8}
{
  \draw (-0.375, \y) rectangle (-1.025, \z);
}


\foreach \y/\z in {0/0.06, 0.09/0.15, 0.29/0.35, 0.45/0.51, 0.54/0.6, 0.65/0.71, 0.74/0.8}
{
  \draw (-1.32, \y) rectangle (-1.055, \z);
}

\foreach \y/\z in {0/0.06, 0.2/0.26, 0.29/0.35, 0.45/0.51, 0.54/0.6, 0.65/0.71, 0.74/0.8}
{
  \draw (1.32, \y) rectangle (1.055, \z);
}
\fill [blue!20] (-2, 0.2) rectangle (-1.35, 0.35);
\draw [red, style={decorate,decoration={snake,amplitude = 0.5}}] (-2, 0.275) -- (-1.35, 0.275);

\fill [blue!20] (2, 0.65) rectangle (1.35, 0.8);
\draw [red, style={decorate,decoration={snake,amplitude = 0.5}}] (2, 0.725) -- (1.35, 0.725);

\fill [blue!20] (-0.325, 0.2) rectangle (0.325, 0.35);
\draw [red, style={decorate,decoration={snake,amplitude = 0.5}}] (-0.325, 0.275) -- (0.325, 0.275);

\fill [blue!20] (0.375, 0.45) rectangle (1.025, 0.6);
\draw [red, style={decorate,decoration={snake,amplitude = 0.5}}] (0.375, 0.525) -- (1.025, 0.525);

\fill [blue!20] (-0.375, 0) rectangle (-1.025, 0.15);
\draw [red, style={decorate,decoration={snake,amplitude = 0.5}}] (-0.375, 0.075) -- (-1.025, 0.075);

\fill [blue!20] (-1.32, 0.2) rectangle (-1.055, 0.26);
\draw [red, style={decorate,decoration={snake,amplitude = 0.5}}] (-1.32, 0.23) -- (-1.055, 0.23);

\fill [blue!20] (1.32, 0.09) rectangle (1.055, 0.15);
\draw [red, style={decorate,decoration={snake,amplitude = 0.5}}] (1.32, 0.12) -- (1.055, 0.12);
\draw [orange, style={decorate,decoration={snake,amplitude = 0.3}}] (-1.4, 0.2) -- (-1.4, 0.4) -- (-1.27, 0.4) -- (-1.27, 0.2);

\draw [orange, style={decorate,decoration={snake,amplitude = 0.3}}] (-1.1, 0.26) -- (-1.1, 0.02) -- (-0.975, 0.02) -- (-0.975, 0.15);

\draw [orange, style={decorate,decoration={snake,amplitude = 0.3}}] (-0.425, 0) -- (-0.425, 0.33) -- (-0.275, 0.33) -- (-0.275, 0.2);

\draw [orange, style={decorate,decoration={snake,amplitude = 0.3}}] (0.275, 0.2) -- (0.275, 0.58) -- (0.425, 0.58) -- (0.425, 0.45);

\draw [orange, style={decorate,decoration={snake,amplitude = 0.3}}] (0.975, 0.6) -- (0.975, 0.05) -- (1.1, 0.05) -- (1.1, 0.15);

\draw [orange, style={decorate,decoration={snake,amplitude = 0.3}}] (1.27, 0.15) -- (1.27, 0.05) -- (1.4, 0.05) -- (1.4, 0.8);
\draw [blue] (-1.675, -1) -- (-1.675, -0.65);
\draw (-1.675, -0.55) -- (-1.675, -0.2);

\draw [blue] (-1.1875, -1) -- (-1.1875, -0.65);
\draw (-1.1875, -0.55) -- (-1.1875, -0.2);

\draw [blue] (-0.7, -1) -- (-0.7, -0.65);
\draw (-0.7, -0.55) -- (-0.7, -0.2);

\draw [blue] (1.675, -0.55) -- (1.675, -0.2);
\draw (1.675, -1) -- (1.675, -0.65);

\draw [blue] (1.1875, -1) -- (1.1875, -0.65);
\draw (1.1875, -0.55) -- (1.1875, -0.2);

\draw [blue] (0.7, -0.55) -- (0.7, -0.2);
\draw (0.7, -1) -- (0.7, -0.65);

\draw [blue] (0, -1) -- (0, -0.65);
\draw (0, -0.55) -- (0, -0.2);

\foreach \y/\z in {-2/-1.35, -1.32/-1.055, -1.025/-0.375, -0.325/0.325, 0.375/1.025, 1.055/1.32, 1.35/2}
{ \draw (\y, -1) -- (\z, -1);
  \draw (\y, -0.2) -- (\z, -0.2);
}
\end{tikzpicture}
\caption{{\bf Skeleton of $\cross^{*, 200am_\Gamma + 1}_2$ (above) and its $200am_\Gamma$-coarsening (below).} The rectangles in the figure at the top are the descendants of $\tilde V_{200am_\Gamma + 1; 2}^\Gamma$ in $\widetilde \descend_{\ell}$. $\cross^{*, 200am_\Gamma + 1}_2$ is indicated by the red and orange lines. The filled rectangles define its skeleton $S$. The red line inside a rectangle $R$ in $S$ is the crossing $\cross^{*, R}$. The orange lines correspond to gadgets $g_{b', b'+1}$'s used for connecting the crossings through two successive rectangles in $S$. The blue vertical segments in the figure at the bottom indicates the $200am_\Gamma$-coarsening for $\cross^{*, 200am_\Gamma + 1}_2$. The horizontal segments which they stand on correspond to intervals in $\mathscr C_{a 200m_\Gamma + 1, \Gamma, 0; 2}$. }
\label{fig:skeleton}
\end{figure}
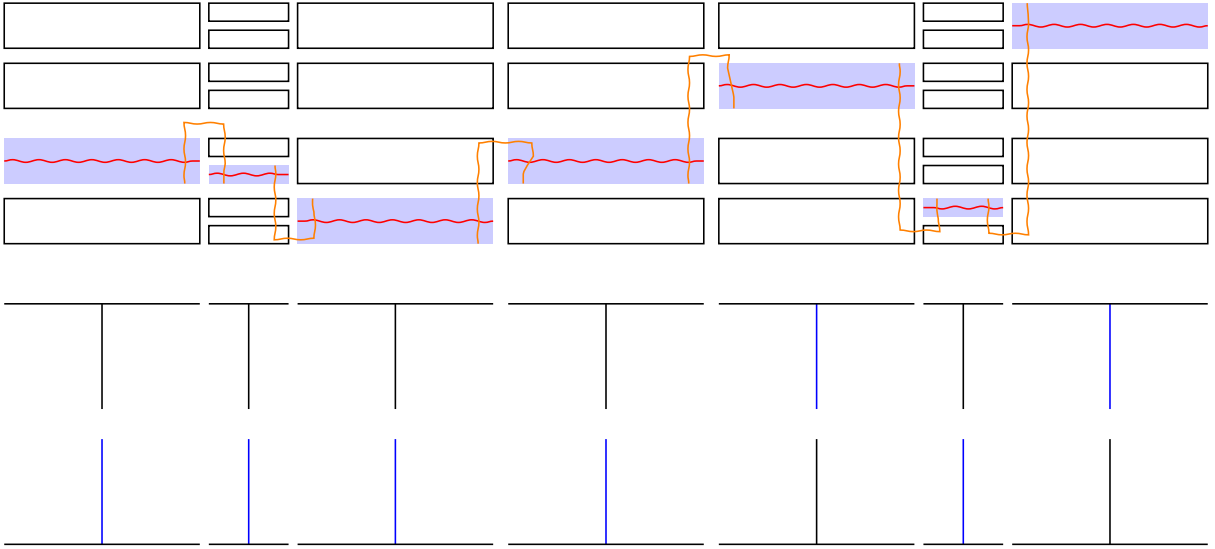

We are now ready to describe our algorithm. Since our construction of crossings at different levels are  interrelated, it is more convenient to start from level 
$200am_\Gamma$ instead. 
For descendants of $\tilde{V}_\ell^\Gamma$ in first $m_\Gamma + 100m$ levels (i.e., between the levels $(200am_\Gamma$ and $(200a+1)m_\Gamma + 100m - 1)$) we employ a very simple strategy, say 
\emph{Strategy I}, as outlined below. For convenience we only describe the construction of $\cross^{*, \ell'}_2$.\\
First we define an oriented sequence $S_{2, \ell'}$ which is a skeleton for the crossing $\cross^{*, 
\ell'}_2$. For any $I \in \mathscr C_{\ell', \Gamma, 0; \ell' \% 200m_\Gamma + 1}$, let $d_I$ be the integer such that the spans of rectangles in $\widetilde \descend_\ell$ based on $I$ form the collection $\mathscr C_{\ell' + 1, 1, 0; d_I, \principal}$ (see the 
discussion in the previous paragraph). Let $I_{\mathrm{left}, \ell'}$ be the leftmost interval in $\mathscr C_{\ell', \Gamma, 0; \ell' \% 200m_\Gamma + 
1}$. Select an interval $J_{I_{\mathrm{left}, \ell'}, \ell', 2}$ uniformly from $\mathscr C_{\ell', 1, 0; \ell' \% 200m_\Gamma + 1, \principal}$ (notice that $d_{I_{\mathrm{left}, \ell'}} = \ell' 
\% 200m_\Gamma + 1$). Now starting from the interval that is immediately to the right of $I_{\mathrm{left}, \ell'}$ we select, for each $I \in \mathscr C_{\ell', \Gamma, 0; \ell' \% 200m_\Gamma + 1}$, an interval $J_{I, \ell', 2}$ in $\mathscr C_{\ell'+1, 1, 0; d_I, 
\principal}$ recursively as follows. Let $I'$ be the last interval to the left of $I$ such that $d_{I'} \leq 
d_I$ (or in other words $|I'| \geq |I|$). We select $J_{I, \ell', 2}$ uniformly from all the intervals in $\mathscr C_{\ell'+1, 1, 0; d_I, \principal}$ that are 
also subsets of $J_{I', \ell', 2}$. To be precise, this selection is independent of all the previous selections. 
We now define $S_{2, \ell'}$ as the sequence of rectangles $I \times J_{I, \ell', 2}$ ordered according to the natural alignment of $I$'s in $\mathscr C_{\ell', \Gamma, 0; \ell' \% 200m_\Gamma + 
1}$. Finally we join the crossings through successive rectangles in $S_{2, \ell'}$ using appropriate gadgets (to be made precise later) to construct $\cross_2^{*, \ell'}$. See Figure~\ref{fig:strategy1} for an illustration.


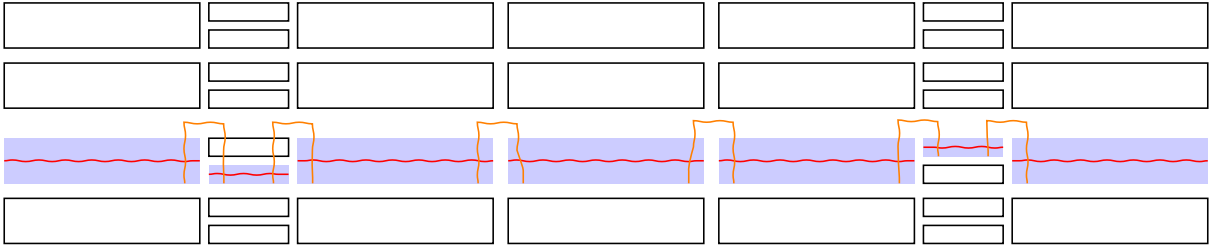
\begin{figure}[!htb]
\centering
\begin{tikzpicture}[semithick, scale = 4]

\foreach \y/\z in {0/0.15, 0.45/0.6, 0.65/0.8}
{
  \draw (-2, \y) rectangle (-1.35, \z);
  \draw (2, \y) rectangle (1.35, \z);
  \draw (-0.325, \y) rectangle (0.325, \z);
  \draw (0.375, \y) rectangle (1.025, \z);
  \draw (-0.375, \y) rectangle (-1.025, \z);
  }


\foreach \y/\z in {0/0.06, 0.09/0.15, 0.29/0.35, 0.45/0.51, 0.54/0.6, 0.65/0.71, 0.74/0.8}
{
  \draw (-1.32, \y) rectangle (-1.055, \z);
}

\foreach \y/\z in {0/0.06, 0.09/0.15, 0.2/0.26, 0.45/0.51, 0.54/0.6, 0.65/0.71, 0.74/0.8}
{
  \draw (1.32, \y) rectangle (1.055, \z);
}
\fill [blue!20] (-2, 0.2) rectangle (-1.35, 0.35);
\fill [blue!20] (2, 0.2) rectangle (1.35, 0.35);
\fill [blue!20] (-0.325, 0.2) rectangle (0.325, 0.35);
\fill [blue!20] (0.375, 0.2) rectangle (1.025, 0.35);
\fill [blue!20] (-0.375, 0.2) rectangle (-1.025, 0.35);

\fill [blue!20] (-1.32, 0.2) rectangle (-1.055, 0.26);
\fill [blue!20] (1.32, 0.29) rectangle (1.055, 0.35);
\draw [red, style = {decorate,decoration={snake,amplitude = 0.3}}] (-2, 0.275) -- (-1.35, 0.275);

\draw [red, style = {decorate,decoration={snake,amplitude = 0.3}}] (2, 0.275) -- (1.35, 0.275);

\draw [red, style = {decorate,decoration={snake,amplitude = 0.3}}] (-0.325, 0.275) -- (0.325, 0.275);

\draw [red, style = {decorate,decoration={snake,amplitude = 0.3}}] (0.375, 0.275) -- (1.025, 0.275);

\draw [red, style = {decorate,decoration={snake,amplitude = 0.3}}] (-0.375, 0.275) -- (-1.025, 0.275);

\draw [red, style = {decorate,decoration={snake,amplitude = 0.3}}] (-1.32, 0.23) -- (-1.055, 0.23);

\draw [red, style = {decorate,decoration={snake,amplitude = 0.3}}] (1.32, 0.32) -- (1.055, 0.32);
\draw [orange, style = {decorate,decoration={snake,amplitude = 0.3}}] (-1.4, 0.2) -- (-1.4, 0.37) -- (-1.27, 0.37) -- (-1.27, 0.2);

\draw [orange, style = {decorate,decoration={snake,amplitude = 0.3}}] (-1.105, 0.2) -- (-1.105, 0.37) -- (-0.975, 0.37) -- (-0.975, 0.2);

\draw [orange, style = {decorate,decoration={snake,amplitude = 0.3}}] (-0.425, 0.2) -- (-0.425, 0.37) -- (-0.275, 0.37) -- (-0.275, 0.2);

\draw [orange, style = {decorate,decoration={snake,amplitude = 0.3}}] (0.425, 0.2) -- (0.425, 0.37) -- (0.275, 0.37) -- (0.275, 0.2);

\draw [orange, style = {decorate,decoration={snake,amplitude = 0.3}}] (1.105, 0.29) -- (1.105, 0.37) -- (0.975, 0.37) -- (0.975, 0.2);

\draw [orange, style = {decorate,decoration={snake,amplitude = 0.3}}] (1.4, 0.2) -- (1.4, 0.37) -- (1.27, 0.37) -- (1.27, 0.29);
\end{tikzpicture}
\caption{{\bf Construction of $\cross^{*, 200am_\Gamma + 1}_2$ using Strategy I.} The filled rectangles define the skeleton $S_{2, 200am_\Gamma + 1}$ and $\cross^{*, 200am_\Gamma + 1}_2$ is indicated by the red and orange lines. The red and orange lines have the same meaning as in Figure~\ref{fig:skeleton}.}
\label{fig:strategy1}
\end{figure}
If $\ell \leq (200a + 1)m_\Gamma + 100m - 1$, 
then we are done. Otherwise we use \emph{Strategy~II} from level $(200a + 
1)m_\Gamma + 100m$ upwards. This strategy is the one that employs switchings to construct efficient 
crossings. We elaborate the construction for $\cross_2^{*, \ell'}$ in what follows.\\
Fix $\beta = \delta^{-40}$. For $j \in [2]$, partition the base of $\tilde V_{\ell'; 2, j, 1}^\Gamma$ (or equivalently $\tilde V_{\ell'; 2, j, 2}^\Gamma$) into intervals of length $\lfloor \beta a_{\ell' - 1} \rfloor + 1$ with the last interval of possibly smaller 
length. We will make small adjustments to the endpoints of these intervals to so that they ``almost'' coincide with the endpoints of intervals in $\mathscr C_{\ell, 
\Gamma, 0; \ell \% 200m_\Gamma + 1}$. There can be several ways to do this but we adopt a particular 
convention. For any two adjacent intervals we move the right endpoint of the left interval to the nearest right endpoint of an interval in $\mathscr C_{\ell, \Gamma, 0; 
\ell\%200m_\Gamma + 1}$. The left endpoint of the 
right interval is then moved accordingly. Thus we have a sequence of intervals $\{I_{\ell',j, j'}\}_{j' \in 
[\Gamma_{\ell', \beta}]}$ from left to right. Here $\Gamma_{\ell', \beta}$ is the total number of intervals 
which depends solely on $\Gamma, \beta$ and $\ell'$. 
We say that an interval $I \in \mathscr C_{\ell, \Gamma, 0; \ell\%200m_\Gamma + 1}$ \emph{overlaps} with some $I_{\ell, j, j'}$ if the right endpoint of $I$ lies in 
$I_{\ell, j, j'}$. Notice that $I$ can overlap with at 
most one $I_{\ell, j, j'}$. These new intervals correspond to our potential switching locations at level 
$\ell'$. Let us describe it in a formal way. Denote by $\tilde V^\Gamma_{\ell'; 2, j, j', k}$ the sub-rectangle of $\tilde V^\Gamma_{\ell'; 2, j, k}$ based on 
$I_{\ell', j, j'}$ ($k \in [2]$). For each $(j, j')$,  select (by a certain rule to be specified) a number $k_{\ell'; 2, j, j'} \in 
[2]$. These numbers will determine the $\ell' - 1$-coarsening of $\cross^{*, \ell'}_2$. Also select a rectangle $\tilde V_{\ell' - 1; \midd}^\Gamma$ from $\tilde V_{\ell'; \midd, 2}$ and some $k_{\midd, 2}$ 
from $[2]$. Now consider crossings $\cross^{*, \tilde V_{\ell'; 2, j}}_k$ for $j, k \in [2]$, and also $\cross^{*, \tilde V_{\ell' - 1; 
\midd}^\Gamma}_{k_{\midd, 2}}$. Let $S_{2, \ell', j, k}$ and $S_{2, \ell', \midd}$ denote \emph{the} skeletons (we will indeed keep track of a specific one) of the crossings $\cross^{*, \tilde V_{\ell'; 2, j}}_k$'s and $\cross^{*, \tilde V_{\ell' - 1; \midd}^\Gamma}_{k_{\midd, 2}}$ 
respectively. Also denote by $S_{2, \ell', j, j', k}$ the subsequence of $S_{2, \ell', j, j'}$ consisting of rectangles which are based on intervals overlapping with 
$I_{\ell', j, j'}$. 
Now define the oriented sequence $S_{2, \ell'}$ to be the one obtained by concatenating the subsequences $S_{2, \ell', j, j', k_{\ell'; 2, j, j'}}$'s and $S_{2, \ell', \midd}$ in appropriate order. 
From the definition of skeleton we know that there are gadgets that join the crossings through any two successive rectangles 
in each $S_{2, \ell', j, j', k_{\ell'; 2, j, j'}}$. But if $k_{\ell'; 2, j, j'} \neq k_{\ell'; 2, j, j'+ 1}$ for some $j'$ (i.e., we are making a switch at $I_{\ell', j, j'}$), we need to connect the rightmost and leftmost rectangles in $S_{2, \ell', j, j', k_{\ell'; 2, j, j'}}$ and $S_{2, \ell', j, j'+1, k_{\ell'; 2, j, j'+1}}$ respectively 
with some new gadgets. Let $R_{j, j', \mathrm{left}}$ and $R_{j, j', \mathrm{right}}$ denote these two rectangles. Also let $R_{j, j', \mathrm{right}; -1}$ be the rectangle that appears immediately before $R_{j, j', \mathrm{right}}$ in 
$S_{2, \ell', j, k_{\ell'; 2, j, j'+1}}$. Then there is a gadget $g_{j', j' + 1}$ joining the crossing through $R_{j, j', \mathrm{right}}$ with the crossing through $R_{j, j', \mathrm{right}; -1}$.
We now build a gadget $g_{\ell', 2, j; j', j'+1}$ so that the union of $g_{\ell', 2, j; j', j'+1}$ and $g_{j', j' + 1}$ forms a connected (multi) set which joins the crossings through $R_{j, j', \mathrm{left}}$ and $R_{j, 
j', \mathrm{right}}$. We refer to this procedure as 
\emph{gluing the junction at $I_{\ell, j, j'}$}. Finally we glue the junctions at the two ends of $\tilde V_{\ell' - 1; \midd}^\Gamma$ and get $\cross^{*, 
\ell'}_2$. See Figure~\ref{fig:strategy2} for an illustration.
\begin{figure}[!htb]
\centering
\begin{tikzpicture}[semithick, scale = 4]
\foreach \y/\z in {0/0.06, 0.09/0.15, 0.29/0.35, 0.45/0.51, 0.54/0.6, 0.74/0.8}
{\draw (-2, \y) rectangle (-1.735, \z); 
}
\foreach \y/\z in {0/0.06, 0.09/0.15, 0.29/0.35, 0.45/0.51, 0.54/0.6, 0.74/0.8}
{
 \draw (-1.615, \y) rectangle (-1.35, \z);
 
}
\foreach \y/\z in {0/0.06, 0.09/0.15, 0.29/0.35, 0.45/0.51, 0.54/0.6, 0.74/0.8}
{
 \draw (-1.32, \y) rectangle (-1.055, \z);

}

\foreach \y/\z in {0.09/0.15, 0.2/0.26, 0.29/0.35, 0.45/0.51, 0.54/0.6, 0.74/0.8}
{
 \draw (-0.760, \y) rectangle (-1.025, \z);
 
}

\foreach \y/\z in {0/0.06, 0.2/0.26, 0.29/0.35, 0.45/0.51, 0.65/0.71, 0.74/0.8}
{
 \draw (-0.375, \y) rectangle (-0.640, \z);
 
}

\foreach \y/\z in {0/0.06, 0.09/0.15, 0.2/0.26, 0.29/0.35, 0.54/0.6, 0.65/0.71, 0.74/0.8}
{
 \draw (-0.325, \y) rectangle (-0.06, \z);
}

\foreach \y/\z in {0/0.06, 0.09/0.15, 0.2/0.26, 0.29/0.35, 0.45/0.51, 0.65/0.71, 0.74/0.8}
{
 \draw (0.06, \y) rectangle (0.325, \z);
}

\foreach \y/\z in {0/0.06, 0.2/0.26, 0.29/0.35, 0.45/0.51, 0.54/0.6, 0.74/0.8}
{
 \draw (0.375, \y) rectangle (0.640, \z);
 
}

\foreach \y/\z in {0.09/0.15, 0.2/0.26, 0.29/0.35, 0.45/0.51, 0.65/0.71, 0.74/0.8}
{
 \draw (0.760, \y) rectangle (1.025, \z);
 
}

\foreach \y/\z in {0/0.06, 0.09/0.15, 0.2/0.26, 0.45/0.51, 0.65/0.71, 0.74/0.8}
{
 \draw (1.32, \y) rectangle (1.055, \z);
}

\foreach \y/\z in {0/0.06, 0.2/0.26, 0.29/0.35, 0.45/0.51, 0.54/0.6, 0.74/0.8}
{
 \draw (1.615, \y) rectangle (1.35, \z);
 
}

\foreach \y/\z in {0.09/0.15, 0.2/0.26, 0.29/0.35, 0.45/0.51, 0.65/0.71, 0.74/0.8}
{
 \draw (2, \y) rectangle (1.735, \z); 
 
}

\foreach \y / \z in {0/0.024, 0.036/0.06, 0.09/0.114, 0.2/0.224, 0.236/0.26, 0.29/0.314, 0.326/0.35, 0.45/0.474, 0.486/0.51, 0.576/0.6, 0.65/0.674, 0.686/0.71, 0.74/0.764, 0.776/0.8}
{ \draw (-1.725, \y) rectangle (-1.625, \z); 

}

\foreach \y / \z in {0/0.024, 0.036/0.06, 0.09/0.114, 0.2/0.224, 0.236/0.26, 0.29/0.314, 0.326/0.35, 0.45/0.474, 0.486/0.51, 0.576/0.6, 0.65/0.674, 0.686/0.71, 0.74/0.764, 0.776/0.8}
{ 
  \draw (-0.75, \y) rectangle (-0.65, \z);
}

\foreach \y / \z in {0/0.024, 0.036/0.06, 0.09/0.114, 0.2/0.224, 0.236/0.26, 0.29/0.314, 0.326/0.35, 0.45/0.474, 0.486/0.51,  0.576/0.6, 0.65/0.674, 0.686/0.71, 0.74/0.764, 0.776/0.8}
{ 
  \draw (0.65, \y) rectangle (0.75, \z); 
  
}

\foreach \y / \z in {0/0.024, 0.036/0.06, 0.126/0.15, 0.2/0.224, 0.236/0.26, 0.29/0.314, 0.326/0.35, 0.45/0.474, 0.486/0.51, 0.54/0.564, 0.65/0.674, 0.686/0.71, 0.74/0.764, 0.776/0.8}
{
  \draw (1.725, \y) rectangle (1.625, \z); 
 
}

\foreach \y / \z in {0/0.024, 0.036/0.06, 0.09/0.114, 0.126/0.15, 0.2/0.224, 0.236/0.26, 0.29/0.314, 0.326/0.35, 0.45/0.474, 0.486/0.51, 0.54/0.564, 0.65/0.674, 0.686/0.71, 0.74/0.764, 0.776/0.8}
{ 
  \draw (-0.05, \y) rectangle (0.05, \z); 
 
}
\fill [blue] (-2, 0.65) rectangle (-1.735, 0.71); 
\draw [red] (-2, 0.68) -- (-1.735, 0.68); 

\fill [blue!20] (-2, 0.2) rectangle (-1.735, 0.26); 

\draw [red] (-1.785, 0.65) -- (-1.785, 0.73) -- (-1.71, 0.73) -- (-1.71, 0.54); 

\fill [blue] (-1.725, 0.54)  rectangle (-1.625, 0.564);
\draw [red] (-1.725, 0.552)  -- (-1.625, 0.552);

\fill [blue] (-1.615, 0.65) rectangle (-1.35, 0.71);
\draw [red] (-1.615, 0.68) -- (-1.35, 0.68); 

\draw [red] (-1.640, 0.54) -- (-1.640, 0.73) -- (-1.565, 0.73) -- (-1.565, 0.65);

\fill [blue!20] (-1.615, 0.2) rectangle (-1.35, 0.26);

\fill [blue!20] (-1.32, 0.65) rectangle (-1.055, 0.71);

\fill [blue] (-1.32, 0.2) rectangle (-1.055, 0.26);
\draw [red] (-1.32, 0.23) -- (-1.055, 0.23); 

\draw [orange, style = {decorate,decoration={snake,amplitude = 0.2}}] (-1.4, 0.71) -- (-1.4, 0.18) -- (-1.27, 0.18) -- (-1.27, 0.26);

\fill [blue!20] (-1.025, 0.65) rectangle (-0.760, 0.71);

\fill [blue] (-1.025, 0) rectangle (-0.760, 0.06);
\draw [red] (-1.025, 0.03) -- (-0.760, 0.03); 

\draw [red] (-1.105, 0.2) -- (-1.105, 0.275) -- (-0.975, 0.275) -- (-0.975, 0);

\fill [blue] (-0.75, 0.126) rectangle (-0.65, 0.15);
\draw [red] (-0.75, 0.1358) -- (-0.65, 0.1358);

\draw [red] (-0.81, 0) -- (-0.81, 0.18) -- (-0.735, 0.18) -- (-0.735, 0.126);

\fill [blue] (-0.640, 0.54) rectangle (-0.375, 0.6);
\draw [red] (-0.640, 0.57) -- (-0.375, 0.57);

\draw [orange, style = {decorate,decoration={snake,amplitude = 0.2}}] (-0.665, 0.126) -- (-0.665, 0.63) -- (-0.59, 0.63) -- (-0.59, 0.54);

\fill [blue!20] (-0.640, 0.09) rectangle (-0.375, 0.15);

\fill [green!20] (-0.325, 0.45) rectangle (-0.06, 0.51);
\draw [red] (-0.325, 0.48) -- (-0.06, 0.48);

\draw [orange, style = {decorate,decoration={snake,amplitude = 0.2}}] (-0.425, 0.54) -- (-0.425, 0.62) -- (-0.275, 0.62) -- (-0.275, 0.45); 

\fill [green!20] (-0.05, 0.576) rectangle (0.05, 0.6); 
\draw [red] (-0.05, 0.588) rectangle (0.05, 0.588);

\draw [red] (-0.11, 0.45) -- (-0.11, 0.62) -- (-0.035, 0.62) -- (-0.035, 0.576);

\fill [green!20] (0.06, 0.54) rectangle (0.325, 0.6);
\draw [red] (0.06, 0.57) -- (0.325, 0.57);

\draw [red] (0.035, 0.576) -- (0.035, 0.62) -- (0.11, 0.62) -- (0.11, 0.54);


\fill [blue] (0.375, 0.65) rectangle (0.640, 0.71);
\draw [red] (0.375, 0.68) -- (0.640, 0.68);

\draw [orange, style = {decorate,decoration={snake,amplitude = 0.2}}] (0.275, 0.54) -- (0.275, 0.7) -- (0.425, 0.7) -- (0.425, 0.65); 

\fill [blue!20] (0.375, 0.09) rectangle (0.640, 0.15);

\fill [blue] (0.65, 0.54) rectangle (0.75, 0.564); 
\draw [red] (0.65, 0.552) rectangle (0.75, 0.552); 

\fill [blue] (0.760, 0.54) rectangle (1.025, 0.6);
\draw [red] (0.760, 0.57) -- (1.025, 0.57);

\draw [red] (0.59, 0.65) -- (0.59, 0.73) -- (0.665, 0.73) -- (0.665, 0.54);
\draw [red] (0.735, 0.54) -- (0.735, 0.62) -- (0.81, 0.62) -- (0.81, 0.54);

\fill [blue!20] (0.760, 0)    rectangle (1.025, 0.06);

\fill [blue!20] (1.055, 0.54) rectangle (1.32, 0.6);

\fill [blue] (1.055, 0.29) rectangle (1.32, 0.35);
\draw [red] (1.055, 0.32) -- (1.32, 0.32);

\draw [orange, style = {decorate,decoration={snake,amplitude = 0.2}}] (0.975, 0.6) -- (0.975, 0.01) -- (1.105, 0.01) -- (1.105, 0.35);


\fill [blue!20] (1.615, 0.65) rectangle (1.35, 0.71);

\fill [blue] (1.615, 0.09) rectangle (1.35, 0.15);
\draw [red] (1.615, 0.12) -- (1.35, 0.12);

\draw [red] (1.27, 0.35) -- (1.27, 0.07) -- (1.4, 0.07) -- (1.4, 0.15);

\fill [blue] (1.725, 0.09) rectangle (1.625, 0.114); 
\draw [red] (1.725, 0.102)  -- (1.625, 0.102); 
\draw [red] (1.565, 0.09) -- (1.565, 0.17) -- (1.645, 0.17) -- (1.645, 0.09);

\fill [blue] (2, 0.54)     rectangle (1.735, 0.6); 
\draw [red] (2, 0.57)   -- (1.735, 0.57);

\fill [blue!20] (2, 0)     rectangle (1.735, 0.06);

\draw [orange, style = {decorate,decoration={snake,amplitude = 0.2}}] (1.71, 0.09) -- (1.71, 0.62) -- (1.785, 0.62) -- (1.785, 0.54);

\fill [blue!20] (-1.725, 0.126) rectangle (-1.625, 0.15); 

\fill [blue!20] (-0.75, 0.54) rectangle (-0.65, 0.564);

\fill [blue!20] (0.65, 0.126) rectangle (0.75, 0.15);

\fill [blue!20] (1.725, 0.576)   rectangle (1.625, 0.6);







\draw [dashed] (-2.03, 0.42) rectangle (-0.355, 0.83);
\draw [dashed] (2.03, 0.42) rectangle (0.355, 0.83);
\draw [dashed] (2.03, -0.03) rectangle (0.355, 0.38);
\draw [dashed] (-2.03, -0.03) rectangle (-0.355, 0.38);
\draw [<->] (2.375 -2, -0.06) -- (2.375 - 1.67, -0.06);
\draw [<->] (2.375 - 1.65, -0.06) -- (2.375 - 1.33, -0.06);
\draw [<->] (2.375 - 1.31, -0.06) -- (2.375 - 0.99, -0.06);	
\draw [<->] (2.375 - 0.97, -0.06) -- (2.375 - 0.65, -0.06);
\draw [<->] (2.375 - 0.64, -0.06) -- (2.375 - 0.38, -0.06);

\draw [<->] (-2, -0.06) -- (-1.67, -0.06);
\draw [<->] (-1.65, -0.06) -- (-1.33, -0.06);
\draw [<->] (-1.31, -0.06) -- (-0.99, -0.06);	
\draw [<->] (-0.97, -0.06) -- (-0.65, -0.06);
\draw [<->] (-0.64, -0.06) -- (-0.38, -0.06);
\draw [<->] (-2, -0.11) -- (-1.625, -0.11);
\draw [<->] (-1.605, -0.11) -- (-1.35, -0.11);
\draw [<->] (-1.31, -0.11) -- (-1.045, -0.11);
\draw [<->] (-1.015, -0.11) -- (-0.65, -0.11);
\draw [<->] (-0.64, -0.11) -- (-0.38, -0.11);

\draw [<->] (2.375 - 2, -0.11) -- (2.375 - 1.625, -0.11);
\draw [<->] (2.375 - 1.605, -0.11) -- (2.375 - 1.35, -0.11);
\draw [<->] (2.375 - 1.31, -0.11) -- (2.375 - 1.045, -0.11);
\draw [<->] (2.375 - 1.015, -0.11) -- (2.375 - 0.65, -0.11);
\draw [<->] (2.375 - 0.64, -0.11) -- (2.375 - 0.38, -0.11);

\node [scale = 0.7, below]  at (-1.8125, -0.11) {$I_{\ell', 1, 1}$};

\node [scale = 0.7, below]  at (-1.4775, -0.11) {$I_{\ell', 1, 2}$};

\node [scale = 0.7, below]  at (-1.1775, -0.11) {$I_{\ell', 1, 3}$};

\node [scale = 0.7, below]  at (-0.8325, -0.11) {$I_{\ell', 1, 4}$};

\node [scale = 0.7, below]  at (-0.51, -0.11) {$I_{\ell', 1, 5}$};

\node [scale = 0.7, below]  at (2.375 - 1.8125, -0.11) {$I_{\ell', 2, 1}$};

\node [scale = 0.7, below]  at (2.375 - 1.4775, -0.11) {$I_{\ell', 2, 2}$};

\node [scale = 0.7, below]  at (2.375 - 1.1775, -0.11) {$I_{\ell', 2, 3}$};

\node [scale = 0.7, below]  at (2.375 - 0.8325, -0.11) {$I_{\ell', 2, 4}$};

\node [scale = 0.7, below]  at (2.375 - 0.51, -0.11) {$I_{\ell', 2, 5}$};






\end{tikzpicture}
\caption{{\bf Construction of $\cross^{*, \ell'}_2$ using Strategy II.} The crossing is indicated by red and orange lines. The rectangles without broken boundary lines are the elements of $\widetilde \descend_{\ell'}$ which are subsets of $\tilde V_{\ell'; 2}$. The 4 rectangles with broken boundary lines are only meant to demarcate the elements of $\widetilde \descend_{\ell'}$ that are contained in $\tilde V_{\ell'; 2, j, k}$ for some $j, k \in [2]$. The first set of arrows (i.e. the ones nearest to the broken lines) indicate the original partition of the base of $\tilde V_{\ell'; 2, j, k}$ whereas the second set of arrows indicate $I_{\ell', j, j'}$'s. Hence $\Gamma_{\ell', \beta} = 5$ in this figure. Rectangles filled with blue (both light and deep) define the skeletons $S_{2, \ell', j, k}$'s for $j, k \in [2]$ and the rectangles filled with green define the skeleton $S_{2, \ell', \midd}$. $S_{2, \ell'}$ consists of the rectangles filled with deep blue. Thus, in this figure, $k_{\ell';2, j, j'} = 1$ for $j' \in \{1, 2, 5\}$ and $ = 2$ for $j' \in \{3, 4\}$ for all $j \in [2]$. The orange lines indicate the gadgets used for gluing the junctions at switching locations. All other gadgets are inherited from $\cross_k^{*, \tilde V_{\ell'; 2, j}}$'s and $\cross_{k_{\midd, 2}}^{*, \tilde V_{\ell' - 1; \midd}}$.}
\label{fig:strategy2}
\end{figure}
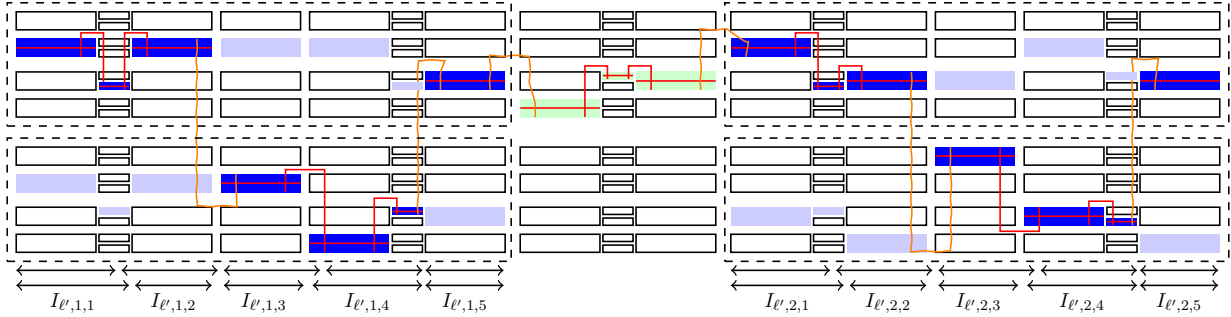

It would be useful to have some convenient notations for the weight of crossings as well as their expected 
values. To this end define for any rectangle $B \in \B_\ell$,
$$D_{\gamma, B, i} = \sum_{v \in \cross^{*, B}_i}m_{B,i}(v)\e^{\gamma \eta_{n, B, v}} \mbox{ and } d_{\gamma, B} = \E  D_{\gamma, B, i}\,,$$ 
where $m_{B, i}(v)$ is the multiplicity of $v$ in the multiset 
$\cross^{*, B}_i$ and $\eta_{n, B, v}$ is the fine field on $B$ (see the end of Subsection~\ref{subsec:GFF_represent}). 
Let $D_{\gamma, B, i, \join}$ denote the difference between $D_{\gamma, B, i}$ and the total weight of crossings (with respect to $\{\eta_{n, B, .}\}$) through the rectangles in the skeleton of $\cross^{*, B}_i$. Thus $D_{\gamma, B, i, \join}$ is the total weight of gadgets that we have used to build $\cross^{*, 
B}_i$ from its skeleton. We denote its expectation by 
$d_{\gamma, B, \join}$. When the underlying rectangle is $\tilde{V}_\ell^\Gamma$, we replace ``$B$'' with ``$\ell$'' in these notations. Finally for a subset $S$ of $B_i$ define,
\begin{equation}\label{eq-def-d}
D_{\gamma, B_i, S} = \sum_{v \in \cross^{*, B}_i \cap S}m_{B,i}(v)\e^{\gamma \eta_{n, B, v}} \mbox{ and } d_{\gamma, B_i, S} = \E  D_{\gamma, B_i, S}\,.
\end{equation}

\subsection{The induction hypotheses}
\label{subsec:induct_hypo}
We need several induction hypotheses to carry out the 
proof of Theorem~\ref{theo:main} which we state in this 
subsection. In Subsection~\ref{subsec:descrip}, we discussed a general scheme to construct crossings through $B \in \B_\ell$ using the fine field on $B$ and 
some ``extra'' random variables. Before we formulate our hypothesis, it is necessary to have a precise 
description of these random variables. To this end, let us introduce a collection of random variables $\Xi_{B, i}$ for each $\ell' \leq n, B \in \B_{\ell'}$ and $i \in 
[5]$. The variables in $\cup_{B \in \B_\ell, \ell \leq n, i \in [5]}\Xi_{B, i}$ are independent of each 
other. They are also independent of the field 
$\{\eta_{n, v}\}_{v \in \tilde V_n^\Gamma}$. Further the members of $\Xi_{B, i}$ can be arranged in a sequence $(\xi_{B, i, 1}, \xi_{B, i, 2}, \ldots, \xi_{B, i, |\Xi_{B, j}|})$ so that these sequences are identically 
distributed for all $B \in \B_{\ell'}$. We will identify $\Xi_{B, i}$ with this sequence whenever we need to use 
$\Xi_{B, i}$ as a random vector. We do not define $\xi_{B, i, j}$'s a priori, rather they will be defined explicitly while we carry out the induction step. So at the beginning one can think of the families $\Xi_{B,i}$'s as collections of independent random 
variables. Let $B \in \B_{\ell'}$. We denote by $\Xi_B$ the union of $\Xi_{B', i}$'s for all descendants $B'$ of $B$ and $i \in [5]$. If $B = \tilde V_{\ell'}^\Gamma$, we denote these families as $\Xi_{\ell', j}, 
\Xi_{\ell'}$ etc. Now we are ready to state the induction hypotheses. At the beginning of step $\ell$ we assume that each of these hypotheses holds for all 
$0 \leq \ell' < \ell$ (unless otherwise mentioned) and all $B \in \B_{\ell'; \principal}$. To avoid complications we phrase them only for $B = \tilde V_{\ell'}^\Gamma$.

Our first hypothesis asserts that the construction of $\cross_1^{*, \ell'}$ and $\cross_2^{*, \ell}$ are ``internal'' to $\tilde V_{\ell'}^\Gamma$:
\vspace{-0.2in}
\begin{hypothesis}
\label{hypo:measurabilty}
\mbox{$\cross_1^{*, \ell'}$ and $\cross_2^{*, \ell'}$ are determined by the random variables $\{\eta_{n, \ell', v}\}_{v \in \tilde V_{\ell'}^\Gamma}$ and $\Xi_{\ell'}$}\,.
\end{hypothesis}
Let $\overline \cross_1^{*, \ell'}$ be the image of $\cross_1^{*, \ell'}$ with respect to the reflection of 
$\tilde V_{\ell'}^\Gamma$. Then $\overline \cross_1^{*, \ell'}$ is a crossing through $\tilde V_{\ell'; 
2}^\Gamma$. The following hypothesis states that $\overline \cross_1^{*, \ell'}$, $\cross_2^{*, \ell'}$ are identically distributed as and so are their weights. Formally,
\begin{hypothesis}
\label{hypo:identical_distribution}
\mbox{$(\overline \cross_1^{*, \ell'}, \{\eta_{n, \ell', \overline v}\}_{v \in \tilde V_{\ell'}^\Gamma})$ is identically distributed as $(\cross_2^{*, \ell'}, \{\eta_{n, \ell', v}\}_{v \in \tilde V_{\ell'}^\Gamma})$}\,,
\end{hypothesis}
where $\overline v$ is the reflected image of $v$.

Our third hypothesis, which is our main hypothesis, gives an upper bound on the increase in expected weight of crossings between two successive 
levels. One of the keys for driving this hypothesis is the limit result given in 
Lemma~\ref{lem:Brownian_scaling}. But for this lemma to be effective, we need $N$ (as in the lemma) to be sufficiently large so that the error term 
becomes sufficiently small. As a consequence our 
desired upper bounds will only be posed for large $\ell'$. In order to incorporate this, we 
denote by $a'$ the smallest integer bigger than $4$ such that the error term in Lemma~\ref{lem:Brownian_scaling} is less than $10^{-6}$ for all $N \geq \lfloor (2 + \delta)^{200a'm_\Gamma} \rfloor + 1$ whenever $\theta = \delta / (2 + 3\delta)$, $\Upsilon = \Gamma\delta / 16$ and $\mathcal I = \tfrac{\delta}{2 + 3\delta} + \tfrac{2 + \delta}{2 + 3\delta}\mathcal I_{0, 1, 0; d, \principal}$ (see Subsection~\ref{subsec:self_similar}) for some $d \in 
[m_\Gamma + 100m + 1, 200m_\Gamma + m]$. This class of subsets (of $\R$) are  limits of sets that are union of spans of all possible rectangles in $S_{i, \ell'}$ based on a particular interval $I$ in $\mathscr C_{\ell', \Gamma, 0; \ell' \% 200m_\Gamma + 1}$, scaled (after a translation) so that the span of $\tilde V_{\ell'}^\Gamma$ converges to $[0, 1]$ as 
$\ell' \to \infty$. The reason for the choice of $\Upsilon$ will be clarified in Subsection~\ref{subsec:hard-remaining}. Recalling definition of $d_{\gamma, \ell'}$ in \eqref{eq-def-d} and preceding texts, our third hypothesis states
\begin{hypothesis}
\label{hypo:expected_weight}
d_{\gamma, \ell'} \leq d_{\gamma, \ell' - 1}(2 + \delta + e_{\ell'} \gamma^2)\,.
\end{hypothesis}
where
$$e_{\ell'} = \begin{cases}
-0.045 &\mbox{if }\ell' \geq 200a'm_\Gamma\mbox{ and }\ell' \% 200m_\Gamma \geq m_\Gamma + 100m
\,,\\
1 &\mbox{if }\ell' \geq 200a'm_\Gamma\mbox{ and }\ell' \% 200m_\Gamma < m_\Gamma + 100m
\,,\\
O(1)\log(1 / \delta) &\mbox{if }1 \leq \ell' < 200a'm_\Gamma\,.
\end{cases}$$
We also need a lower bound on the increment (for convenience of analysis). For $\ell' \geq 1$,
\begin{hypothesis}
\label{hypo:short_range_ratio}
d_{\gamma, \ell'} \geq d_{\gamma, \ell' - 1}(2 + \delta - 4\delta^{-1/8}\gamma^2)\,.
\end{hypothesis}
The next  hypothesis states that the total expected weight of gadgets used at any level is negligible:
\begin{hypothesis}
\label{hypo:gadget_negligible}
d_{\gamma, \ell', \join} \leq 8\log(1 / \delta) d_{\gamma, \ell' - 1}(\ell' \% 200m_\Gamma + 1)\gamma^2 \,.
\end{hypothesis}
A crucial component of our proof is the symmetry of 
switchings we make at every step. Our next hypothesis 
gives a formal formulation. Let $i \in [2]$, $I \in \mathscr C_{\ell', \Gamma, 0; \ell' \% 200m_\Gamma}$ and $R_{I, i}^{\ell'}$ be the unique interval in the skeleton $S_{i, \ell'}$ of $\cross^{*, \ell'}_i$ that is 
based on $I$. Also denote by $\descend_{\ell', I, i}$ the collection of all possible choices for $R_{I, i}^{\ell'}$ and by $\B_{\ell', I, i}$ the collection of all rectangles in $\descend_{\ell'}$ that contain some 
member of $\descend_{\ell', I, i}$. Then,
\begin{hypothesis}
\label{hypo:symmetry}
\mbox{$R_{I, i}^{\ell'}$ is uniform in $\descend_{\ell', I, i}$ and is determined by $(\{X_{n, B, \ell'', .}\}_{B \in \B_{\ell', I, i}, \ell'' \in [\lfloor \ell' / 200m_\Gamma\rfloor, \ell']}, \Xi_{\ell'})$\,. 
}
\end{hypothesis}
\vspace{-0.05in}

While describing Strategy II, we defined the sequence $S_{i, \ell', \midd}$. This is essentially the subsequence of $S_{i, \ell'}$ whose elements (i.e. rectangles) are based on intervals contained in $\lfloor \mathcal I_{\ell - m - 1, \Gamma, \lfloor \Gamma a_{\ell 
- 1}\rfloor} \rfloor$. Hence we can define the same for 
Strategy I as well. Our next induction hypothesis states 
an important property of $S_{i, \ell', \midd}$: for $\ell'_m = (\ell' - 100m + 1) \vee (\lfloor \ell' / 200m_\Gamma\rfloor - m)$,
\begin{hypothesis}
\label{hypo:irrelevance}
\mbox{the $\ell'_m$-coarsening of $S_{i, \ell', \midd}$ lies in $\Xi_{\ell', i + 2}$}\,.
\end{hypothesis}
Finally we want a fixed bound on the total number of switches that we make at every step when we apply 
Strategy II. Thus the following induction hypothesis applies only for levels satisfying $\ell' \geq a'200m_\Gamma$ and $\ell' \% 200m_\Gamma \geq 
m_\Gamma + 100m$. For $i \in [2]$,
\begin{hypothesis}
\label{hypo:bound_n_switch}
\mbox{The number of switches made at step~$\ell'$ to construct $\cross_{i}^{*, \ell'}$ is at most $3\alpha$ 
}\,.
\end{hypothesis}
In accordance with our formulation of \eqref{hypo:expected_weight}, we will split our discussion 
of the induction step into three different cases. The base case includes all $0 \leq \ell < 200a'm_\Gamma$ and the easy case includes all $\ell \geq 200a'm_\Gamma$ 
such that $\ell \% 200m_\Gamma < m_\Gamma + 100m$. We discuss these two cases together in Subsection~\ref{subsec:easy} 
as we use Strategy~I for both. The hard case includes all $\ell \geq a'200m_\Gamma$ such that $\ell \% 
200m_\Gamma \geq m_\Gamma + 100m$. We discuss this case in Subsections~\ref{subsec:induct_step_hard}, \ref{subsec:error_terms} and 
\ref{subsec:hard-remaining}. In the ensuing analysis we repeatedly use the fact that $\delta$ is fixed but small 
and $\gamma \ll \delta$. For the sake of convenience, we do not provide explicit bounds on $\delta$ or $\gamma$ that are required for any given bound or inequality to 
hold. However these requirements should be clear at any particular context.

\subsection{Induction step for the base and the easy case} \label{subsec:easy}
We have already defined the skeleton $S_{i, \ell}$ of $\cross^{*, \ell}_i$ in a precise way when we described 
Strategy I in Subsection~\ref{subsec:descrip}. We further point out that the interval valued random variables 
and $J_{I, \ell, i}$'s (for $I \in \mathscr C_{\ell, \Gamma, 0; \ell \% 200m_\Gamma+1}$) lie in $\Xi_{\ell', i}$. In what follows, we verify the induction hypothesis in the base and easy cases.

For the base case, i.e., when $0 \leq \ell < 200a'm_\Gamma$, we use the rectangles $\tilde V_{\ell', i'}^\Gamma$'s for $-(m + 1) \leq \ell' \leq -1$ and $i' 
\in [2]$ as building blocks in $S_{i, \ell}$. Since these rectangles are essentially straight lines, we do not need any separate gadget to join the crossings through successive rectangles in the skeleton. 
Evidently , $\cross^{*, \ell}_i$ satisfies the hypotheses~\eqref{hypo:measurabilty}, \eqref{hypo:identical_distribution}, \eqref{hypo:gadget_negligible}, \eqref{hypo:symmetry} and 
\eqref{hypo:irrelevance}. Since $a_{\ell' - \ell''} (1 - \tfrac{O(\gamma^2)}{\alpha \delta}) \leq d_{\gamma, \ell'} / d_{\gamma, \ell''} \leq a_{\ell' - \ell''} (1 + \tfrac{O(\gamma^2)}{\alpha \delta})$ for all $-(m + 1) \leq \ell'' < \ell' \leq 0$, it is easy to verify that $\cross^{*, \ell}_i$ obeys \eqref{hypo:expected_weight}. 
This is because the coarse field variance is at most 
$O(\log(1 / \delta))$ at each level by 
Lemma~\ref{lem:free_field_var}. Finally \eqref{hypo:short_range_ratio} follows from the induction hypothesis \eqref{hypo:expected_weight}.

For the easy case, i.e., when $\ell \geq 200a'm_\Gamma$ and $\ell\% 200m_\Gamma < m_\Gamma + 100m$, we need to define some gadgets in order to join the crossings through successive 
rectangles in $S_{i, \ell}$. To this end consider the crossing $\cross^{*, I, \ell, i}$ through $I \times J_{I, \ell', i}$ which is a rectangle in $\widetilde 
\descend_\ell$. We can link the crossings $\cross^{*, I, \ell, i}$'s in a simple way which we call \emph{sewing} 
for convenience. We describe this technique in a general setting as we will use it 
several times. The reader is referred to Fig~\ref{fig:gadget_sew} for an illustration.

Consider two adjacent intervals $I_1$ and $I_2$ in $\mathscr C_{n, \Gamma, 0; r}$ where $n - r$ is big 
enough so that $a_{n - r} \geq \delta$. From the description of $\mathscr C_{n, \Gamma, 0; r}$, we know that $|I_j| = \lfloor \Gamma a_{m_j} \rfloor + 1$ for some $m_j \in \{n - r - m, \ldots, n - r\}$ (here $j \in 
[2]$). Suppose, without loss of generality, that $I_1$ 
is longer (or of equal length) than $I_2$. Let $I_2' \subseteq I_1'$ be intervals of lengths $\lfloor a_{m_2} \rfloor + 1$ and $\lfloor a_{m_1} \rfloor + 1$ respectively such that the rectangles $R_{I_j} = I_j 
\times I_j'$'s are contained in $\lfloor \mathcal I_{n, \Gamma, 0} \rfloor \times \lfloor \mathcal I_{n + 1, 1, 
0} \rfloor$. Also define three additional rectangles $R_{1, 2; 1} = \lfloor [r_{I_1} - \lfloor a_{m_1 - m_\Gamma}\rfloor , r_{I_1}]\rfloor \times I_1$, $R_{1,2;2} = \lfloor [r_{I_1}+ 1, r_{I_1} + 1 + \lfloor a_{m_2 - m_\Gamma}\rfloor]\rfloor \times I_2$, and $R_{1,2;3} = \lfloor [r_{I_1} - \lfloor a_{m_1 - m_\Gamma}\rfloor, r_{I_1} - \lfloor a_{m_1 - m_\Gamma}\rfloor + \lfloor a_{m_{1} - m_\Gamma + 2}\rfloor] \rfloor \times \lfloor [p_{I_2'}, p_{I_2'} + \lfloor a_{m_{1} - 2m_\Gamma + 
2}\rfloor]\rfloor$. If we have ``already'' built crossings at these levels, then we can construct up-down crossings $\cross^{*, R_{1, 2; 1}}$ , $\cross^{*, R_{1, 2; 2}}$ for $R_{1, 2; 1}$ and $R_{1, 2; 2}$ respectively; and a left-right crossing $\cross^{*, 
R_{1, 2; 3}}$ for $R_{1, 2; 3}$. Now let $\cross_{R_{I_1}}$ and $\cross_{R_{I_2}}$ be (left-right) crossings for $R_{I_1}$ and $R_{I_2}$ 
respectively. Notice that union (as multisets) of the crossings $\cross_{R_{I_1}}$, $\cross_{R_{I_2}}$, $\cross^{*, R_{1, 2; 1}}$, $\cross^{*, R_{1, 2; 2}}$ and $\cross^{*, R_{1, 2; 3}}$ is a crossing between 
$\partial_{\mathrm{left}}R_{I_1}$ and 
$\partial_{\mathrm{left}}R_{I_2}$. We refer to this as the crossing obtained from sewing $\cross_{R_{I_1}}$ and 
$\cross_{R_{I_2}}$. Similar construction could be done if the longer dimensions of the rectangles $R_{I_1}$ and 
$R_{I_2}$ were in vertical direction. The extra random variables that are used to build these gadgets are independent for all different gadgets and are elements of $\Xi_{\ell, 5}$.
\begin{figure}[!htb]
\centering
\begin{tikzpicture}[scale = 10]
\draw [dashed] (-0.7, -0.22) rectangle (0.7, 0.22);

\draw (-0.65, -0.05) rectangle (-0.15, 0.1);
\draw (-0.14, -0.05) rectangle (0.15, 0.04);

\draw [purple, style={decorate,decoration={snake,amplitude = 0.5}}] (-0.65, 0.05) -- (-0.15, 0.05);
\draw [purple, style={decorate,decoration={snake,amplitude = 0.5}}] (-0.14, 0) -- (0.15, 0);

\draw [blue, style={decorate,decoration={snake,amplitude = 0.5}}] (-0.17, -0.05) -- (-0.17, 0.1);
\draw [red, dashed] (-0.19, -0.05) -- (-0.19, 0.1);

\draw [blue, style={decorate,decoration={snake,amplitude = 0.5}}] (-0.125, -0.05) -- (-0.125, 0.04);
\draw [red, dashed] (-0.11, -0.05) -- (-0.11, 0.04);

\draw [blue, style={decorate,decoration={snake,amplitude = 0.5}}] (-0.19, -0.04) -- (-0.08, -0.04);
\draw [red, dashed] (-0.19, -0.05) rectangle (-0.08, -0.03);

\node [scale = 0.7] at (-0.4, 0.025) {$R_{I_1}$};
\node [scale = 0.7, below] at (0.005, -0.01) {$R_{I_2}$};

\end{tikzpicture}
\caption{{\bf Sewing $\cross_{R_{I_1}}$ and $\cross_{R_{I_2}}$.} The crossings $\cross_{R_{I_1}}$ and $\cross_{R_{I_2}}$ are indicated by purple lines. The two vertical blue lines indicate the crossings $\cross^{*, R_{1, 2; 1}}$ (left) and $\cross^{*, R_{1, 2; 2}}$ (right). The horizontal blue line indicates the crossing $\cross^{*, R_{1, 2; 3}}$.}
\label{fig:gadget_sew}
\end{figure}
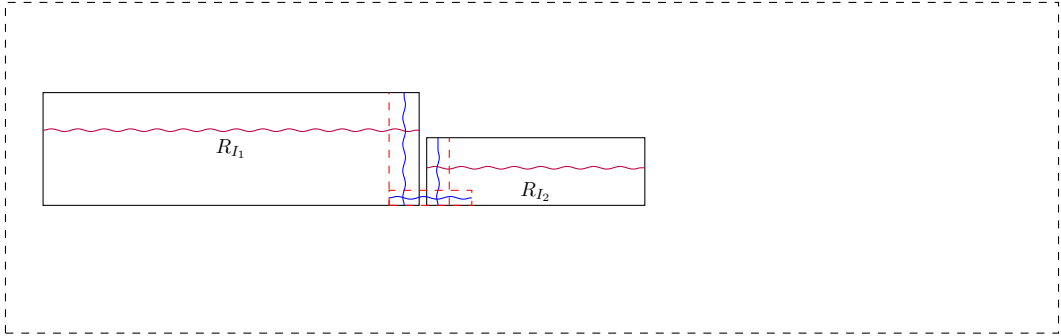

Thus we obtain our crossing $\cross^{*, \ell}_i$ through 
$\tilde V_{\ell; i}^\Gamma$. It is clear that $\cross^{*, \ell}_i$'s satisfy the hypotheses \eqref{hypo:measurabilty},\eqref{hypo:identical_distribution}, \eqref{hypo:symmetry}, \eqref{hypo:irrelevance}. 
Verifying \eqref{hypo:expected_weight} for $\ell 
\geq a'200m_\Gamma$ requires some work. We will perform a very similar calculation in Subsection~\ref{subsec:hard-remaining} to verify the 
same hypothesis for hard case. So we defer its 
discussion till then. \eqref{hypo:short_range_ratio} holds because of the induction hypothesis 
\eqref{hypo:expected_weight}. From the same hypothesis we can also deduce \eqref{hypo:gadget_negligible}.
\subsection{Induction step for the hard case: construction of the crossing}
\label{subsec:induct_step_hard}
We will first derive an ``approximate'' expression for the weight $\tilde{D}_{\gamma, \ell, i}$ that $\cross^{*, \ell}_i$ inherits from the rectangles in 
$\widetilde \descend_\ell$. This expression will then 
guide our particular switching strategy. We begin with 
some definitions. Consider an interval $I$ in $\mathscr C_{\ell, \Gamma, 0; \ell \% \Theta(\log \Gamma) + 1}$. 
Recall from the statement of \eqref{hypo:symmetry} that $\descend_{\ell, I, i}$ is the collection of all rectangles in $\widetilde{\descend}_\ell$ which are based on $I$ and whose spans are contained in the span 
of $\tilde{V}^\Gamma_{\ell; i}$. Define $\descend_{\ell, I}$ as $\descend_{\ell, I, 1} \cup \descend_{\ell, I, 
2}$. For $v \in \cup_{B \in \descend_{\ell, I, i}} B$, denote the unique rectangle in $\descend_{\ell, I, i}$ 
containing $v$ as $B_{\ell, I, i}(v)$. The counterpart for $\descend_{\ell, I}$ is denoted by $B_{\ell, I}(v)$. From definition of $\descend_\ell$ it follows that any two rectangles in $\descend_{\ell, I}$ are translates of 
each other. Call a point $w \in \cup_{B \in \descend_{\ell, I}}B$ a \emph{shift} of another point $v \in \cup_{B \in \descend_{\ell, I}}B$ (or vice versa) if $v$ gets mapped to $w$ when $B_{\ell, I}(v)$ is mapped 
to $B_{\ell, I}(w)$ via translation. We can similarly define all these terms if the underlying rectangle is $\tilde{V}_{\ell;i, j}^\Gamma$ instead of 
$\tilde{V}_\ell^\Gamma$. The corresponding notations are then modified as $\descend_{\ell, i, j; I, k}$, $\descend_{\ell, i, j; I}$, $B_{\ell, i, j; I, k}(v)$ 
and $B_{\ell, i, j; I}(v)$ respectively. Let $[v]_I = [v]_{I, \ell}$ denote the collection of all shifts of $v$ in $\cup_{B \in \descend_{\ell, I}}B$, which we also refer to as a \emph{shift class}. 
Thus we have
\begin{equation}
\label{eq:expr_total_weight}
\tilde{D}_{\gamma, \ell, i} = \sum_{I \in \mathscr C_{\ell, \Gamma, 0; \ell \% 200m_\Gamma + 1}} \sum_{[v]_I \in \descend_{\ell, I}} D_{\gamma, B_{\ell, I, i}(v^*), v^*}\e^{\gamma X_{n, B_{\ell, I, i}(v^*), \ell,v^*}}\,,
\end{equation}
where the range of second summation includes all shift classes inside $\descend_{\ell, I, i}$ and $v^* = v^{*, \ell, i}_I$ is the unique representative from 
$[v]_I$ in $\cross^{*, \ell}_i$. Suppose our algorithm at level $\ell$ respects \eqref{hypo:symmetry} and \eqref{hypo:irrelevance} (this, of course, has to be 
verified). Then it follows from our induction hypotheses that choice of $v^*$ is independent of the fields $\{\eta_{n, B, .}\}$ for $B \in \descend_{\ell, I}$. 
Consequently,
\begin{equation}
\label{eq:expr_total_weight2}
\E\tilde{D}_{\gamma, \ell, i} = \sum_{I \in \mathscr C_{\ell, \Gamma, 0; \ell \% 200m_\Gamma + 1}} \sum_{[v]_I \in \descend_{\ell, I}} d_{\gamma, \ell, [v]_I} \E \e^{\gamma X_{n, B_{\ell, I, i}(v^*), \ell,v^*}}\,,
\end{equation}
where $d_{\gamma, \ell, [v]_I}$ is the expected value of $D_{\gamma, B_{\ell, I}(v), v}$ for some (hence all) $v 
\in [v]_I$. Thus as far as expected weights are concerned, it suffices to consider the expression
\begin{equation}
\label{eq:expr_total_weight3}
\tilde{D}_{\gamma, \ell, i}^\star = \sum_{I \in \mathscr C_{\ell, \Gamma, 0; \ell \% \Theta(\eqref{eq:optimization_BM3}) + 1}} \sum_{[v]_I \in \descend_{\ell, I}} d_{\gamma, \ell, [v]_I}\e^{\gamma X_{n, B_{\ell, I, i}(v^*), \ell,v^*}}
\end{equation}
instead of \eqref{eq:expr_total_weight}. Now we will rewrite \eqref{eq:expr_total_weight3} in a way that accommodates the strategy employed in step~$\ell$. 
Recall from Subsection~\ref{subsec:descrip} that our strategy also varies along the base of 
$\tilde{V}^\Gamma_{\ell; i}$. To account for this variation we split the family $\mathscr C_{\ell, \Gamma, 0; \ell \% \Theta(\log \Gamma) + 1}$ into three subfamilies namely $\mathscr C_{\ell, 1}$, $\mathscr 
C_{\ell, 2}$ and $\mathscr C_{\ell, 3}$. $\mathscr C_{\ell, 1}$ and $\mathscr C_{\ell, 2}$ consist of intervals contained in the bases of $\tilde{V}_{\ell; i, 1}^\Gamma$ and $\tilde{V}_{\ell; i, 2}^\Gamma$ respectively while $\mathscr C_{\ell, 3}$ consists of intervals contained in $\lfloor \mathcal I_{\ell - m - 1, \Gamma, \lceil \Gamma a_{\ell-1} \rceil}\rfloor$. 
Accordingly we can split the first summation in $\eqref{eq:expr_total_weight3}$ and write
$$\tilde{D}^\star_{\gamma, \ell, i} = \tilde{D}^\star_{\gamma, \ell, i, \mathscr C_{\ell, 1}} + \tilde{D}^\star_{\gamma, \ell, i, \mathscr C_{\ell, 2}} + \tilde{D}^\star_{\gamma, \ell, i, \mathscr C_{\ell, 3}}\,.$$
We will deal with the first two parts and the third part 
separately. To avoid cumbersome notations we will only 
write the expressions for $i = 2$. From description of Strategy~II given in Subsection~\ref{subsec:descrip} we get
\begin{equation}
\label{eq:expr_total_weight4}
\tilde{D}_{\gamma, \ell, 2, \mathscr C_{\ell, j}}^\star = \sum_{I \in \mathscr C_{\ell, j}} \sum_{[v]_I \in \descend_{\ell, I}} d_{\gamma, \ell, [v]_I}\e^{\gamma X_{n, B_{\ell, 2, j; I, k_I}(v^*), \ell - 1, v^*}}\e^{\gamma X_{n, \ell - 1, \ell, v^*}}\,,
\end{equation}
where $j \in [2]$, $v^* = v^{*, \ell - 1, k_I}_I$ is the unique representative from $[v]_I$ in $\cross_{k_I}^{*, \tilde{V}_{\ell; 2, j}^\Gamma}$ and $k_I \in [2]$ is informed by the switchings at step~$\ell$. 
The job now is to optimize the combined weight from $\tilde{D}^\star_{\gamma, \ell, 2, \mathscr C_{\ell, j}}$ and vertical gadgets with respect to the 
choice of $k_I$'s. As we already mentioned in the beginning, we will optimize based on an expression that is similar to but not the same as 
\eqref{eq:expr_total_weight4}. 

We will gradually delve into the terms in $\tilde{D}_{\gamma, \ell, 2, \mathscr C_{\ell, 
j}}^\star$ to arrive at a ``nice'' approximate expression. Let us begin with an expansion of $\e^{\gamma X_{n, \ell - 1, \ell, v^*}}$ as follows:
$$\e^{\gamma X_{n, \ell - 1, \ell, v^*}} = 1 + \gamma X_{n, \ell - 1, \ell, v^*} + \tfrac{\gamma^2}{2}\E X_{n, \ell - 1, \ell, v^*}^2 + \quadr_{n,\ell, v^*} + \Taylor_{n, \ell, v^*}\,,$$
where $\quadr_{n, \ell, v^*} = \tfrac{\gamma^2}{2}(X_{n, \ell - 1, \ell, v^*}^2 - \E X_{n, \ell - 1, \ell, v^*}^2)$ and $\Taylor_{n, \ell, v^*}$ consists of the 
remaining terms. Apart from increase in length, the main contribution towards the increment in weight of $\cross_i^{*, \ell}$ comes from the random variables $\tfrac{\gamma^2}{2}d_{\gamma, \ell, [v]_I}\e^{\gamma X_{n, B_{\ell, 2, j ; I, k}(v^*), \ell - 1, v^*}}\E X_{n, \ell -1, \ell, v^*}^2$ which we 
denote as $\increment_{[v]_I, 2, k}$. Here again we took the liberty of changing the notation $v^{*, \ell - 
1, k}_I$ to $v^*$. The following simple claim relates the expected values of $\increment_{[v]_I, 2, k_I}, \increment_{[v]_I, 2, 1}$ and $\increment_{[v]_I, 2, 2}$.
\begin{claim}
\label{claim:invariant_expect}
If our construction of $\cross_2^{*, \ell}$ obeys \eqref{hypo:symmetry}, then
$$\E (\increment_{[v]_I, 2, k_I}) = \tfrac{1}{2}\E(\increment_{[v]_I, 2, 1} + \increment_{[v]_I, 2, 2}) \,,$$ 
$$\mbox{ and }  \quad \E \e^{\gamma X_{n, B_{\ell, 2, j ; I, k_I}(v^*), \ell - 1, v^*}} =  \E \e^{\gamma X_{n, B_{\ell, 2, j ; I, 1}(v^*), \ell - 1, v^*}} = \E \e^{\gamma X_{n, B_{\ell, 2, j ; I, 2}(v^*), \ell - 1, v^*}}\,.$$
\end{claim}
\begin{proof}The crucial observation is the following: $k_I$ is uniform on $\{1, 2\}$ and is independent with the fields $X_{n, B, \ell - 1, .}$'s for all $B \in \descend_{\ell -1, 
I}$. Now the claim follows immediately.
\end{proof}

As is clear from Claim~\ref{claim:invariant_expect} and the discussions that immediately precede it, the only potential contributors to the gain from a particular strategy are the following:
$$\gain_{I, 2, k} = \sum_{[v]_I \in \descend_{\ell, I}}\gamma d_{\gamma, \ell, [v]_I}\e^{\gamma X_{n, B_{\ell, 2, j ; I, k}(v^*), \ell - 1, v^*}}X_{n, \ell - 1, \ell, v^*}\,,$$
where $v^* = v^{*, \ell - 1, k}_I$. Consider an interval $I \in \mathscr C_{\ell, j}$ and a rectangle 
$B \in \descend_{\ell, 2, j; I}$. Fix $\epsilon = 
\delta^{100}$. There is a unique interval of length $\lfloor \epsilon a_\ell \rfloor + 1$ in $\mathscr C_{\ell, 1, 0; 100m, \principal}$ containing the span of $B$. Let $\nu_{B, 1}$ and $\nu_{B, 2}$ 
respectively denote the right and left endpoints of that 
interval. Also from the description of switching locations we know that there is a unique $j' \in [\Gamma_{\ell, \beta}]$ such that $I$ overlaps with 
$I_{\ell; j, j'}$. In view of Lemma~\ref{lem:field_smoothness} we can then approximate $X_{n, \ell - 1, \ell, v}$ for any point $v$ in $B$ by the average coarse field value along $(I_{\ell; j, j'} \cap I) \times \{\nu_{B, 1}\}$ or 
$(I_{\ell; j, j'} \cap I) \times \{\nu_{B, 2}\}$. Now 
let us revisit the summands in $\gain_{I, 2, k}$. Denote by $\overline X_{n, \ell, I}$ the average coarse field value along $(I_{\ell; j, j'} \cap I) \times \{\nu_{B_{\ell, 2, j;I, k}(v^*), k}\}$ and by $\resid_{n, \ell, v*}$ the difference $X_{n, \ell - 1, \ell, v^*} - \overline X_{n, 
\ell, 2, I}$. Thus we can decompose $\gain_{I, 2, k}$ as
\begin{equation}
\label{eq:expr_total_weight5}
\gain_{I, 2, k} = \gain_{I, 2, k}^\star + \widetilde{\resid}_{n, \ell, I, 2, k}\,,
\end{equation}
where
$$\gain_{I, 2, k}^\star = \sum_{[v]_I \in \descend_{\ell, I}}\gamma d_{\gamma, \ell, [v]_I}\e^{\gamma X_{n, B_{\ell, 2, j ; I, k}(v^*), \ell - 1, v^*}}\overline X_{n, \ell, 2, I}\,,$$
and
$$\widetilde{\resid}_{n, \ell, I, 2, k} = \sum_{[v]_I \in \descend_{\ell, I}}\gamma d_{\gamma, \ell, [v]_I}\e^{\gamma X_{n, B_{\ell, 2, j ; I, k}(v^*), \ell - 1, v^*}}\resid_{n, \ell, v^*}\,.$$
We have not said anything about the coefficients $\e^{\gamma X_{n, B_{\ell, 2, j ; I, k}(v^*), \ell - 1, 
v^*}}$ so far. The following lemma shows that these coefficients are reasonably close to 1.
\begin{lemma}
\label{lem:error_control}
Let $M_{n, \ell} = \max\limits_{I \in \mathscr C_{\ell, 1} \cup \mathscr C_{\ell, 2}}\max\limits_{B \in \descend_{\ell, I}}\max\limits_{v \in B}X_{n, B, \ell-1, v}$, then there exists a positive number $C_\delta$ depending solely on $\delta$ such that
\begin{equation}
\label{lem:error_control1}
\E \e^{\gamma M_{n, \ell}}\mathbf{1}_{\{M_{n, \ell} \geq C_\delta\log \Gamma\}} = O(\Gamma^{-2})\e^{\gamma O_\delta(1)\log \Gamma}\,.
\end{equation}
Furthermore if $\abs_{n, \ell} = \max\limits_{I \in \mathscr C_{\ell, 1}}\max\limits_{B \in \descend_{\ell, I}}\max\limits_{v \in B}|X_{n, B_{\ell - 1, I, k_I}(v^*), \ell -1, v^*}|$, then we have
\begin{equation}
\label{lem:error_control2}
\P \Big(\abs_{n, \ell} \geq C_\delta\log \Gamma\Big) = O(\Gamma^{-2})\,,
\end{equation}

\end{lemma}
\begin{proof}
Recall from section~\ref{subsec:self_similar} that $|\mathscr P_{\ell, k, 0, r}| \leq (2 + \delta)^{r+m}$ 
for all $k, r \in \N$. This fact and the definition of $\descend_\ell$ together imply $|\descend_\ell| \leq (2 
+ \delta)^{2\Theta(\log \Gamma) + 2m}$. Now as a consequence of Lemma~\ref{lem:free_field_var} we have
\begin{equation}
\label{eq:error_control1}
\max_{v \in \cup_{B \in \descend_{\ell}}}\E X_{n, B, \ell - 1, v}^2 \leq O(1)\log (\Gamma / \delta)\,.
\end{equation}
Using the bound on $|\descend_\ell|$, \eqref{eq:error_control1} and  Lemmas~\ref{lem:field_smoothness}, \ref{lem:max_expectation} we get
\begin{equation}
\label{eq:error_control2}
\E \Big(\max_{v \in B, B \in \descend_{\ell}} X_{n, B, \ell - 1, v}\Big) \leq O_{\delta}(1)\log \Gamma\,.
\end{equation}
Similarly
\begin{equation}
\label{eq:error_control3}
\E \Big(\min_{v \in B, B \in \descend_{\ell}} X_{n, B, \ell - 1, v}\Big) \geq -O_{\delta}(1)\log \Gamma\,.
\end{equation}
Finally Lemma~\ref{lem:Borell_ineq} and the last three displays yield us
\begin{equation}
\label{eq:error_control4}
\P\Big(\max_{v \in B, B \in \descend_\ell} |X_{n, B, \ell - 1, v}| \geq O_\delta(1)\log \Gamma + u\sqrt{\log (\Gamma/ \delta)}\Big) \leq 2\e^{-\Omega(u^2)}\,.
\end{equation}
\eqref{lem:error_control1} and \eqref{lem:error_control2} now follow for an appropriate choice of $c$.
\end{proof}
We will call the event 
$\{\abs_{n, \ell} \geq C_\delta\log \Gamma\}$ as $G_{n, 
\ell}$. Based on the last lemma we can effectively assume that the coefficients $\e^{\gamma X_{n, B_{\ell, 2, j ; I, k}(v^*), \ell - 1, v^*}}$'s are all equal to 
1. This leads us to approximate $\gain_{I, 2, k}^\star$ with $\gamma d_{\gamma, \ell, I}\overline{X}_{n, \ell, 2, I}$ where $d_{\gamma, \ell, I} = \sum_{[v]_I 
\in\descend_{\ell, I}}d_{\gamma, \ell, [v]_I}$. On the other hand, from \eqref{hypo:short_range_ratio} and \eqref{hypo:gadget_negligible} we get that $d_{\gamma, \ell, I} = \tfrac{d_{\gamma, \ell - 1}}{\Gamma a_{\ell - 1}}|I_{\ell; j, j'} \cap I|(1 + o_{\gamma \to 0; 
\delta}(1))$. Here $I_{\ell; j, j'}$ is the unique switching interval that overlaps 
with $I$. Thus it seems reasonable to replace $d_{\gamma, \ell, I}$ with $\tfrac{d_{\gamma, \ell - 1}}{\Gamma a_{\ell - 1}}|I_{\ell; j, j'} \cap I|$ in the 
corresponding expressions. Rearranging everything in terms of switching intervals we get a new process as follows:
\begin{equation}
\label{eq:expr_total_weight6}
\gain_{n, \ell, 2, j, j', k}^\star = \frac{d_{\gamma, \ell - 1}}{\Gamma a_{\ell - 1}}\sum_{I \subseteq I_{\ell; j, j'}}\sum_{w \in (I_{\ell; j, j'} \cap I)\times \{\nu_{B_{\ell, 2, I, k}, k}\}}\gamma X_{n, \ell - 1, \ell, w}\,,
\end{equation}
where $I \subseteq I_{\ell; j, j'}$ in the first summation should be interpreted as ``$I$ overlaps $I_{\ell; j, j'}$'' and $B_{\ell, 2, I, k}$ is the unique rectangle from $\descend_{\ell, I}$ in the skeleton of $\cross_k^{*, \tilde{V}^\Gamma_{\ell; 2, j}}$.

Since we have an \emph{a priori} upper bound on the number of switches due to \eqref{hypo:bound_n_switch}, the heights $\nu_{B_{\ell, 2, I, k}}$ are the same along $I_{\ell; j, j'}$ for all but 
$O_{\delta}(1)$ many pairs $(j, j')$. Let $B_{\ell, j, j',k,\mathrm{end}}$ be the rightmost rectangle whose base overlaps with $I_{\ell; j, j'}$. 
Denote $\nu_{B_{\ell, j,j',k,\mathrm{end}}, k}$ by 
$\nu_{2, j, j', k}$. Consequently we can further approximate $\gain_{n, \ell, 2, j, j', k}^\star$ with
\begin{equation}
\label{eq:expr_total_weight7}
\gain_{n, \ell, 2, j, j', k}^{\star \star} = \frac{d_{\gamma, \ell - 1}}{\Gamma a_{\ell - 1}}\sum_{j \in [2], j' \in [\Gamma_{\ell, \beta}]}\sum_{v \in I_{\ell; j, j'}\times \{\nu_{2, j, j', k}\}}\gamma X_{n, \ell - 1, \ell, v}\,.
\end{equation}
We are still a few steps away from our final expression. 
Recall that each $X_{n, \ell - 1, \ell, v}$ is a linear combination of $\eta_{n, \ell, v}$'s along the boundary of $\tilde{V}_{\ell; 2, j}^\Gamma$ where the coefficients are the values of corresponding Poisson 
kernel. Based on Lemmas~\ref{lem:kernel_bound2} and \ref{lem:kernel_bound3} we see that the main contribution comes from $\eta_{n, \ell, I_{\ell; j, j'} \times \{\lfloor a_{\ell} \rfloor + \lfloor a_{\ell - m - 1} \rfloor\}}$ and $\eta_{n, \ell, I_{\ell; j, j'} 
\times \{-\lfloor a_{\ell - m - 1} \rfloor\}}$. A direct application of Lemma~\ref{lem:reversibility} now yields that the sum of coefficients at a point $v \in I_{\ell; j, j'} \times \{\lfloor a_\ell \rfloor + \lfloor a_{\ell - m - 1} \rfloor\}$ (or $I_{\ell; j, j'} \times \{-\lfloor a_{\ell - m - 1}\rfloor \}$) is exactly 
$\tfrac{1}{4}G_{\tilde{V}^\Gamma_{\ell; 2, j}}(v_{2, j}, I_{\ell; j, j'}\times \{\nu_{2, j, j',k}\})$, where $v_{2, j}$ is the unique neighbor of $v$ in $\inte(\tilde V_{\ell; 2, j}^\Gamma)$ and 
$$G_{\tilde{V}^\Gamma_{\ell; 2, j}}(v_{2, j}, I_{\ell; j, j'}\times \{\nu_{2, j, j',k}\}) = \sum_{w \in I_{\ell; j, j'}\times \{\nu_{2, j, j',k}\}}G_{\tilde{V}^\Gamma_{\ell; 2, j, k}}(v_{2, 
j}, w)\,.$$ We shall see in the next subsection that as a consequence of $\beta$ and $\Gamma$ being large, for ``almost all'' $v \in I_{\ell; j, j'} \times \{\lfloor a_\ell \rfloor + \lfloor a_{\ell - m - 1} \rfloor\}$ (or $I_{\ell; j, j'} \times \{-\lfloor a_{\ell - m - 1}\rfloor \}$) this coefficient is very close to $\tfrac{\nu_{2, j, j', k} + \lfloor a_{\ell - m - 1}\rfloor }{\vertical_{\ell - 1}}$ (respectively $\tfrac{\lfloor a_\ell \rfloor + \lfloor a_{\ell - m - 
1}\rfloor - \nu_{2, j, j', k}}{\vertical_{\ell - 1}}$). 
Here $\vertical_{\ell - 1}$ represents the length of $\tilde{V}_{\ell; 2, j}^\Gamma$'s or equivalently $\tilde{V}_{\ell - 1}^\Gamma$'s span.
Finally, contribution from $\eta_{n, \ell, I_{\ell; j, j'} \times \{-\lfloor a_{\ell - m - 1}\rfloor \}}$ is insignificant compared to $\eta_{n, \ell, I_{\ell; j, j'} \times \{\lfloor a_\ell \rfloor + \lfloor a_{\ell - m - 1}\rfloor \}}$ since $I_{\ell; j, j'} \times \{-\lfloor a_{\ell - m - 1}\rfloor \}$ lies close to $\partial_\down \tilde{V}^\Gamma_\ell$. 
Putting all these things together we arrive at the approximation given by
\begin{equation}
\label{eq:expr_total_weight8}
\widetilde{\gain}^{\star \star}_{n, \ell, 2, j, j', k} = \frac{d_{\gamma, \ell-1}}{\Gamma a_{\ell - 1}}\sum_{j \in [2], j' \in [\Gamma_{\ell, \beta}]}\sum_{v \in I_{\ell; j, j'} \times \{\lfloor a_\ell\rfloor + \lfloor a_{\ell - m - 1}\rfloor\}}\gamma \frac{(\nu_{2, j, j', k} + \lfloor a_{\ell - m - 1}\rfloor)}{\vertical_{\ell - 1}}\eta_{n, \ell, v}\,.
\end{equation}
Repeating the entire procedure for $i = 1$ we get
\begin{equation}
\label{eq:expr_total_weight9}
\widetilde{\gain}^{\star \star}_{n, \ell, 1, j, j', k} = \frac{d_{\gamma, \ell-1}}{\Gamma a_{\ell - 1}}\sum_{j \in [2], j' \in [\Gamma_{\ell, \beta}]}\sum_{v \in I_{\ell; j, j'} \times \{p_\ell\}}\gamma \frac{(\lfloor a_{\ell + 1}\rfloor + \lfloor a_{\ell - m - 1}\rfloor - \nu_{2, j, j', k})}{\vertical_{\ell - 1}}\eta_{n, \ell, v}\,,
\end{equation}
where $(-\lfloor \Gamma a_{\ell - m - 1}\rfloor, p_\ell)$ is the lower left corner vertex of 
$\tilde{V}^\Gamma_{\ell; 1, 1,}$. Again since $p_\ell$ and $\lfloor a_\ell\rfloor + \lfloor a_{\ell - m - 1}\rfloor$ are very close to each other, we can substitute $\tfrac{\eta_{n, \ell, v} + \eta_{n, \ell, \overline{v}}}{2}$ for $\eta_{n, \ell, v}$ in the last 
two displays where $\overline{v} = (v_x, \lfloor a_\ell \rfloor + \lfloor a_{\ell - m - 1}\rfloor)$ or $(v_x, p_\ell)$ accordingly as $v_y = p_\ell$ or $\lfloor a_\ell\rfloor + \lfloor a_{\ell - m - 1}\rfloor$ 
respectively. The main purpose behind such a modification is to get rid of unnecessary correlations between different random variables that we use for 
deciding switching locations. This will be very helpful 
when we prove Lemma~\ref{lem:gadget_cost_correction}. We also need to ensure the symmetric construction of $\cross^{*, \ell}_1$ and $\cross^{*, \ell}_2$ which then makes the average of $\eta_{n, \ell, v}$ and $\eta_{n, 
\ell, \overline{v}}$ an automatic choice. Thus we get yet another bunch of random variables as follows: 
\begin{equation}
\label{eq:expr_total_weight10}
\widetilde{\gain}_{n, \ell, 2, j, j', k} = \frac{d_{\gamma, \ell-1}}{\Gamma a_{\ell - 1}}\sum_{j \in [2], j' \in [\Gamma_{\ell, \beta}]}\gamma \frac{(\nu_{2, j, j', k} + \lfloor a_{\ell - m - 1}\rfloor)}{\vertical_{\ell - 1}}\eta_{n, \ell,j, j'}\,,
\end{equation}
where $\eta_{n, \ell, j, j'} = \sum_{v \in I_{\ell; j, j'}\times \{\lfloor a_{\ell}\rfloor + \lfloor a_{\ell - m - 1}\rfloor\}}\tfrac{\eta_{n, \ell, v} + 
\eta_{n, \ell, \overline{v}}}{2}$. If the underlying rectangle is $B \in \B_{\ell; \principal}$ instead of $\tilde V^\Gamma_{\ell}$, we denote these random variables as $\eta_{n, B, j, j'}$.

The Gaussian variables $\eta_{n, \ell,2, j, j'}$'s are not 
independent, but they are very weakly correlated. 
Let $S_{j, j'}$ denote the set of all pairs in $[2] \times [\Gamma_{\ell, \beta}]$ that are smaller than 
$(j, j')$ lexicographically. Then the random variables $\tilde \eta_{n, \ell, j, j'} = \eta_{n, \ell, j, j'} - \E (\eta_{n, \ell, j, j'} | \eta_{n, \ell, 
S_{j, j'}})$ are independent. We do not lose much in terms of variance of $\eta_{n, \ell, j, j'}$ in this 
process. In fact from Lemmas~\ref{lem:general_covar} and \ref{lem:Gram_Schimidt} we get $\var (\tilde \eta_{n, \ell, j, j'}) \geq \big(1 - O(\beta^{-2})\big)\var 
(\eta_{n, \ell, j, j'})$. Substituting $\tilde \eta_{n, \ell, j, j'}$'s for $\eta_{n, \ell, j, j'}$'s in \eqref{eq:expr_total_weight10}, we obtain our final approximate expression which is
\begin{equation}
\label{eq:expr_total_weight11}
\frac{d_{\gamma, \ell-1}}{\Gamma a_{\ell - 1}}\sum_{j \in [2], j' \in [\Gamma_{\ell, \beta}]}\gamma \frac{(\nu_{2, j, j', k} + \lfloor a_{\ell - m - 1}\rfloor)}{\vertical_{\ell - 1}}\tilde \eta_{n, \ell, j, j'}\,.
\end{equation}

Having had the final expression we can now focus on the 
optimization. For convenience we will optimize 
separately for $j = 1$ and $2$. Below we discuss the case $j = 1$ only as the other case is similar. Let $G_{\gamma, \ell, 2, 1, j'}$ denote the total weight of the vertical gadget at the right end of $I_{\ell, 1, j'}$ with respect to the field $\{\eta_{n, 
\ell, .}\}$. Then our objective function is 
\begin{equation}
\label{eq:objective_fxn1}
\frac{d_{\gamma, \ell-1}}{2\Gamma a_{\ell - 1}}\sum_{j' \in [\Gamma_{\ell, \beta}]}(-1)^{k_{2, 1, j'}}\gamma \frac{(\nu_{2, 1, j', 1} - \nu_{2, 1, j', 2})}{\vertical_{\ell - 1}}\tilde \eta_{n, \ell, 1, j'} - \sum_{j' \in [\Gamma_{\ell, \beta} - 1]} \mathbf 1_{\{k_{2,1,j'} \neq k_{2,1,j'+1}\}} G_{\gamma, \ell, 2, 1, j'}\,,
\end{equation}
where $k_{2, 1, j'} = k_{\ell; 2, 1, j'} \in [2]$. 

We need to decide upon a specific way to construct the gadgets before we can talk about the expectation of 
\eqref{eq:objective_fxn1}. In Figure~\ref{fig:gadget} 
we illustrate such a construction. The covering $\mathscr C_{\ell, 1, 0; 100m}$ and the choice of 
$\Gamma$ as a power of $2 + \delta$ will be useful here. 
Recall that for $k \in [2]$, $B_{\ell, 2, j', k, \mathrm{end}}$ is the rightmost rectangle in $S_{k, \ell - 1}$ whose base overlaps with $I_{\ell, 1, j'}$ (see the discussion immediately preceding 
\eqref{eq:expr_total_weight7}). Denote by $I_{j'; \ell - 1, \epsilon, k}$ the unique interval in $\mathscr C_{\ell, 1, 0; 100m, \principal}$ containing the 
span of $B_{\ell, 2, j', k, \mathrm{end}}$. Now consider an interval $I$ in $\mathscr C_{\ell, 1, 0; 100m}$ that lies between $I_{j'; \ell - 1, \epsilon, 1}$ and 
$I_{j'; \ell - 1, \epsilon, 2}$. From the definition of $\mathscr C_{\ell, 1, 0; 100m}$ we know that $|I| = \lfloor a_{m_I}\rfloor + 1$ for some $m_I \in \{\ell - 
101m, \ldots, \ell - 100m \}$. Consequently the rectangle $R_{\ell, I} = \lfloor r_{\ell, 1, j'} - [0, \frac{a_{m_I - m_\Gamma}}{\Gamma}]\rfloor \times I$ is a copy of $\lfloor \mathcal I_{m_I - m_\Gamma, 1, 0}\rfloor \times \lfloor \mathcal I_{m_I - m_\Gamma, 
\Gamma, 0}\rfloor$. Let $\tilde R_{\ell, I}$ be a rectangle surrounding $\tilde R_{\ell, I}$ so that $(\tilde R_{\ell, I}, R_{\ell, I})$ can be mapped to $(\tilde V_{m_I - m_\Gamma}, \tilde V_{m_I - m_\Gamma, 1}^\Gamma)$ by rotation and translation of $\Z^2$. 
By our induction hypothesis, we already know how to build crossings between $\partial_{\up} R_{\ell, I}$ and $\partial_{\down} R_{\ell, I}$ when the underlying field 
is the fine field on $\tilde R_{\ell, I}$. Denote this 
crossing as $\cross^{*, R_{\ell, I}}$. The ``extra'' random variables that are used to build this crossings are independent for all such pairs $(j', I)$ (more precisely the triplets $(i, j, j', I)$) and are elements 
of $\Xi_{\ell, 5}$. We now obtain a crossing $\cross^{*, j', \ell}$ between $\partial_{\up} R_{\ell, I_{j'; \ell - 1, \epsilon, 1}}$ and $\partial_{\down} R_{\ell, I_{j'; \ell - 1, \epsilon, 1}}$ by sewing 
successive $\cross^{*, R_{\ell, I}}$'s. But we still need a few more gadgets to be able to glue the junction. 
To this end let $|I_{\ell; 1, j'}| = \lfloor a_{m_{\ell, j'}} \rfloor + 1$. Since $\ell \% 200m_\Gamma \geq m_\Gamma + 100m$ 
, we know that $\ell - 200m_\Gamma \leq m_{\ell} \leq \ell + 1 - 
100m$. Thus we can form a chain of rectangles $R_{j', k}^1 (\equiv S_{k, \ell - 1}), R_{j', k}^2, \ldots, R_{j', k}^{m_j'}$ such that the following conditions are satisfied:\\
(a) $m_j'= O(m_\Gamma)$.\\
(b) $R_{j', k}^2$ has the same base as $R_{j', k}^1$ and span length $\lfloor \Gamma a_{m_{\ell, j'}} \rfloor + 1$; $R_{j', k}^3$ has the same span as $R_{j', k}^2$ and base length $\lfloor \Gamma a_{m_{\ell, j'} + m_\Gamma} \rfloor + 1$ and so on.\\ 
(c) $R_{j', k}^{m_j'}$ is a copy of $\lfloor \mathcal I_{m_k', \Gamma, 0} \rfloor \times \lfloor \mathcal I_{m_k', 1, 0} \rfloor$ for some $m_k'$.\\
(d) $\partial_{\mathrm{left}}R_{j', k}^{m_j'} \subseteq \partial_{\mathrm{left}}R_{\ell, I_{j'; \ell - 
1, \epsilon, 1}}$. \\
Likewise in the case of $R_{\ell, I}$'s, we can use our wisdom from lower levels to construct crossings $\cross^{*, R_{j', k}^r}$'s through $R_{j', k}^r$'s alternately in the horizontal and vertical directions. 
It is now easy to see that (see Fig~\ref{fig:gadget}) the union of $\cross^{*, j', \ell}$ and $\cross^{*, R_{j', k}^1}, \cross^{*, R_{j', k}^2}, \ldots, \cross^{*, R_{j', k}^{m_j'}}$ glues 
the junction at $I_{\ell, 1, j'}$. The extra random variables that are needed to build these gadgets (including the ones used for sewing) are independent for all different gadgets and are elements of $\Xi_{\ell, 5}$.

\begin{figure}[!htb]
\centering
\begin{tikzpicture}[scale = 10]




\draw [blue, style={decorate,decoration={snake,amplitude = 0.5}}] (-0.05, 0) -- (-0.05, 0.16);
\draw [red, dashed] (-0.07, 0) rectangle (0, 0.16);

\draw [blue, style={decorate,decoration={snake,amplitude = 0.5}}] (-0.03, 0.01) -- (-0.03, -0.21);
\draw [red, dashed] (-0.09, 0.01) rectangle (0, -0.21);

\draw [blue, style={decorate,decoration={snake,amplitude = 0.5}}] (-0.04, -0.2) -- (-0.04, -0.2 -0.16);
\draw [red, dashed] (-0.07, -0.2) rectangle (0, -0.2 - 0.16);

\draw [blue, style={decorate,decoration={snake,amplitude = 0.5}}] (-0.03, -0.2 - 0.25) -- (-0.03, -0.2 - 0.15);
\draw [red, dashed] (-0.06, -0.2 - 0.15) rectangle (0, -0.2 - 0.25);

\draw [blue, style={decorate,decoration={snake,amplitude = 0.5}}] (-0.05, -0.2 - 0.24) -- (-0.05, -0.2 - 0.4);
\draw [red, dashed] (-0.07, -0.2 - 0.24) rectangle (0, -0.2 - 0.4);

\draw (-0.04, 0.06) rectangle (0, 0.08);
\draw [blue, style={decorate,decoration={snake,amplitude = 0.2}}] (-0.04, 0.07) -- (0, 0.07);
\draw (-0.04, 0.06) rectangle (0, 0.08);
\draw [blue, style={decorate,decoration={snake,amplitude = 0.2}}] (-0.02, 0.06) -- (-0.02, 0.17);
\draw [blue, style={decorate,decoration={snake,amplitude = 0.4}}] (-0.07, 0.13) -- (0, 0.13);
\draw (-0.04, -0.2 - 0.34) rectangle (0, -0.2 - 0.32);
\draw [blue, style={decorate,decoration={snake,amplitude = 0.2}}] (-0.04, -0.2 - 0.33) -- (0, -0.2 - 0.33);
\draw [blue, style={decorate,decoration={snake,amplitude = 0.2}}] (-0.02, -0.2 - 0.34) -- (-0.02, -0.2 - 0.24);
\draw [blue, style={decorate,decoration={snake,amplitude = 0.4}}] (-0.07, -0.2 - 0.27) -- (0, -0.2 - 0.27);
\draw [blue, style={decorate,decoration={snake,amplitude = 0.4}}] (-0.07, -0.2 - 0.27) -- (0, -0.2 - 0.27);
\end{tikzpicture}
\caption{{\bf Linking by a vertical gadget.}}
\label{fig:gadget}
\end{figure}
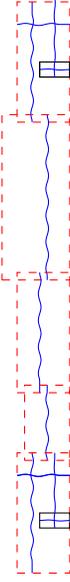

We can now analyze the expected total weight of gadgets in \eqref{eq:objective_fxn1} 
To this end denote by $\mathcal F_\ell$ the $\sigma$-field generated by the random variables $\eta_{n, B, j, j'}$ for $j \in [2], j' \in [\Gamma_{\ell' , \beta}]$, $B \in \B_{\ell; \ell', \principal}$ and $\ell' \geq (200a + 1)m_\Gamma + 100m$ where $\B_{\ell; \ell', \principal}$ consists all rectangles in $\B_{\ell', \principal}$ which are 
descendants of $\tilde{V}_\ell^\Gamma$. Due to \eqref{hypo:irrelevance} we see that  
$$\E \sum_{j' \in [\Gamma_{\ell, \beta} - 1]} \mathbf 1_{\{k_{2,1,j'} \neq 
k_{2,1,j'+1}\}}G_{\gamma, \ell, 2, 1, j'} = \E \sum_{j' \in [\Gamma_{\ell, \beta} - 1]} \mathbf 1_{\{k_{2,1,j'} \neq k_{2,1,j'+1}\}}\E (G_{\gamma, \ell, 2, 1, j'}|\mathcal F_\ell)\,.$$

Now let $\{\eta_{n, \ell, v}^*\}_{v \in \tilde V_\ell^\Gamma}$ be an independent copy of $\{\eta_{n, 
\ell, .}\}$. Denote by $G_{\gamma, \ell, 2, 1, j'}^*$ the total weight of the vertical gadget at the right end of $I_{\ell, 1, j'}$ with respect to 
$\{\eta_{n, \ell, .}^*\}$. We will estimate $\E G^*_{\gamma, \ell, 2, }$. From \eqref{hypo:short_range_ratio}, the expected weight of $\cross^{*, R_{\ell, I}}$ with respect to the fine field on $\tilde R_{\ell, I}$ is at most $\tfrac{d_{\gamma, \ell - 1}}{\Gamma a_{\ell - 1}}(1 + O(\gamma^2))(2 + \delta)^{m_I}$ where $|I| = \lfloor 
a_{m_I}\rfloor + 1$. Since the ratio of dimensions of $R_{\ell, I}$ and $\tilde V_n^\Gamma$ is bounded by some (fixed) power of $\Gamma$, it follows that the expected weight of $\cross^{*, R_{\ell, I}}$ is at most $\tfrac{d_{\gamma, \ell - 1}}{\Gamma a_{\ell - 1}}(2 + 
\delta)^{m_I}(1 + O(\gamma^2 \log \Gamma))$. Summing these over $I$ and using the facts that any two adjacent intervals in $\mathscr C_{\ell, 1, 0; 100m}$ can have at most two points in common and $\ell \geq a'200m_\Gamma$, we get that the expected weight of $\cross^{*, R_{\ell, I}}$ is bounded by $\tfrac{d_{\gamma, \ell - 1}}{\Gamma a_{\ell-1}}(\nu_{2, 1, j', 1} - \nu_{2, 1, j', 2})(1 + 
O(\gamma^2 \log \Gamma))$. Using $\vertical_{\ell - 1} = (2 + O(\delta)) a_{\ell -1}$, we can further modify this bound as $\tfrac{2d_{\gamma, \ell - 1}}{\Gamma a_\ell}\tfrac{\nu_{2, 1, j', 1} - \nu_{2, 1, 
j', 2}}{\vertical_{\ell - 1}}(1 + O(\delta))$. Now there are $O(\tfrac{\nu_{2, 1, j', 1} - \nu_{2, 1, j', 2}}{\epsilon \delta \vertical_{\ell - 1}})$ many different $I$'s and the expected weight of gadgets that are used to sew any two adjacent $\cross^{*, R_{\ell, I}}$'s is bounded by $\tfrac{d_{\gamma, \ell - 1}}{\Gamma^2} O(\epsilon)$ by a similar reasoning as 
before. Thus the expected total weight of gadgets used in sewing is at most $\tfrac{d_{\gamma, \ell - 1}}{\Gamma}\tfrac{\nu_{2, 1, j', 1} - \nu_{2, 1, j', 
2}}{\vertical_{\ell - 1}}O(\gamma^2 / \delta)$. In a similar way we find the expected total weight of $\cross^{*, R_{j', k}^r}$'s to be bounded by $\tfrac{d_{\gamma, \ell - 1}}{\Gamma}\epsilon(1 + 
O(\gamma^2 \log \Gamma))$. Since $\tfrac{\nu_{2, 1, j', 1} - \nu_{2, 1, j', 2}}{\vertical_{\ell - 1}} = \Omega (\delta)$ and $\epsilon = \delta^{100}$, we get from combining preceding discussions
\begin{equation}
\label{eq:expected_gadget_cost1}
\E G_{\gamma, \ell, 2, 1, j'}^* \leq (1 + O(\delta))\frac{2d_{\gamma, \ell - 1}}{\Gamma}\frac{(\nu_{2, 1, j', 1} - \nu_{2, 1, j', 2})}{\vertical_{\ell - 1}} = (1 + O(\delta))\frac{2d_{\gamma, \ell - 1}}{\Gamma}\Delta \tilde \nu_{2, 1, j'}\,,
\end{equation}
where $\Delta \tilde \nu_{2, 1, j'} = \tfrac{\nu_{2, 1, j', 1} - \nu_{2, 1, j', 2}}{\vertical_{\ell - 1}}$
In addition,
\begin{equation}
\label{eq:expected_gadget_cost2}
\E (G_{\gamma, \ell, 2, 1, j'} | \mathcal F_\ell) \leq \frac{\E G_{\gamma, \ell, 2, 1, j'}^*}{\E \e^{\gamma \eta_{n, \ell, \min}}}\e^{\gamma \eta_{n, \ell, \max}}\,,
\end{equation}
where $\eta_{n, \ell, \max} = \max_{v \in \tilde{V}^\Gamma_\ell}\E(\eta_{n, \ell, v} | \mathcal F_\ell)$ and $\eta_{n, \ell, \min} = \min_{v \in \tilde{V}^\Gamma_\ell}\E(\eta_{n, \ell, v} | \mathcal 
F_\ell)$. The following lemma states that $\e^{\gamma \eta_{n, \ell, \max}}$ and $\e^{\gamma \eta_{n, \ell, \min}}$ are close to 1.
\begin{lemma} We have that
\label{lem:gadget_cost_correction}
\begin{equation}
\label{lem:gadget_error_correction1}
\e^{-O(\gamma)(\log \Gamma)^{1.5}} \leq \E \e^{\gamma \eta_{n, \ell, \min}} \leq \E \e^{\gamma \eta_{n, \ell, \max}} \leq \e^{O(\gamma)(\log \Gamma)^{1.5}}\,.
\end{equation}
Also
\begin{equation}
\label{lem:gadget_error_correction2}
\E (\e^{\gamma\eta_{n, \ell, \max}} - 1)^2 \leq O(\gamma)(\log \Gamma)^{1.5}\,.
\end{equation}
\end{lemma}
\begin{proof} 
Define
$$\eta_{n, \ell, \max; \ell'} = \max_{v \in \tilde{V}_\ell^\Gamma}\E(\eta_{n, \ell, v} | \{\eta_{n, B, j, j'}\}_{j \in [2], j' \in [\Gamma_{\ell', \beta}], B \in \B_{\ell; \ell', \principal}})\,$$
and
$$\eta_{n, \ell, \min; \ell'} = \min_{v \in \tilde{V}_\ell^\Gamma}\E(\eta_{n, \ell, v} | \{\eta_{n, B, j, j'}\}_{j \in [2], j' \in [\Gamma_{\ell', \beta}], B \in \B_{\ell; \ell', \principal}})\,.$$
Note that the sequences of random variables $\{\eta_{n, B, j, j'}\}_{j \in [2], j' \in [\Gamma_{\ell', B}]}$ are independent 
for different $B$. Also the rectangles in $\B_{\ell; \ell', \principal}$ are disjoint. Thus for any $v \in \cup_{B \in \B_{\ell; \ell', \principal}} \inte(B)$ we have
$$\E (\eta_{n ,\ell, v} | \{\eta_{n, B, j, j'}\}_{j \in [2], j' \in [\Gamma_{\ell', \beta}], B \in \B_{\ell; \ell', \principal}})  = \E (\eta_{n ,\ell, v} | \{\eta_{n, B_{\ell'}(v), j, j'}\}_{j \in [2], j' \in [\Gamma_{\ell', \beta}]})\,,$$
where $B_{\ell'}(v)$ is the unique rectangle in $\B_{\ell; \ell', \principal}$ containing $v$. On the other hand this conditional expectation is 0 for all $v \notin \cup_{B \in 
\B_{\ell; \ell', \principal}} \inte(B)$. Now let $u, v \in \inte(B)$ such that $||u - v||_\infty \leq \vertical_{\ell'}$ (recall that $\vertical_{\ell'}$ is the length of the span of 
$B$). Then from Remark~\ref{remark:symmetrize} (see also Lemma~\ref{lem:smoothness_conditional}) and Lemma~\ref{lem:cond_general} we get
$$\E (\eta_{n, \ell, u} - \eta_{n, \ell, v} | \{\eta_{n, B_{\ell'}(v), j, j'}\}_{j \in [2], j' \in [\Gamma_{\ell', \beta}]})^2 \leq O(1) \frac{|u - v|}{\vertical_{\ell'}} \,.$$
From Lemma~\ref{lem:general_covar} and Lemma~\ref{lem:cond_general} it also follows that 
$$\E (\eta_{n, \ell, u}| \{\eta_{n, B_{\ell'}(v), j, j'}\}_{j \in [2], j' \in [\Gamma_{\ell', \beta}]})^2 \leq O(1)\,,$$
for all $u \in \inte(B)$. Since the number of rectangles in $\B_{\ell; \ell', \principal}$ is $2^{\ell - \ell'}$, we can apply Lemma~\ref{lem:max_expectation} in view of the last two displays to get 
\begin{equation}
\label{eq:gadget_cost_correction1_prep1}
\E \eta_{n, \ell, \max; \ell'} \leq O(\sqrt{\log \Gamma})\,.
\end{equation}
Also from Lemma~\ref{lem:Borell_ineq}, we get
\begin{equation}
\label{eq:gadget_cost_correction1_prep2}
\P(|\eta_{n, \ell, \max; \ell'} - \E \eta_{n, \ell, \max; \ell'}| \geq x) \leq 2 \e^{-\Omega(x^2)}\,,
\end{equation}
for all $x > 0$. Thus, 
\begin{equation}
\label{eq:gadget_cost_correction1}
\E \e^{c\gamma \eta_{n, \ell, \max; \ell'}} \leq \e^{O(\gamma)\sqrt{\log \Gamma}} (1 + O(\gamma))\,,
\end{equation}
where $c > 0$ is bounded. Similarly,
\begin{equation}
\label{eq:gadget_cost_correction2}
\E \e^{\gamma \eta_{n, \ell, \min; \ell'}} \geq \e^{-O(\gamma)\sqrt{\log \Gamma}}\,.
\end{equation}
Since $\eta_{n, \ell, \max} \leq \sum_{\ell'}\eta_{n, \ell, \max; \ell'}$, $\eta_{n, \ell, \min} \geq \sum_{\ell'}\eta_{n, \ell, \min; \ell'}$ and $(\eta_{n, \ell, \max; \ell'}, \eta_{n, \ell, \min; \ell'})$'s are independent, the last two displays give us
\begin{equation}
\label{eq:gadget_cost_correction3}
\E \e^{\gamma \eta_{n, \ell, \max}} \leq \e^{O(\gamma)(\log \Gamma)^{1.5}}(1 + O(\gamma \log\Gamma))\mbox{ and }\E \e^{\gamma \eta_{n, \ell, \min}} \geq \e^{-O(\gamma)(\log \Gamma)^{1.5}}\,.
\end{equation}
This proves the first part of the lemma. Now as $\E(\e^{\gamma\eta_{n, \ell, \max}} - 1)^2 \leq \E \e^{2\gamma\eta_{n, \ell, \max}} + 1$, we can again apply \eqref{eq:gadget_cost_correction1} to get
\begin{equation}
\label{eq:gadget_cost_correction4}
\E(\e^{\gamma\eta_{n, \ell, \max}} - 1)^2 \leq \e^{O(\gamma)(\log \Gamma)^{1.5}}(1 + O(\gamma \log \Gamma)) + 1\,,
\end{equation}
which proves the second part of the lemma.
\end{proof}
From \eqref{eq:expected_gadget_cost1}, \eqref{eq:expected_gadget_cost2} and Lemma~\ref{lem:gadget_cost_correction} we obtain
\begin{equation}
\label{eq:expected_gadget_cost3}
\E \sum_{j' \in [\Gamma_{\ell, \beta} - 1]}\mathbf 1_{\{k_{2, 1,j'} \neq k_{2,1, j' + 1}\}}\E(G_{\gamma, \ell, 2, 1, j'} | \mathcal F_\ell) \leq \E_1 + \E_2\,,
\end{equation}
where 
$$\E_1 = (1 + O(\delta))\E \sum_{j' \in [\Gamma_{\ell, \beta} - 1]}\mathbf 1_{\{k_{2,1,j'} \neq k_{2,1,j'+1}\}}\frac{2d_{\gamma, \ell - 1}}{\Gamma}\Delta \tilde \nu_{2, 1, j'}\,,$$  
and 
$$\E_2 = O(\gamma^{0.5})(\log \Gamma)^{0.75}\frac{2d_{\gamma, \ell - 1}}{\Gamma}\E (\mbox{number of switches})\,.$$
Using the terms in $\E_1$ as penalties, we can modify the objective function in \eqref{eq:objective_fxn1} as
\begin{equation}
\label{eq:objective_fxn2}
\frac{d_{\gamma, \ell-1}}{2\Gamma a_{\ell - 1}}\sum_{j' \in [\Gamma_{\ell, \beta}]}(-1)^{k_{2, 1, j'}}\gamma \Delta \tilde \nu_{2, 1, j'}\tilde \eta_{n, \ell, 1, j'} - (1 + O(\delta))\sum_{j' \in [\Gamma_{\ell, \beta} - 1]} \mathbf 1_{\{k_{2,1,j'} \neq k_{2,1,j'+1}\}}\frac{2d_{\gamma, \ell - 1}\Delta \tilde \nu_{2, 1, j'}}{\Gamma}\,,
\end{equation}
Call this expression $\mathcal I_\ell(k_{2, 1, 1}, 
\ldots, k_{2, 1, [\Gamma_{\ell, \beta}]})$. We will use Theorem~\ref{thm-total-variation} to devise a switching strategy $\{k_{2, 1, j'}\}_{j' \in [\Gamma_{\ell, \beta}]}$ such that $\E \mathcal I_\ell(k_{2, 1, 1}, \ldots, k_{2, 1, [\Gamma_{\ell, \beta}]})$ 
is large. But for that we need to relate this expression to regularized total variation of a Brownian motion which was defined in the beginning of 
Section~\ref{sec:total_variation}. A natural way (also used in \cite{DG15}) is to extend the discrete time process $\tfrac{d_{\gamma, \ell - 1}}{2\Gamma a_{\ell - 1}}\sum_{j'' \leq j'}\gamma\Delta\tilde \nu_{2,1 ,j''}\tilde \eta_{n, \ell, 1, j''}$ to a standard 
Brownian motion. We can do this with an additional sequence of i.i.d.\ standard Gaussians $\{Z_{\ell, 1; 
m}\}_{m \geq 1}$ using L\'{e}vy's construction. Here we assume that the variables $Z_{\ell, 1; m}$'s are elements of $\Xi_{\ell, 2}$. 
So we have a standard Brownian motion $\{S_{t, \ell}\}_{0 \leq t \leq T_{2, 1, \gamma, \ell}}$ where $T_{2, 1, \gamma, \ell} = \var \big(\tfrac{d_{\gamma, \ell - 1}}{2\Gamma a_{\ell - 1}}\sum_{j' \in [\Gamma_{\ell,\beta}]}\gamma\Delta\tilde 
\nu_{2,1 ,j'}\tilde \eta_{n, \ell,1, j'}| \mathcal 
F_{\ell - 1}\big)$. Recall that $\mathcal F_{\ell - 1}$ is the $\sigma$-field generated by the random variables $\eta_{n, B, j, j'}$ for $j \in [2], j' \in [\Gamma_{\ell' , \beta}]$, $B \in \B_{\ell - 1; \ell', \principal}$ and $\ell' \geq (a + p)200m_\Gamma$. 
The choice of the penalty function $\lambda: [0, T_{2, 1, \gamma, \ell}] \mapsto [0, \infty)$ is now obvious:
\begin{equation}
\label{penalty_function}
\lambda(t) = (1 + O(\delta))\sum_{j' \in [\Gamma_{\ell, \beta}]}\frac{2d_{\gamma, \ell - 1}}{\Gamma}\Delta \tilde\nu_{2,1, j'} \mathbf 1_{(g_{2, 1, \gamma, j'-1}, g_{2, 1, \gamma, j'}]}(t)\,,
\end{equation}
where $g_{2,1,\gamma, j'} = \var \big(\tfrac{d_{\gamma, \ell - 1}}{2\Gamma a_{\ell - 1}}\sum_{j'' \leq j'}\gamma\Delta\tilde \nu_{2,1 ,j''}\tilde \eta_{n, 
\ell, 2, 1, j''}\big)$. For this particular $\lambda$, we have that $\lambda_*$ (see Section~\ref{sec:total_variation}) is given by the following expression:
$$\lambda_*^{-2} = (1 - O(\delta))\frac{\gamma^2}{16a_{\ell - 1}^2}\sum_{j' \in [\Gamma_{\ell, \beta}]}\var (\tilde \eta_{n, \ell, 1, j'})\,.$$
From Lemma~\ref{lem:general_covar} and Remark~\ref{remark:symmetrize} we get
$$(1 - O(\delta))\frac{\gamma^2}{16a_{\ell - 1}^2}(\Gamma_{\ell, \beta} - 1)\beta a_\ell^2 \leq \lambda_*^{-2} \leq (1 + O(\delta))\frac{\gamma^2}{16a_{\ell - 1}^2}\Gamma_{\ell, \beta}\beta a_\ell^2\,.$$
Since $\alpha \leq \Gamma\gamma^2 \leq (2 + \delta)\alpha $ and $\Gamma_{\ell, \beta} \approx \Gamma / \beta$, the last inequality implies
\begin{equation}
\label{eq:lambda*_bound}
(1 - O(\delta))\alpha/4 \leq \lambda_*^{-2} \leq (1 + O(\delta))\alpha/2\,.
\end{equation}
From \eqref{hypo:bound_n_switch} and the bound on $\lambda_*$, it follows that $N_{\lambda, \star} 
\leq O(3^{101m}\alpha) = O(\delta^{-203m}) = O(\alpha^{820m})$ (see Section~\ref{sec:total_variation}). Therefore by Theorem~\ref{thm-total-variation} and Remark~\ref{remark:total-variation-general-time} we can find, for sufficiently large $\alpha$, a partition $\mathcal Q^*_{\ell, 2, 1} = (q_{0; \ell, 2, 1}^*, q_{1; \ell, 2, 1}^*, \ldots,$ $q_{k+1; \ell, 2, 1}^*)$ of $[0, T_{2, 1, \gamma, \ell}]$ such that $k \leq 2/\lambda_*^2$ and 
\begin{eqnarray}
\label{eq:optimization_BM}
\E (\Phi_{\lambda, \mathcal Q^*_{\ell, 2, 1}} (S_{.,\ell}) | \mathcal F_{\ell - 1}) \geq  0.9999\E(\int_{[0,T_{2, 1, \gamma, \ell}]} \frac{1}{\lambda(t)} d t | \mathcal F_{\ell - 1}) - O(\E(\lambda_\infty \lambda_*^{-1.5} | \mathcal F_{\ell -1}))\,.
\end{eqnarray} 

We can now describe a strategy $\{k_{2, 1, j'}^*\}_{j' \in [\Gamma_{\ell, \beta}]}$ using the partition $\mathcal Q_{\ell, 2, 1}^*$:
$$k^*_{2,1,j'} = \begin{cases}
1 &\mbox{ if } q_{0; \ell, 2, k'}^* \leq g_{2, 1, \gamma, j' - 1} < q_{0; \ell, 2, k'+1}^*\mbox{ such that }S_{q_{0; \ell, 2, k'}^*} > S_{q_{0; \ell, 2, k'+1}^*}\,, \\
2 &\mbox{ if } q_{0; \ell, 2, k'}^* \leq g_{2, 1, \gamma, j' - 1} < q_{0; \ell, 2, k'+1}^*\mbox{ such that }S_{q_{0; \ell, 2, k'}^*} < S_{q_{0; \ell, 2, k'+1}^*}\,.
\end{cases}$$
$k^*_{2, 1, j'}$'s are not necessarily uniform on $\{1, 2\}$. But we can make them uniform in the following way. 
Let $s_{\ell, 2, 1}$ be a fair Bernoulli variable that is an element of $\Xi_{\ell, 5}$ and independent of all 
the extra random variables used so far. In particular 
$s_{\ell, 2, 1}$ is independent of $\{S_{. \ell}\}$. Now if $s_{\ell, 2, 1} = 1$, we simply define $k_{2, 1, 
j'} = k^*_{2, 1, j'}$. Otherwise we reconstruct $\mathcal Q^*_{\ell, 2, 1}$ starting from a $\lambda_*$-downtick (see section~\ref{sec:total_variation}) and define $\{k_{2, 1, j'}\}_{j' \in [\Gamma_{\ell, \beta}]}$ to be the strategy obtained from this new partition similarly as 
before. 
Notice that $\E \mathcal I_\ell(k_{2, 1, 1}, \ldots, k_{2, 1, [\Gamma_{\ell, \beta}]}) = \E \mathcal I_\ell(k^*_{2; 1, 1}, \ldots, 
k^*_{2, 1, [\Gamma_{\ell, \beta}]})$. We construct the crossing $\cross^{*,  \tilde V_{\ell; 2, 1}}$ through $\tilde V_{\ell; 2, 1}$ using this strategy.

So we have built two crossings $\cross^{*, \tilde V_{\ell; 2, 1}}$ and $\cross^{*, \tilde V_{\ell; 2, 2}}$ through $\tilde V_{\ell; 2, 1}$ and $\tilde V_{\ell; 2, 
2}$ respectively. What remains is to join them into a 
crossing for $\tilde V_{\ell; 2}$. To this end we select an interval $I_{\midd, \ell, 2}$ uniformly from 
$\mathscr C_{\ell, 1, 0; m, \principal}$. Let $\tilde V_{I_{\midd, \ell, 2}}^\Gamma$ be the rectangle in $\tilde V_{\ell-1, \midd}$ corresponding to $I_{\midd, 
\ell, 2}$. We will construct a crossing $\cross^{*, \midd, \ell, 2}$ through $\tilde V_{I_{\midd, \ell, 2}}^\Gamma$ by using a modification of Strategy I 
discussed in Subsection~\ref{subsec:descrip}. Instead of going down to the rectangles in $\widetilde \descend_\ell$ along the branches descending from $\tilde V_{I_{\midd, \ell, 2}}^\Gamma$ (in $\T_n$), we stop once we meet a node of depth $\leq \ell + 1 - 
100m$. From our induction hypotheses we can construct efficient crossings through each of the 
rectangles so obtained. Now we work upwards from these crossings to $\cross^{*, \midd, \ell, 2}$ following 
Strategy I. The extra (interval valued) random variables that we need for this purpose are elements of 
$\Xi_{\ell, 4}$ (of $\Xi_{\ell, 3}$ if $i$ was 1). Finally we join $\cross^{*, \midd, \ell, 2}$ to $\cross^{*, \tilde V_{\ell; 2, 1}}$ and $\cross^{*, \tilde V_{\ell; 2, 2}}$ by simply gluing 
the corresponding junctions. It is clear from our discussions so far that our construction satisfies the hypotheses \eqref{hypo:measurabilty}, \eqref{hypo:identical_distribution}, 
\eqref{hypo:symmetry} and \eqref{hypo:irrelevance}. It also obeys \eqref{hypo:bound_n_switch} due to \eqref{eq:lambda*_bound}.

\subsection{Induction for the hard case: justifying the approximations in Subsection~\ref{subsec:induct_step_hard}}
\label{subsec:error_terms}
Here we will show that the cumulative effect of different approximations that we made in Subsection~\ref{subsec:induct_step_hard} is negligible. 
We will address all the error terms one by one and in doing so will frequently use the notations introduced in 
the previous subsection. Let us begin with $\Taylor_{n, \ell, .}$ which is the easiest candidate on our list. 
Notice that
\begin{equation}\tag{E1}
\label{eq:error_bound_Taylor}
|\Taylor_{n, \ell, v}| \leq \gamma^3 |X_{n, \ell, v}|^3(\e^{\gamma X_{n, \ell, v}} + \e^{-\gamma X_{n, \ell, v}})\,.
\end{equation}
Thus for $i, j \in [2]$,
\begin{equation}
\label{eq:err_bnd1}
\E \Big(\sum_{k \in [2], I \in \mathscr C_{\ell, j}}d_{\gamma, \ell, [v]_I}\e^{\gamma X_{n, B_{\ell, i, j; I, k}(v^*), \ell - 1, v^*}}|\Taylor_{n, \ell, v^*}|\Big) = O_\delta(\gamma^3)d_{\gamma, \ell-1}\,,
\end{equation}
where $v^* = v_I^{*, \ell - 1, k}$.

For all other error terms we will restrict ourselves to 
the ``good'' event $G_{n, \ell}$. Henceforth all the moments in this subsection involving $X_{n, \ell - 1, \ell, v}$'s should be implicitly assumed to be 
conditioned on the event $G_{n, \ell}$. An important piece of observation is that the event $G_{n, \ell}$ is independent of the random variables $X_{n, \ell - 1, 
\ell, v}$'s. Keeping this in mind we now move on to our next error term  $\quadr_{n, \ell, .}$. We will 
tackle this process in two stages. In the first stage we will approximate $X_{n, \ell - 1, \ell, v}$ by another Gaussian variable $Y_{n, \ell, v}$ such that $\E (X_{n, \ell - 1, \ell, v} - Y_{n, \ell, v})^2$ is very small. Moreover the processes $\{Y_{n, \ell, v}\}_{v \in \tilde{V}_{\ell; i, j, j', k}}$ and $\{Y_{n, \ell, v}\}_{v \in \tilde{V}_{\ell; i, j, j', k}}$ will be independent whenever $j' \neq j''$ are of same 
parity. In the second stage we will exploit the aforementioned independence of the processes $\{Y_{n, \ell, v}\}_{v \in \tilde{V}_{\ell; i, j, j', k}}$'s to argue that the maximum and minimum partial sums are 
small in magnitude. Here $\tilde{V}_{\ell; i, j, j', k}$ is the sub-rectangle of $\tilde{V}_{\ell; i, j, k}$ 
based on $I_{\ell; j, j'}$. Let us just focus on odd $j'$'s as the analysis for even $j'$'s is similar. Place vertical segments halfway between successive 
$\tilde{V}_{\ell; i, j, j', k}$'s. As a result we get, for each $j'$, two rectangles $\tilde{V}^*_{\ell; i, j, j'}$ and $\tilde{V}^{\star}_{\ell; i, j, j'}$ containing 
$\tilde V_{\ell; i, j, j', k}$ that are sub-rectangles of $\tilde{V}_{\ell; i, j, j'}$ and 
$\tilde{V}_{\ell}^\Gamma$ respectively. See Figure~\ref{fig:coarse_field_quadr} for an illustration.
\begin{figure}[!htb]
\centering
\begin{tikzpicture}[semithick, scale = 3.4]


\draw  (-2.1, 0.35) rectangle (-0.2, 0.55);

\draw  (-2.1, 0.1) rectangle (-0.2, 0.3);
\foreach \i in {-1.625, -1.150, -0.675}
{\draw [dashed] (\i, 0.1) -- (\i, 0.3);
 \draw [dashed] (\i, 0.35) -- (\i, 0.55);
 }
 
\node [scale = 0.55] at (-1.8625, 0.45) {$\tilde V_{\ell; 1, 1, 1, 1}$};
\node [scale = 0.55] at (-1.8625, 0.2) {$\tilde V_{\ell; 1, 1, 1, 2}$};

\node [scale = 0.55] at (-0.9125, 0.45) {$\tilde V_{\ell; 1, 1, 3, 1}$};
\node [scale = 0.55] at (-0.9125, 0.2) {$\tilde V_{\ell; 1, 1, 3, 2}$};

\draw [dashed, blue] (-1.3875, -0.8) -- (-1.3875, 0.8);

\draw [dashed, blue] (-0.4375, -0.8) -- (-0.4375, 0.8);

\draw [dashed, blue] (-2.15, -0.8) -- (-2.15, 0.8);

\draw[dashed] (-2.4, -0.8) rectangle (2.4, 0.8);

\draw[dashed] (-0.08, 0.05) rectangle (-2.22, 0.6);

\end{tikzpicture}
\caption{{\bf The rectangles $\tilde{V}^*_{\ell; i, j, j'}$ and $\tilde{V}^{\star}_{\ell; i, j, j'}$.} In this figure we only illustrate for $j' = 1$ and $3$. The blue and black broken lines define two rectangles containing each $\tilde V_{\ell; i, j, j', k}$. The smaller one is $\tilde{V}^*_{\ell; i, j, j'}$ and the bigger one is $\tilde{V}^\star_{\ell; i, j, j'}$.}
\label{fig:coarse_field_quadr}
\end{figure}
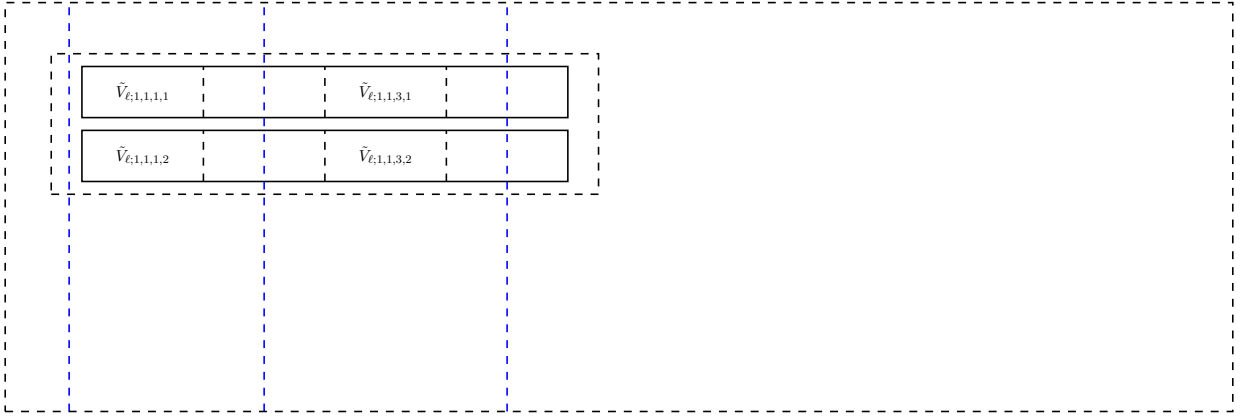
Now define $Y_{n, \ell, v}$ at $v \in \tilde{V}_{\ell; i, j, j', k}$ as
$$Y_{n, \ell, v} = \E (\eta_{n, \ell, v}| \eta_{n, \ell, \partial \tilde{V}^*_{\ell; i, j, j'}}) - \E (\eta_{n, \ell, v}| \eta_{n, \ell, \partial \tilde{V}^\star_{\ell; i, j, j'}})\,.$$
The processes $\{Y_{n, \ell, v}\}_{v \in \tilde{V}_{\ell; i, j, 1, k}}, \{Y_{n, \ell, v}\}_{v \in \tilde{V}_{\ell; i, j, 3, k}}, \ldots$ 
are independent due to Markov random field property of 
GFF. Let $\eta_{n, \ell, v}^*$ denote the random variable $\eta_{n, \ell, v} - \E (\eta_{n, \ell, v}| \eta_{n, \ell, \partial \tilde{V}^*_{\ell; 1, j', k}})$. 
By repeated applications of Markov random field property we can decompose $\eta_{n, \ell, v}$ into a sum of three independent random variables in two different ways, namely 
\begin{equation*}
\eta_{n, \ell, v} = \eta_{n, \ell, v}^* + \E (\eta_{n, \tilde V^\Gamma_{\ell; i, j}, v}| \eta_{n, V^\Gamma_{\ell; i, j}, \partial \tilde{V}^*_{\ell; i, j, j'}}) + X_{n, \ell -1, \ell, v}\,
\end{equation*}
and
\begin{equation*}
\eta_{n, \ell, v} = \eta_{n, \ell, v}^* + \E (\eta_{n, \ell, v}| \eta_{n, \ell, \partial \tilde{V}^\star_{\ell; i, j, j'}}) + Y_{n, \ell, v}\,.
\end{equation*}
The last two displays along with Lemma~\ref{lem:greens_overshoot} imply
\begin{equation*}
\E (X_{n, \ell - 1, \ell, v} - Y_{n, \ell, v})^2 
= \E (\E (\eta_{n, \ell, v}| \eta_{n, \ell, \partial \tilde{V}^\star_{\ell; i, j, j'}}) - \E (\eta_{n, \tilde V^\Gamma_{\ell; i, j}, v}| \eta_{n, V^\Gamma_{\ell; i, j}, \partial \tilde{V}^*_{\ell; i, j, j'}}))^2 \leq \e^{-\Omega(\beta)}\,.
\end{equation*}
Let $\widetilde \quadr_{n, \ell, v^*} = \tfrac{\gamma^2}{2}(Y_{n, \ell - 1, \ell, v^*}^2 - \E Y_{n, \ell - 1, \ell, v^*}^2)$. Then from the last display we get that for $i, j \in [2]$,
\begin{eqnarray}
\label{eq:quad_expec_bound1}
&&\E \Big( \sum_{k \in [2], I \in \mathscr C_{\ell, j}}d_{\gamma, \ell, [v]_I}\e^{\gamma X_{n, B_{\ell, i, j; I, k}(v^*), \ell - 1, v^*}}|\quadr_{n, \ell, v^*} - \widetilde \quadr_{n, \ell, v^*}| \Big) \nonumber \\
&\leq& O(\log (1/\delta))\e^{-\Omega(\beta)}d_{\gamma, \ell - 1}\gamma^2 = O(\delta^{20}) d_{\gamma, \ell - 1}\gamma^2\,.
\end{eqnarray}
Thus the first stage of our analysis is completed. For the second stage we will use the next lemma.
\begin{lemma}
\label{lem:gamma_chaining}
Let $G_{j,k} \in \Gamma(\mu_{j, k}, 1/2)$ for $k \in [N], j \in [n_k]$ such that processes $\{G_{j, 1}\}_{j \in [n_1]}$, $\{G_{j, 2}\}_{j \in [n_2]}, \ldots, 
\{G_{j, N}\}_{j \in [n_N]}$ are independent. Also let $\{c_{j, k}\}_{k \in [N], j \in [n_k]}$ be positive numbers with total sum 1 satisfying $\sum_{j \in 
[n_k]}c_{j, k} \leq O(N^{-1})$ for all $k \in [N]$. 
Denote the maximum and minimum partial sums of the sequence $\sum_{j \in [n_1]}c_{j, 1}(G_{j, 1} - \mu_{j, 1}), \sum_{j \in [n_2]}c_{j, 2}(G_{j, 2} - \mu_{j, 2}), \ldots, \sum_{j \in [n_N]}c_{j, N}(G_{j, N} - \mu_{j, 
N})$ by $M_N$ and $m_N$ respectively. Then we have
$$\E {M_N}^2,  \E {m_N}^2 = \frac{O(\mu^2)}{N}\,,$$
where $\mu = \sup_{j, k} \E G_{j, k} = \sup_{j, k}\mu_{j, k}$.
\end{lemma}
\begin{proof}
Denote 
$$\tilde G_k = \sum_{j \in [n_k]}c_{j, k}(G_{j, k} - \mu_{j, k})\,.$$
$\tilde G_k$'s are independent and $\var(\tilde G_k) \leq O(\mu^2)/N^2$ all $k \in [N]$. The bounds now follow from Doob's inequality.
\end{proof}
Applying Lemma~\ref{lem:gamma_chaining} to $Y_{n, \ell, v}$'s separately for $v \in \cup_{j'}\tilde V_{\ell; i, j, 2j' + 1, k}$ and $\cup_{j'}\tilde V_{\ell; i, j, 2j', k}$, we get
\begin{equation}\tag{E2}
\label{eq:error_bound_quad}
\E \Big(\max_{j_1', j_2' \in [\Gamma_{\ell, \beta}]}\mid \sum_{j' \in [j_1', j_2']} \sum_{I \subseteq I_{\ell, j, j'}}\e^{\gamma X_{n, B_{\ell, i, j; I, k}(v^*), \ell - 1, v^*}}\widetilde \quadr_{n, \ell, v}\mid \Big) \leq O_\delta(\gamma^3)d_{\gamma, \ell - 1}\,,
\end{equation}
where, as in Subsection~\ref{subsec:induct_step_hard}, ``$I \subseteq I_{\ell; j, j'}$'' means ``$I$ overlaps $I_{\ell; j, j'}$''.

Next we deal with $\widetilde \resid_{n, \ell, 
I, i, k}$'s for  $I \in \mathscr C_{\ell, 1} \cup 
\mathscr C_{\ell, 2}$. The following obvious lemma will be  useful.
\begin{lemma}
\label{lem:var_bound}
Let $Y_1, Y_2, \ldots, Y_N$ be random variables such that $\var Y_i \leq A_1$ and $\cov (Y_i, Y_j) \leq A_2\rho^{|i - j|}$ for some $A_1, A_2 > 0$ and  $0 < \rho < 1$. Then $\var(Y_1 + Y_2 + \ldots Y_N) \leq N (A_1 + 2A_2\tfrac{\rho}{1 - \rho})$.
\end{lemma}
From Lemma~\ref{lem:field_smoothness} we have $\E (\resid_{n, \ell, v})^2 \leq O(1 / \delta^3)\epsilon^2$. 
Denote the sum $\sum_{I \subseteq I_{\ell; j, j'}}\widetilde \resid_{n, \ell, I, i, k}$ by $\widetilde 
\resid_{n, \ell, i, j, j', k}$. We will use the following obvious bound on the variance of $\widetilde \resid_{n, \ell, i, j, j', k}$:
$$\var\big(\widetilde \resid_{n, \ell, i, j, j', k}\big) = O(\gamma^2/\delta^3)\frac{d_{\gamma, \ell - 1}^2}{\Gamma_{\ell, \beta}^2}\epsilon^2\,.$$
In order to estimate the covariance we will use Lemma~\ref{lem:free_field_cov} and the argument used in the proof of 
Lemma~\ref{lem:general_covar}. These give us,
$$|\cov\big(\widetilde \resid_{n, \ell, i, j, j', k}, \widetilde \resid_{n, \ell, i, j, j'', k}\big)| = O(\gamma^2 / \delta^3)\e^{-\Omega(\beta)(|j'' - j'|-1)} \frac{(\log \beta)^2}{\beta^2}\frac{d_{\gamma, \ell - 1}^2}{\Gamma_{\ell, \beta}^2}\epsilon^2\,.$$
The last two displays together with Lemma~\ref{lem:var_bound} imply
$$\var \Big(\sum_{j' \in [j_1', j_2']}\widetilde \resid_{n, \ell, i, j, j', k}\Big) = O(\gamma^2 / \delta^3)\frac{j_2' - j_1' + 1}{\Gamma_{\ell, \beta}}\frac{d_{\gamma, \ell - 1}^2}{\Gamma_{\ell, \beta}}\epsilon^2\,.$$
From Lemma~\ref{lem:basic_chaining}, we see that
\begin{equation}\tag{E3}
\label{eq:error_bound_resid}
\E \Big(\max_{j_1', j_2' \in [\Gamma_{\ell, \beta}]} \sum_{j' \in [j_1', j_2']}\widetilde \resid_{n, \ell, i, j, j', k}\Big) \leq O(\gamma^2 / \delta^{1.5})d_{\gamma, \ell - 1}\sqrt{\beta / \alpha}\epsilon\,.
\end{equation}

Now we will justify the approximation of $\sum_{I \subseteq I_{\ell; j, j'}}\gain^\star_{I, i, k}$ from \eqref{eq:expr_total_weight5} with $\gain^\star_{n, 
\ell, 2, j, j', k}$ in \eqref{eq:expr_total_weight6}. 
Denote their difference by $\Delta \gain^\star_{n, \ell, 
i, j, j', k}$. Since $\var X_{n, \ell - 1, \ell, v} = O(\log (1 / \delta))$ for all $v \in \tilde V_{\ell; i, j, k}^\Gamma$, it follows that
$$\var (\Delta \gain^\star_{n, \ell, i, j, j', k}) = o_{\gamma \to 0; \delta}(\gamma^2)\frac{d_{\gamma, \ell - 1}^2}{\Gamma_{\ell, \beta}^2}\,.$$
For the covariance, by
Lemma~\ref{lem:free_field_cov} and a similar argument as in the the proof of Lemma~\ref{lem:general_covar}, we obtain
$$|\cov(\Delta \gain^\star_{n, \ell, i, j, j', k}, \Delta \gain^\star_{n, \ell, i, j, j'', k})| = o_{\gamma \to 0; \delta}(\gamma^2)\e^{-\Omega(\beta)(|j'' - j'| - 1)}\frac{(\log \beta)^2}{\beta^2}\frac{d_{\gamma, \ell - 1}^2}{\Gamma_{\ell, \beta}^2}\,.$$
Combining the last two displays and Lemma~\ref{lem:var_bound}, we get that
$$\var \Big(\sum_{j' \in [j_1', j_2']}\Delta \gain^\star_{n, \ell, i, j, j', k}\Big) = o_{\gamma \to 0; \delta}(\gamma^2)\frac{j_2' - j_1' + 1}{\Gamma_{\ell, \beta}}\frac{d_{\gamma, \ell - 1}^2}{\Gamma_{\ell, \beta}}\,.$$
Hence from Lemma~\ref{lem:basic_chaining} it follows that
\begin{equation}\tag{E4}
\label{eq:error_bound_Delta_gain_star}
\E \Big(\max_{j_1', j_2' \in [\Gamma_{\ell, \beta}]} \sum_{j' \in [j_1', j_2']}\Delta \gain^\star_{n, \ell, i, j, j', k}\Big) \leq o_{\gamma \to 0; \delta}(\gamma^2)d_{\gamma, \ell - 1}\sqrt{\beta / \alpha}\,.
\end{equation}
The approximation of $\gain^\star_{n, \ell, 2, j, j', k}$ with $\gain^{\star, \star}_{n, \ell, 2, j, j', k}$ in \eqref{eq:expr_total_weight7} is rather
straightforward. Denote by $J_{\ell, 2}$ the set of all pairs $(j, j') \in [2] \times [\Gamma_{\ell, \beta}]$ such that the heights $\nu_{B_{\ell, 2, I, k}}$'s are not same for all $I \subseteq I_{\ell; j, j'}$ and $k \in [2]$ (see the discussion immediately preceding 
\eqref{eq:expr_total_weight7}). Since $|J_{\ell, 2}| = O_{\delta}(1)$, 
\begin{equation}\tag{E5}
\label{eq:error_bound_gain_star_star}
\E \Big( \sum_{(j, j') \in J_{\ell, 2}, k \in [2]}|\gain^\star_{n, \ell, 2, j, j', k} - \gain^{\star, \star}_{n, \ell, 2, j, j', k}|\Big) \leq O_{\delta}(1)\gamma\frac{d_{\gamma, \ell - 1}}{\Gamma} = O_{\delta}(1)d_{\gamma, \ell - 1}\gamma^3\,.
\end{equation}

Next we will deal with the error terms arising from the approximation of $\gain^{\star \star}_{n, \ell, i, j, j', k}$ with $\widetilde \gain^{\star \star}_{n, \ell, 
i, j, j', k}$. Let $\Delta \widetilde \gain^{\star, \star}_{n, \ell, i, j, j', k} = \widetilde\gain^{\star \star}_{n, \ell, i, j, j', k} -  \gain^{\star \star}_{n, \ell, i, j, j', k}$. Recall from Subsection~\ref{subsec:induct_step_hard} that
\begin{equation}
\label{eq:err_bnd2}
\Delta \widetilde \gain^{\star, \star}_{n, \ell, i, j, j', k} = \frac{\gamma d_{\gamma, \ell - 1}}{\Gamma a_{\ell - 1}}\sum_{v \in \partial \tilde V_{\ell; i, j}^\Gamma} (\tilde c_{v; \ell, i, j, j', k} - c_{v; \ell, i, j, j', k})\eta_{n, \ell, v}\,,
\end{equation}
where $c_{v; \ell, i, j, j', k} = \sum_{w \in I_{\ell; j, j'} \times \{\nu_{2, j, j', k}\}} H_{\tilde V_{\ell; i, j}^\Gamma}(w, v)$ and $\tilde c_{v; \ell, i, j, j', k}$ is the coefficient we substituted for $c_{v; \ell, 
i, j, j', k}$. 
In Subsection~\ref{subsec:induct_step_hard} we described 
different stages of the approximation for $i = 2$. The case for $i=1$ is similar, so its discuss is omitted.  The key part is the computation of $\var (\Delta \widetilde \gain^{\star, \star}_{n, \ell, 2, j, j', k})$ and $\cov(\Delta \widetilde \gain^{\star, \star}_{n, \ell, 2, j, j', k},$ $\Delta \widetilde \gain^{\star, \star}_{n, \ell, 2, j, j'', k})$, which we carry out in our next lemma.
\begin{lemma}
\label{lem:involved_var_covar}
For all $j \in [2], k \in [2]$ and $j' \in [\Gamma_{\ell, \beta}]$,
$$\var (\Delta \widetilde \gain^{\star, \star}_{n, \ell, 2, j, j', k}) = O(\gamma^2)\frac{d_{\gamma, \ell -1}}{\Gamma^2} \beta\delta 
\,.$$
Also for $j' < j'' \in [\Gamma_{\ell, \beta}]$,
$$\cov(\Delta \widetilde \gain^{\star, \star}_{n, \ell, 2, j, j', k}, \Delta \widetilde \gain^{\star, \star}_{n, \ell, 2, j, j'', k}) \leq  O(\gamma^2)\frac{d_{\gamma, \ell - 1}^2}{\Gamma^2}\e^{-\Omega(\beta)(|j'' - j'| - 1)}\log \beta \,.$$
\end{lemma}
\begin{proof}
First let us compute $\cov(\Delta \widetilde \gain^{\star, \star}_{n, \ell, 2, j, j', k}, \Delta\widetilde\gain^{\star, \star}_{n, \ell, 2, j, j'', k})$ which is relatively easier. We can expand this as
$$\cov(\Delta \widetilde \gain^{\star, \star}_{n, \ell, 2, j, j', k}, \Delta\widetilde\gain^{\star, \star}_{n, \ell, 2, j, j'', k}) = \cov_{1, 1} + \cov_{2, 2} - \cov_{1, 2} - \cov_{2,1}\,,$$
where $\cov_{1, 1} = \cov(\widetilde \gain^{\star, \star}_{n, \ell, 2, j, j', k}, \widetilde \gain^{\star, \star}_{n, \ell, 2, j, j'', k})$,  $\cov_{2, 2} = \cov(\gain^{\star, \star}_{n, \ell, 2, j, j', k}, \gain^{\star, \star}_{n, \ell, 2, j, j'', k})$, $\cov_{1, 2} = \cov(\widetilde \gain^{\star, \star}_{n, \ell, 2, j, j', k}, \gain^{\star, \star}_{n, \ell, 2, j, j'', k})$ and $\cov_{2, 1} = \cov(\gain^{\star, \star}_{n, \ell, 2, j, j', k}, \widetilde \gain^{\star, 
\star}_{n, \ell, 2, j, j'', k})$. Since $\widetilde \gain^{\star, \star}_{n, \ell, 2, j, j', k}$ and $\gain^{\star, \star}_{n, \ell, 2, j, j', k}$ consist entirely of positively correlated Gaussian variables it suffices 
to bound $\cov_{1, 1}$ and $\cov_{2, 2}$. Now notice that
$$\cov_{1, 1} \leq \frac{d_{\gamma, \ell - 1}^2\gamma^2}{\Gamma^2 a_{\ell - 1}^2}\cov\big(\sum_{v \in I_{\ell; j, j'} \times \{\nu_{2, j, j', k}\}}\eta_{n, \ell, v}, \sum_{w \in I_{\ell; j, j''} \times \{\nu_{2, j, j'', k}\}}\eta_{n, \ell, w}\big)\,.$$
Lemma~\ref{lem:general_covar} gives that the last covariance is bounded by $O(\gamma^2)\log 
\beta\e^{-\Omega(\beta)(|j'' - j'| - 1)}a_{\ell - 1}^2$. Thus,
$$\cov_{1, 1} \leq O(\gamma^2)\frac{d_{\gamma, \ell - 1}^2}{\Gamma^2}\e^{-\Omega(\beta)(|j'' - j'| - 1)}\log 
\beta \,.$$
One can bound $\cov_{2, 2}$ in exactly similar way. Hence,
$$\cov(\Delta \widetilde \gain^{\star, \star}_{n, \ell, 2, j, j', k}, \Delta\widetilde\gain^{\star, \star}_{n, \ell, 2, j, j'', k}) \leq O(\gamma^2)\frac{d_{\gamma, \ell - 1}^2}{\Gamma^2}\e^{-\Omega(\beta)(|j'' - j'| - 1)}\log \beta \,.$$

The variance computation requires some extra effort. 
To this end divide the base of $\tilde V_{\ell; i, j}^\Gamma$ into several disjoint intervals as shown in 
Figure~\ref{fig:covariance_computation}. Let us also define an additional interval $I_{j, j', 0; \beta}$ which is obtained by adding a segment of length $50 \log \beta a_{\ell - 1}$ to each end of $I_{j, j', 0}$.
\begin{figure}[!htb]
\centering
\begin{tikzpicture}[semithick, scale = 4]
\draw (-2.2, 0) rectangle (-0.5, 0.3);

\draw [dashed] (-1.9, 0) -- (-1.9, 0.3);
\node [scale = 0.6] at (-2.05, 0.15) {$I_{j, j', -1}$};
\draw [<->, thin] (-2.17, 0.35) -- (-1.93, 0.35);
\node [scale = 0.6, above] at (-2.05, 0.35) {$100\log \beta a_{\ell - 1}$};

\draw [dashed] (-0.8, 0) -- (-0.8, 0.3);
\node [scale = 0.6] at (-0.65, 0.15) {$I_{j, j', 1}$};
\draw [<->, thin] (-0.77, 0.35) -- (-0.53, 0.35);
\node [scale = 0.6, above] at (-0.65, 0.35) {$100\log \beta a_{\ell - 1}$};

\node [scale = 0.6] at (-1.35, 0.15) {$I_{j, j', 0}$};


\draw [dashed] (-2.5, 0) rectangle (-2.2, 0.3);
\node [scale = 0.6] at (-2.35, 0.15) {$I_{j, j', -2}$};
\draw [<->, thin] (-2.47, -0.05) -- (-2.23, -0.05);
\node [scale = 0.6, below] at (-2.35, -0.05) {$100\log \beta a_{\ell - 1}$};

\draw [dashed] (-0.2, 0) rectangle (-0.5, 0.3);
\node [scale = 0.6] at (-0.35, 0.15) {$I_{j, j', 2}$};
\draw [<->, thin] (-0.47, -0.05) -- (-0.23, -0.05);
\node [scale = 0.6, below] at (-0.35, -0.05) {$100\log \beta a_{\ell - 1}$};

\draw [dashed] (-0.2, 0) rectangle (0.4, 0.3);
\node [scale = 0.6] at (0.1, 0.15) {$I_{j, j',3}$};

\draw [dashed] (-2.5, 0) rectangle (-3.1, 0.3);
\node [scale = 0.6] at (-2.8, 0.15) {$I_{j, j',-3}$};




\draw [decorate,decoration={brace,amplitude=10pt, mirror, raise=4pt},yshift=0pt](-2.2,-0.1) -- (-0.5,-0.1); 

\node [scale = 0.6, below] at (-1.35, -0.3) {$I_{\ell;j,j'}$};
\end{tikzpicture}
\caption{{\bf Partitioning the base of $\tilde V^\Gamma_{\ell; i, j}$ based on the distance from $I_{\ell; j, j'}$.} The largest rectangle indicates $\tilde V^\Gamma_{\ell; i, j}$. Each sub-rectangle (one which has no further subdivisions) is based on $I_{j, j', q}$ for some $-3 \leq q \leq 3$ which is indicated 
by the label at its center. The numbers next to the arrows indicate the lengths of corresponding intervals.}
\label{fig:covariance_computation}
\end{figure}
From Lemma~\ref{lem:reversibility} we know that $c_{v; \ell, 2, j, j', k} = \tfrac{1}{4}G_{\tilde V^\Gamma_{\ell; 2, j}}(v_{2, j}, I_{\ell; j, j'} \times \{\nu_{2, j, j', k}\})$ where $v_{2, j'}$ is the unique neighbor of $v$ in $\inte(\tilde V^\Gamma_{\ell; 2, 
j})$. We chose $\tilde c_{v; \ell, 2, j, j', k}$ to be $\tfrac{\nu_{2, j, j', k} + \lfloor a_{\ell - m - 1} \rfloor}{\vertical_{\ell - 1}}$ for $v \in I_{\ell; j, j'} \times \{\lfloor a_\ell\rfloor + \lfloor a_{\ell - m 
- 1} \rfloor \}$ and 0 for all remaining $v$. As already discussed in the previous subsection, the first estimate is essentially $\tfrac{1}{4}G^\star(v_y - 1, \nu_{2, j, j', k}) \equiv \tfrac{1}{4}G^{1, \star}_{[-a_{\ell - m - 1}, \lfloor a_\ell \rfloor + 
a_{\ell - m - 1}]\cap \Z}(v_y - 1, \nu_{2, j, j', k})$.
This approximation is rather good when $v_x \in I_{j, j', 
0}$. To see this we first write $\tilde c_{v; \ell, i, j, j', k}$ for any such point $v$ as 
\begin{equation}
\label{eq:involved_var_covar1}
\tilde c_{v; \ell, 2, j, j', k} = \frac{1}{4}\big(\P^{v_{2, j}}(\tau_{\tilde V_{\ell; 2, j}^\Gamma} > \tau_{A_{\ell, \down}}) + \P^{v_{2, j}}(\tau_{\tilde V_{\ell; 2, j}^\Gamma} \leq  \tau_{A_{\ell, \down}}) \big)G^{\star}(\nu_{2, j, j', k}, \nu_{2, j, j', k})\,,
\end{equation}
where $A_{\ell, \down}$ is the strip $\{(x, y): x \in \Z, y \in [-a_{\ell - m - 1}, \lfloor a_\ell \rfloor + 
a_{\ell - m - 1}]\cap \Z\}$. From Lemma~\ref{lem:kernel_bound2} we know that $\P^{v_{2, j}}(\tau_{\tilde V_{\ell; 2, j}^\Gamma} \leq \tau_{A_{\ell, \down}}) \leq O(a_{\ell - 1}^{-1})\e^{-\Theta(\Gamma \delta)}$. Thus,
\begin{equation}
\label{eq:involved_var_covar2}
\tilde c_{v; \ell, 2, j, j', k} \leq \frac{1}{4}\P^{v_{2, j}}(\tau_{\tilde V_{\ell; 2, j}^\Gamma} > \tau_{A_{\ell, \down}})G^{\star}(\nu_{2, j, j', k}, \nu_{2, j, j', k}) + O(1)\e^{-\Theta(\Gamma \delta)}\frac{\nu_{2, j, j', k} + \lfloor a_{\ell - m - 1}\rfloor}{\vertical_{\ell - 1}}\,.
\end{equation}
We can further split the first summation into two parts as follows:
\begin{equation}
\label{eq:involved_var_covar3}
\frac{1}{4}\sum_{u \in \Hl_{\nu_{2, j, j', k}}}\P^{v_{2, j}}(S_{\tau_{A_{\ell, \down}}} = u, \tau_{\tilde V_{\ell; 2, j}^\Gamma} > \tau_{A_{\ell, \down}})G^{\star}(\nu_{2, j, j', k}, \nu_{2, j, j', k}) = \Sigma_1 + \Sigma_2\,.
\end{equation} 
where $\Sigma_1, \Sigma_2$ contain the terms corresponding to $u_x \in I_{j, j', 0; \beta}$ and 
$I_{j, j', 0; \beta}^c$ respectively. Similarly we can write $c_{v; \ell, i, j, j', k}$ as
\begin{eqnarray}
\label{eq:involved_var_covar4}
c_{v; \ell, 2, j, j', k} &=& \frac{1}{4}\sum_{u \in \Hl_{\nu_{2, j, j', k}}}\P^{v_{2, j}}(S_{\tau_{A_{\ell, \down}}} = u, \tau_{\tilde V_{\ell; 2, j}^\Gamma} > \tau_{A_{\ell, \down}})G_{\tilde V_{\ell; 2, j}^\Gamma}(u, I_{\ell; j, j'} \times \{\nu_{2, j, j', k}\}) \nonumber \\
&=& \Sigma_1' + \Sigma_2'\,,
\end{eqnarray} 
where $\Sigma_1', \Sigma_2'$ respectively contain the terms corresponding to $u_x \in I_{j, j', 0; \beta}$ and 
$u_x \in I_{j, j', 0; \beta}^c$. Since $G_{\tilde V_{\ell; 2, j}^\Gamma}(u, I_{\ell; j, j'} \times \{\nu_{2, j, j', k}\})$ is bounded by $G^{\star}(\nu_{2, j, j', k}, \nu_{2, j, j', k})$ for all $u \in \Z^2$, it follows that $\Sigma_1 \geq \Sigma_1'$ and $\Sigma_2 
\geq \Sigma_2'$. In addition by \eqref{eq:general_covar3}, we have that
$$G^{\star}(\nu_{2, j, j', k}, \nu_{2, j, j', k}) - G_{\tilde V_{\ell; 2, j}^\Gamma}(u, I_{\ell; j, j'} \times \{\nu_{2, j, j', k}\}) = 
O(\beta^{-20})G^{\star}(\nu_{2, j, j', k}, \nu_{2, j, j', k})\,,$$
for all $u \in I_{j, j', 0, \beta} \times \{\nu_{2, j, 
j', k}\}$. Thus $\Sigma_1 - \Sigma_1' = 
O(\beta^{-20}) \tfrac{\nu_{2, j, j', k} + \lfloor 
a_{\ell - m - 1}\rfloor}{\vertical_{\ell - 1}}$. Also Lemma~\ref{lem:kernel_bound2} implies 
$$\P^{v_{2, j}}(S_{\tau_{A_{\ell, \down}}} \in I_{j, j', 0; \beta} \times \{\nu_{2, j, j', k}\}) = O(\beta^{-20})a_{\ell - 1}^{-1}\,.$$
This implies $\Sigma_2 - \Sigma_2' \leq O(\beta^{-20})\tfrac{\nu_{2, j, j', k} + \lfloor 
a_{\ell - m - 1}\rfloor}{\vertical_{\ell - 1}}$. Altogether, we get 
\begin{equation}
\label{eq:involved_var_covar5}
|\Delta \tilde c_{v; \ell, 2, j, j', k}| \equiv |\tilde c_{v; \ell, 2, j, j', k} - c_{v; \ell, 2, j, j', k}| = O(\beta^{-20})\,,
\end{equation}
for all $v \in I_{j, j', 0} \times \{\lfloor a_\ell 
\rfloor + \lfloor a_{\ell - m - 1}\rfloor \}$. From what we discussed it is also clear that $\tilde c_{v; \ell, 2,  j, j', k} \geq c_{v; \ell, 2,  j, j', k}$ for all $v \in I_{\ell; j, j'} \times \{\lfloor a_\ell \rfloor + 
\lfloor a_{\ell - m - 1}\rfloor\}$ and $c_{v; \ell, 2, j, j', k} \leq \tfrac{1}{4}G^{\star}(\lfloor a_{\ell}\rfloor + \lfloor a_{\ell - m - 1}\rfloor - 1, \nu_{2, j, j', k})$ for all $v \in \partial_\up \tilde 
V_{\ell; 2, j}^\Gamma$. Hence 
\begin{equation}
\label{eq:involved_var_covar6}
|\Delta \tilde c_{v; \ell, 2, j, j', k}| = O(1)\,,
\end{equation}
for $v \in  \cup_{q \in \{-1, -2, 1, 2\}}I_{j, j', q} \times \{\lfloor a_\ell \rfloor + \lfloor a_{\ell - m - 
1} \rfloor\}$. The same bound obviously holds for $v \in I_{\ell; j, j'}^\star \times \{-\lfloor a_{\ell - m - 
1}\rfloor\}$, where $I_{\ell; j, j'}^\star = I_{\ell; j, 
j'} \cup I_{j, j', -2} \cup I_{j, j', 2}$. Now suppose that $v_x$ lies outside the interval $I_{\ell; j, 
j'}^\star$. Drawing upon our previous discussion again we get,
$$c_{v; \ell, 2, j, j', k} \leq \frac{1}{4}\P^{v_{2, j}}(S_t\mbox{ reaches $I_{\ell; j, j'} \times \{\nu_{2, j, j', k}\}$ before $\tau_{\tilde V_{\ell; 2, j}^\Gamma}$})G^\star(\nu_{2, j, j', k}, \nu_{2, j, j', k})\,.$$
Let $p^\star_{j, j'}(v)$ be the endpoint of $I_{\ell; j, 
j'}$ that is nearest to $v_x$. Applying Lemma~\ref{lem:kernel_bound2} to bound the probability 
in the last expression, we obtain
\begin{equation}
\label{eq:involved_var_covar7}
|\Delta \tilde c_{v; \ell, 2, j, j', k}| = O(1)\e^{-\Theta\big(\frac{|v_x - p^\star_{j, j'}(v)|}{a_{\ell - 1}}\big)}\,,
\end{equation}
for all $v \in  (\partial_\up\tilde V^\Gamma_{\ell; 2, j} \cup \partial_\down \tilde V^\Gamma_{\ell; 2, j}) \setminus (I^\star_{\ell; j, j'} \times \{\lfloor a_\ell \rfloor + \lfloor a_{\ell - m - 1}\rfloor,  -\lfloor a_{\ell - m - 1}\rfloor\})$.
Finally let $v \in \partial_\mathrm{left}\tilde V^\Gamma_{\ell; 2, j}$ $\cup 
\partial_\mathrm{right}\tilde V^\Gamma_{\ell; 2, j}$. In this case we can directly apply Lemma~\ref{lem:kernel_bound3} to get
\begin{equation}
\label{eq:involved_var_covar8}
|\Delta \tilde c_{v; \ell, 2, j, j', k}| = O(1)\e^{-\Theta(\Gamma \delta)}\,.
\end{equation}

Using these bounds on $|\Delta\tilde c_{v; \ell, 2, j, 
j', k}|$ we can now estimate $\var(\Delta \widetilde 
\gain^{\star, \star}_{n, \ell, 2, j, j', k})$. We will split the terms in $\Delta \widetilde \gain^{\star, \star}_{n, \ell, 2, j, j', k}$ (see \eqref{eq:err_bnd2}) into several groups based on the particular segment of $\partial \tilde V^\Gamma_{\ell; 2, j}$ they correspond 
to and deal separately with each group. To this end,  for any $A \subseteq \partial \tilde V^\Gamma_{\ell; 2, j}$ we define the random variable $\Delta \widetilde \gain^{\star, \star}_A = \sum_{v \in A}\Delta \tilde c_{v; \ell, 2, j, j', k}\eta_{n, \ell, 
v}$. The first segment (or group) we consider is $I_{j, j', 0} \times \{\lfloor a_\ell \rfloor + \lfloor a_{\ell    
- m - 1}\rfloor \}$. From the definition of $\Delta \widetilde \gain^{\star, \star}_A$, we get that
\begin{equation}
\label{eq:involved_var_covar9}
\var(\Delta \widetilde \gain^{\star, \star}_{I_{j, j', 0} \times \{\lfloor a_\ell \rfloor + \lfloor a_{\ell - m - 1}\rfloor \}} ) \leq \frac{\gamma^2 d_{\gamma, \ell - 1}^2}{\Gamma^2 a_{\ell - 1}^2}\Delta \tilde C_{\ell, 2, j, j', k, 0}^2\var\big(\sum_{v \in I_{j, j', 0}\times \{\lfloor a_\ell \rfloor + \lfloor a_{\ell - m - 1}\rfloor \}}\eta_{n, \ell, v}\big)\,.
\end{equation}
Here $\Delta \tilde C_{\ell, 2, j, j', k, q} = \max_{v \in I_{j, j', q} \times \{\lfloor a_\ell \rfloor + \lfloor a_{\ell - m - 1} \rfloor\}}|\Delta \tilde c_{v; 
\ell, 2, j, j', k}|$. We will refer to this expression 
for some other cases as well. By \eqref{eq:involved_var_covar5}, we have that $\Delta \tilde C_{\ell, 2, j, j', k, 0}^2 = O(\beta^{-40})$.
Also Lemma~\ref{lem:general_covar} gives 
$$\var\big(\sum_{v \in I_{j, j', 0}\times \{\lfloor a_\ell \rfloor + \lfloor a_{\ell - m - 1}\rfloor \}}\eta_{n, \ell, v}\big) = O(\beta)a_{\ell - 1}^2\,.$$
Thus, we get that
\begin{equation*}
\var(\Delta \widetilde \gain^{\star, \star}_{I_{j, j', 0} \times \{\lfloor a_\ell \rfloor + \lfloor a_{\ell - m - 1}\rfloor \}}) = O(\gamma^2)\frac{d_{\gamma, \ell - 1}^2}{\Gamma^2}\beta^{-39}\,.
\end{equation*}
Following a similar derivation we obtain
\begin{equation*}
\var(\Delta \widetilde \gain^{\star, \star}_{I_{j, j', q} \times \{\lfloor a_\ell \rfloor + \lfloor a_{\ell - m - 1}\rfloor \}}) = O(\gamma^2)\frac{d_{\gamma, \ell - 1}^2}{\Gamma^2}\log \beta\,,
\end{equation*}
for $q \in \{-2, -1, 1, 2\}$ and
\begin{equation*}
\var(\Delta \widetilde \gain^{\star, \star}_{\partial_{\mathrm{left}}\tilde V^\Gamma_{\ell; 2, j} \cup \partial_{\mathrm{right}}\tilde V^\Gamma_{\ell; 2, j}}) = O(\gamma^2)\frac{d_{\gamma, \ell - 1}^2}{\Gamma^2}\e^{-\Theta(\Gamma \delta)}\,.
\end{equation*}
The next segment we are going to consider is $I_{j, j', 3} \times \{\lfloor a_\ell \rfloor + \lfloor a_{\ell - m 
- 1}\rfloor \}$. In this case we need a slightly more refined bound on the variance than 
\eqref{eq:involved_var_covar9}. By \eqref{eq:involved_var_covar7}, we can write
\begin{equation*}
\var(\Delta \widetilde \gain^{\star, \star}_{I_{j, j', 3} \times \{\lfloor a_\ell \rfloor + \lfloor a_{\ell - m - 1}\rfloor \}} ) \leq \frac{\gamma^2 d_{\gamma, \ell - 1}^2}{\Gamma^2 a_{\ell - 1}^2}\e^{-\Theta\big(\frac{|v_x - p^\star_{j, j'}(v)|}{a_{\ell - 1}}\big)}O(\beta^{-20})\sum_{v \in I_{j, j', 3}\times \{\lfloor a_\ell \rfloor + \lfloor a_{\ell - m - 1}\rfloor \}}\cov_v\,,
\end{equation*}
where
$$\cov_v = \cov(\eta_{n, \ell, v}, \mbox{$\sum_{w \in I_{j, j', 3}\times \{\lfloor a_\ell \rfloor + \lfloor a_{\ell - m - 1}\rfloor \}}$}\eta_{n, \ell, w})\,.$$
Since $\cov_v = O(a_{\ell - 1})$ by Lemma~\ref{lem:general_covar}, we get that
\begin{equation*}
\var(\Delta \widetilde \gain^{\star, \star}_{I_{j, j', 3} \times \{\lfloor a_\ell \rfloor + \lfloor a_{\ell - m - 1}\rfloor \}} ) = O(\gamma^2)\frac{d_{\gamma, \ell - 1}^2}{\Gamma^2}\beta^{-40}\,.
\end{equation*}
Similarly we can bound the variances of $\Delta \widetilde \gain^{\star, \star}_{I_{j, j', -3} \times \{\lfloor a_\ell \rfloor + \lfloor a_{\ell - m - 1}\rfloor \}}$, $\Delta \widetilde \gain^{\star, \star}_{I_{j, j', 3} \times \{-\lfloor a_{\ell - m - 1}\rfloor \}}$ and $\Delta \widetilde \gain^{\star, \star}_{I_{j, j', -3} \times \{-\lfloor a_{\ell - m - 
1}\rfloor \}}$. The last segment we will deal with is $I^\star_{\ell; j, j'} \times \{-\lfloor a_{\ell - m - 
1}\rfloor \}$. From \eqref{eq:involved_var_covar9} we have
\begin{equation*}
\var(\Delta \widetilde \gain^{\star, \star}_{I^\star_{\ell; j, j'} \times \{\lfloor a_\ell \rfloor + \lfloor a_{\ell - m - 1}\rfloor \}} ) \leq \frac{\gamma^2d_{\gamma, \ell - 1}^2}{\Gamma^2a_{\ell - 1}^2}\var\big(\sum_{v \in I^\star_{j, j', 0} \times \{-\lfloor a_{\ell - m - 1}\rfloor\}}\eta_{n, \ell, v}\big)\,,
\end{equation*}
where we have used the fact that $|\Delta\tilde c_{v; \ell, 2, j, j', k}| = O(1)$ for all $v \in I^\star_{\ell; j, j'} \times \{\lfloor a_\ell \rfloor + 
\lfloor a_{\ell - m - 1}\rfloor \}$. Now Lemma~\ref{lem:general_covar} gives us
$$\var\big(\sum_{v \in I^\star_{j, j', 0} \times \{-\lfloor a_{\ell - m - 1}\rfloor\}}\eta_{n, \ell, v}\big) = O(\beta \delta)a_{\ell - 1}^2\,.$$
Combining the last two displays we get
\begin{equation*}
\var(\Delta \widetilde \gain^{\star, \star}_{I_{j, j', 3} \times \{-\lfloor a_{\ell - m - 1}\rfloor \}} ) = O(\gamma^2)\frac{d_{\gamma, \ell - 1}^2}{\Gamma^2}\beta \delta\,.
\end{equation*}
Since the segments we discussed form a partition of $\partial \tilde V^\Gamma_{\ell; 2, j}$, we can conclude
\begin{equation*}
\var(\Delta \widetilde \gain^{\star, \star}_{n, \ell, 2, j, j', k} ) = O(\gamma^2)\frac{d_{\gamma, \ell - 1}^2}{\Gamma^2}(\beta \delta + \log \beta)\,. \qedhere
\end{equation*}
\end{proof}
By Lemma~\ref{lem:involved_var_covar}, we get that
$$\var \Big(\sum_{j' \in [j_1', j_2']}\Delta \widetilde \gain^{\star, \star}_{n, \ell, i, j, j', k}\Big) = O(\gamma^2)\frac{d_{\gamma, \ell - 1}^2}{\Gamma^2}(j_2' - j_1' + 1)\beta \delta 
\,.$$
From Lemma~\ref{lem:basic_chaining} it then follows
\begin{equation}\tag{E6}
\label{eq:error_bound_Delta_widetilde_gain_star2}
\E \Big(\max_{j_1', j_2' \in [\Gamma_{\ell, \beta}]} \sum_{j' \in [j_1', j_2']}\Delta \widetilde \gain^{\star, \star}_{n, \ell, i, j, j', k}\Big) \leq  O(\gamma^2)d_{\gamma, \ell - 1}\sqrt{\frac{\delta}{\alpha}} 
\,.
\end{equation}

Now let us analyze the approximation of $\widetilde \gain^{\star, \star}_{n, \ell, i, j, j', k}$ with 
$\widetilde \gain_{n, \ell, i, j, j', k}$. Define $$\Delta\widetilde \gain_{n, \ell, i, j, j', k} = \widetilde \gain_{n, \ell, i, j, j', k} - \widetilde \gain^{\star, \star}_{n, \ell, i, j, j', k}\,.$$
To avoid clumsy notations we will discuss the case $i = 
2$ only. Notice that
$$\Delta\widetilde \gain_{n, \ell, 2, j, j', k} = \frac{d_{\gamma, \ell - 1}}{\Gamma a_{\ell - 1}}\sum_{j \in [2], j' \in [\Gamma_{\ell, \beta}]}\gamma\frac{\nu_{2, j, j', k} + \lfloor a_{\ell - m - 1}\rfloor}{\vertical_{\ell - 1}}\Delta \eta_{n, \ell, j, j'}\,,$$
where $\Delta \eta_{n, \ell, j, j'} = \sum_{v \in I_{\ell; j, j'} \times \{\lfloor a_{\ell}\rfloor + \lfloor a_{\ell - m - 1}\rfloor\}}\tfrac{\eta_{n, \ell, 
\overline{v}} - \eta_{n, \ell, v}}{2}$. Recall from Subsection~\ref{subsec:induct_step_hard} that $\overline v = (v_x, p_\ell)$ where $p_\ell$ left endpoint of the span 
of $\tilde V^\Gamma_{\ell; 1, 1}$. Our next lemma provides upper bounds on the variance and covariance of $\Delta\widetilde \gain_{n, \ell, 2, j, j', k}$'s.
\begin{lemma}
\label{lem:symmetrization}
For all $j \in [2], k \in [2]$ and $j' \in [\Gamma_{\ell, \beta}]$,
$$\var (\Delta \widetilde \gain_{n, \ell, 2, j, j', k}) = O(\gamma^2)\frac{d_{\gamma, \ell -1}^2}{\Gamma^2}\beta\delta\,.$$
Also for $j' < j'' \in [\Gamma_{\ell, \beta}]$,
$$\cov(\Delta \widetilde \gain_{n, \ell, 2, j, j', k}, \Delta \widetilde \gain_{n, \ell, 2, j, j'', k}) \leq  O(\gamma^2)\frac{d_{\gamma, \ell - 1}^2}{\Gamma^2}\e^{-\Omega(\beta)(j'' - j' - 1)}\log \beta \,.$$
\end{lemma}
\begin{proof}
For convenience we will denote $\lfloor a_\ell \rfloor + \lfloor a_{\ell - m - 1}\rfloor$ by $r_\ell$. In order to bound the variance we need a lower bound on the covariance between $\overline \eta_{n, \ell, j, j', 1}$ and $\overline \eta_{n, \ell, j, j', 2}$ where 
$$\overline \eta_{n, \ell, j, j', 1} = \sum_{v \in I_{\ell; j, j'} \times \{r_\ell\}}\frac{\eta_{n, \ell, 
v}}{2}\,,\mbox{ and \hspace{0.2cm}}\overline \eta_{n, \ell, j, j', 2} = \sum_{w \in I_{\ell; j, j'} \times \{p_\ell\}}\frac{\eta_{n, \ell, w}}{2}\,.$$
Notice that $\cov(\eta_{n, \ell, w}, \overline \eta_{n, \ell, j, j', 1}) = G_{\tilde V^\Gamma_{\ell; 2, j'}}(\overline{v},  I_{\ell; j, j'} \times \{r_\ell\})$ for all $w \in 
I_{\ell; j, j'} \times \{p_\ell\}$. Since $\beta$ and $\Gamma$ are large, a good estimate for this covariance is the one dimensional lazy random walk Green's function $G^\star(p_\ell, r_\ell) = G^{1, \star}_{[-a_{\ell - m}, \lfloor a_{\ell+1} \rfloor + a_{\ell - m}] \cap \Z}(p_\ell, r_\ell)$. 
Observing $p_\ell - r_\ell = O(\delta)$, we can then see that this is $(1 - O(\delta))G^\star(r_\ell, r_\ell)$. 
Now recall that we did a similar thing while approximating $\gain^{\star, \star}_{n, \ell, 1, j, j', k}$ by $\widetilde \gain^{\star, \star}_{n, \ell, 1, j, j', 
k}$. In fact an analogous derivation for \eqref{eq:involved_var_covar5}  (with minor modifications) gives us
$$\cov(\eta_{n, \ell, w}, \overline \eta_{n, \ell, j, j', 1}) \geq (1 - O(\delta) - O(\beta^{-20}))G^\star(r_\ell, r_\ell)\,.$$
Exchanging the roles of $r_\ell$ and $p_\ell$, we also get 
$$\cov(\eta_{n, \ell, v}, \overline \eta_{n, \ell, j, j', 2}) \geq (1 - O(\delta) - O(\beta^{-20}))G^\star(p_\ell, p_\ell)\,.$$
for all $v \in I_{\ell; j, j'} \times \{r_\ell\}$. On the hand we know that
$\var (\overline \eta_{n, \ell, j, j', 1}) \leq |I_{\ell;j, j'}|G^\star(p_\ell, p_\ell)$ and $\var (\overline \eta_{n, \ell, j, j', 2}) \leq |I_{\ell;j, j'}|G^\star(r_\ell, r_\ell)$. Thus
$$\var (\Delta\eta_{n, \ell, j, j'}) \leq (O(\delta) + O(\beta^{-20}))|I_{\ell; j, j'}|(G^\star(p_\ell, p_\ell) + G^\star(r_\ell, r_\ell)) = (O(\delta) + O(\beta^{-20}))\beta a_{\ell - 1}^2\,.$$
This proves the variance part. Notice that this argument and Lemma~\ref{lem:general_covar} together imply
\begin{equation}
\label{eq:symmetrization1}
\var(\eta_{n, \ell, j, j'}) \geq (1 - O(\delta))\beta a_{\ell - 1}\beta a_{\ell - 1}a_{\ell + 1}\,.
\end{equation}
The covariance part follows from Lemma~\ref{lem:general_covar}.
\end{proof}
From Lemma~\ref{lem:symmetrization} we get
$$\var \Big(\sum_{j' \in [j_1', j_2']}\Delta \widetilde \gain_{n, \ell, i, j, j', k}\Big) = O(\gamma^2)\frac{d_{\gamma, \ell - 1}^2}{\Gamma^2}(j_2' - j_1' + 1)\beta \delta 
\,.$$
From Lemma~\ref{lem:basic_chaining} it then follows
\begin{equation}\tag{E7}
\label{eq:error_bound_Delta_widetilde_gain}
\E \Big(\max_{j_1', j_2' \in [\Gamma_{\ell, \beta}]} \sum_{j' \in [j_1', j_2']}\Delta \widetilde \gain_{n, \ell, i, j, j', k}\Big) \leq  O(\gamma^2)d_{\gamma, \ell - 1}\sqrt{\frac{\delta}{\alpha}} 
\,.
\end{equation}
Finally we need to address the error terms arising from the diagonalization of $\eta_{n, 
\ell, i, j, j'}$'s (see the paragraph following \eqref{eq:expr_total_weight10}). We discuss the case for $i = 2$ only and the other case is similar. Define
$$\Delta\gram_{n, \ell, 2, j, j', k} = \frac{d_{\gamma, \ell-1}}{\Gamma a_{\ell - 1}}\sum_{j \in [2], j' \in [\Gamma_{\ell, \beta}]}\gamma \frac{\nu_{2, j, j', k} + \lfloor a_{\ell - m - 1}\rfloor}{\vertical_{\ell - 1}}(\tilde \eta_{n, \ell, j, j'} - \eta_{n, \ell, j, j'})\,.$$
From Lemma~\ref{lem:Gram_Schimidt} and Lemma~\ref{lem:general_covar} we immediately get that
$$\var (\Delta\gram_{n, \ell, 2, j, j', k}) = O(\gamma^2)\frac{d_{\gamma, \ell - 1}^2}{\Gamma^2}\frac{(\log \beta)^2}{\beta}\mbox{ for }j' \in [\Gamma_{\ell, \beta}]\,,$$
and 
$$\cov (\Delta\gram_{n, \ell, 2, j, j', k}, \Delta\gram_{n, \ell, 2, j, j'', k}) \leq  O(\gamma^2)\frac{d_{\gamma, \ell - 1}^2}{\Gamma^2}\e^{-\Omega(\beta)(j'' - j')}\frac{(\log \beta)^2}{\beta}\mbox{ for }j' < j''\in [\Gamma_{\ell, \beta}]\,.$$
These two displays imply that
$$\var \Big(\sum_{j' \in [j_1', j_2']}\Delta\gram^{\star, \star}_{n, \ell, i, j, j', k}\Big) = O(\gamma^2)\frac{d_{\gamma, \ell - 1}^2}{\Gamma^2}(j_2' - j_1' + 1)\frac{(\log \beta)^2}{\beta}\,.$$
From Lemma~\ref{lem:basic_chaining} it then follows
\begin{equation}\tag{E8}
\label{eq:error_bound_diagonalize}
\E \Big(\max_{j_1', j_2' \in [\Gamma_{\ell, \beta}]} \sum_{j' \in [j_1', j_2']}\Delta \gram_{n, \ell, i, j, j', k}\Big) \leq  O(\gamma^2)d_{\gamma, \ell - 1}\frac{\log \beta}{\beta \sqrt{\alpha}}\,.
\end{equation}


\subsection{Induction step for the hard case: verifying the remaining hypotheses} \label{subsec:hard-remaining}
In this subsection we will verify the hypotheses \eqref{hypo:expected_weight}, \eqref{hypo:short_range_ratio} and \eqref{hypo:gadget_negligible}.

Let us begin with \eqref{hypo:expected_weight}. We will first estimate the expected gain from switchings as 
given by \eqref{eq:objective_fxn2}. From definition of $\lambda$ in Subsection~\ref{subsec:induct_step_hard} it follows:
\begin{equation*}
\E(\int_{[0,T_{2, 1, \gamma, \ell}]} \frac{1}{\lambda(t)} d t | \mathcal F_{\ell - 1}) \geq (1 - O(\delta)) \frac{d_{\gamma, \ell -1}\gamma^2}{8\Gamma a_{\ell - 1}^2}\sum_{j' \in [\Gamma_{\ell, \beta}]}\Delta \tilde \nu_{2, 1, j'}\var(\tilde \eta_{n, \ell, j'})\,.
\end{equation*}
We saw in \eqref{eq:symmetrization1} that $\var(\eta_{n, \ell, j'}) \geq (1 -  
O(\delta))\beta a_{\ell - 1} a_{\ell + 1}$. Since $\var(\tilde \eta_{n, \ell, j'}) \geq (1 - O(\beta^{-2}))\var(\eta_{n, \ell, j'})$, we therefore have
\begin{equation*}
\E(\int_{[0,T_{2, 1, \gamma, \ell}]} \frac{1}{\lambda(t)} d t | \mathcal F_{\ell - 1}) \geq 0.9998\frac{d_{\gamma, \ell -1}\beta \gamma^2 a_{\ell + 1}}{8\Gamma a_{\ell - 1}}\sum_{j' \in [\Gamma_{\ell, \beta}]}\Delta \tilde \nu_{2, 1, j'}\,.
\end{equation*}
On the other hand \eqref{eq:lambda*_bound} and the definition of $\lambda$ imply that
\begin{equation*}
\lambda_\infty \lambda_*^{-1.5} \geq d_{\gamma, \ell - 1}\gamma^2O(\alpha^{-0.25})\,.
\end{equation*}
Plugging the last two bounds into \eqref{eq:optimization_BM} and taking expectations on both sides we get
\begin{equation*}
\E (\Phi_{\lambda, \mathcal Q^*_{\ell, 2, 1}} (S_{.,\ell})) \geq 0.9998\frac{d_{\gamma, \ell -1}\beta \gamma^2 a_{\ell + 1}}{8\Gamma a_{\ell - 1}}\sum_{j' \in [\Gamma_{\ell, \beta}]}\E \Delta \tilde \nu_{2, 1, j'} - d_{\gamma, \ell - 1}\gamma^2 O(\alpha^{-0.25})\,.
\end{equation*}
Now from \eqref{hypo:symmetry} we have that $\E \Delta \tilde \nu_{2, 1, j'} \geq (1 - O(\delta))\tfrac{1}{2}$. 
Thus
\begin{equation}
\label{eq:optimization_BM2}
\E (\Phi_{\lambda, \mathcal Q^*_{\ell, 2, 1}} (S_{.,\ell})) \geq 0.9997d_{\gamma, \ell -1}\frac{\gamma^2}{4}\,.
\end{equation} 

Due to our particular choice of strategy we have the following bound on the difference between $\E \mathcal I_\ell(k_{2,1,1}, \ldots, k_{2,1, [\Gamma_{\ell, \beta}]})$ and $\E \Phi_{\lambda, \mathcal Q^*_{\ell, 2, 1}}(S_{., \ell})$:
$$\E \mathcal I_\ell(k_{2,1,1}, \ldots, k_{2,1, [\Gamma_{\ell, \beta}]} | \mathcal F_{\ell - 1}) - \E \Phi_{\lambda, \mathcal Q^*_{\ell, 2, 1}}(S_{., \ell} | \mathcal F_{\ell - 1}) \leq 2\E (M_\mathrm{dis}| \mathcal F_{\ell - 1}) / \lambda_*^2\,,$$
where $M_{\mathrm{dis}} = \max_{j' \in [\Gamma_{\ell, \beta}]}\max_{s, t \in [g_{2, 1, \gamma, j' - 1}, g_{2, 
1, \gamma, j'}]}|S_{t; \ell} - S_{s; \ell}|$. From Lemma~\ref{lem:general_covar} and \eqref{eq:lambda*_bound} we then get that
\begin{equation*}
\E \mathcal I_\ell(k_{2,1,1}, \ldots, k_{2,1, [\Gamma_{\ell, \beta}]}) - \E \Phi_{\lambda, \mathcal Q^*_{\ell, 2, 1}}(S_{., \ell}) \leq O(\gamma (\beta\log \Gamma)^{0.5})d_{\gamma, \ell - 1}\gamma^2 \leq  0.0001d_{\gamma, \ell - 1}\gamma^2\,.
\end{equation*}
Combined with \eqref{eq:optimization_BM2} and \eqref{eq:expected_gadget_cost3}, this yields
\begin{equation}
\label{eq:optimization_BM3}
\E \mathcal I_\ell(k_{2,1,1}, \ldots, k_{2,1, [\Gamma_{\ell, \beta}]}) \geq 0.9995 d_{\gamma, \ell -1}\frac{\gamma^2}{4}\,.
\end{equation}
Having estimated the expected gain from switchings let us now turn our attention to the increment term from 
Claim~\ref{claim:invariant_expect}. The key tools that we will use for this purpose are the self-similar nature of the coverings $\mathscr C_{n, k, x;r}$'s and the 
limit result given in Lemma~\ref{lem:Brownian_scaling}. 
But first we need to get rid of the coefficients $\e^{X_{n, B, \ell-1, v^*}}$'s in the expression of $\increment_{[v]_I, 2, k}$'s. Notice that:
$$\increment_{[v]_I, 2, k} \leq \increment_{[v]_I, 2, k, \star} + \increment_{[v]_I, 2, k, *}\,,$$
where 
$$\increment_{[v]_I, 2, k, \star} = \e^{C_\delta\gamma \log\Gamma}\frac{\gamma^2}{2}d_{\gamma, \ell, [v]_I}\E X_{n, \ell - 1, \ell, v^*}^2\,,$$ 
and 
$$\increment_{[v]_I, 2, k, *} = \frac{\gamma^2}{2}d_{\gamma, \ell, [v]_I}\e^{\gamma X_{n, B_{\ell, 2, j; I, k}(v^*), \ell - 1, v^*}}\mathbf 1_{\{X_{n, B_{\ell, 2, j; I, k}(v^*), \ell - 1, v^*} \geq C_\delta\log \Gamma\}}\E X_{n, \ell - 1, \ell, v^*}^2\,.$$
The constant $C_\delta$ is same as in 
Lemma~\ref{lem:error_control}. From that lemma we also get
\begin{equation}
\label{eq:increment*}
\E (\increment_{[v]_I, 2, k, *}) = O(\gamma^{5.9})d_{\gamma, \ell, [v]_I}\,.
\end{equation}
Thus we only need to bound $\E(\increment_{[v]_I, 2, k, 
\star})$. By induction hypothesis \eqref{hypo:symmetry} we can write
$$\E (\increment_{[v]_I, 2, k, \star}) = \e^{C_\delta\gamma\log\Gamma}\frac{\gamma^2}{2}\frac{d_{\gamma, \ell, [v]_I}}{|[v]_I^{2, k}|}\sum_{w \in [v]_I^{2, k}}\E X_{n, \ell - 1, \ell, w}^2\,,$$
where $[v]_I^{2, k} = \cup_{B \in \descend_{\ell - 1, I, 
k}}([v]_I \cap B)$. From our construction cardinality of the spans of rectangles in $\descend_{\ell - 1, I, k}$ is at most $O(\epsilon a_\ell / \Gamma)$. Let $s_{I, B}$ denote the span of a rectangle $B$ in $\descend_{\ell, I}$. Then Lemma~\ref{lem:field_smoothness} tells us
$$\Big |\E X_{n, \ell - 1, \ell, w_B}^2 - \frac{1}{|s_{I, B}|}\sum_{u \in \{v_x\} \times s_{I, B}}\E X_{n, \ell - 1, \ell, u}^2 \Big| \leq O(\epsilon / \delta^{1.25})\sqrt{\log (1 / \delta)}\gamma^2\,,$$
where $w_B$ is the unique vertex from $B$ in $[v]_I^{2, 
k}$. Notice that $v_x$ is the common horizontal coordinate of the points in $[v]_I^{2, k}$. Denote by $s_{I, \ell, 2, k}$ the union of spans of rectangles in $\descend_{\ell - 1, I, 
k}$. The last display and the fact that rectangles in $\descend_{\ell, I}$ have identical dimensions imply
\begin{equation}
\label{eq:increment_star}
\E (\increment_{[v]_I, 2, k, \star}) \leq \e^{C_\delta\gamma\log\Gamma}\frac{\gamma^2}{2}\frac{d_{\gamma, \ell, [v]_I}}{|s_{I, \ell, 2, k}|}\sum_{u \in \{v_x\} \times s_{I, \ell, 2, k}}\E X_{n, \ell - 1, \ell, u}^2 + O(\delta^{98})d_{\gamma, \ell, [v]_I}O(\gamma^2)\,,
\end{equation}
Now define a new quantity $A_{[v]_I, 2}$ by
\begin{equation}
\label{eq:increment1}
A_{[v]_I, 2} = \e^{C_\delta\gamma\log\Gamma}\frac{\gamma^2}{2}\frac{d_{\gamma, \ell, [v]_I}}{|s_{I,\ell, 2}|}\sum_{u \in \{v_x\} \times s_{I, \ell, 2}}\E X_{n, \ell - 1, \ell, u}^2\,,
\end{equation}
where $s_{I, \ell, 2} = s_{I, \ell, 2, 1} \cup s_{I, \ell, 2, 
2}$. Recall that $\eta_{n , \ell, v} = X_{n, \ell - 1, \ell, u} + \eta_{n , \ell - 1, v}$ and $X_{n, \ell - 1, \ell, u}, \eta_{n , \ell - 1, v}$ are independent. 
Thus, 
$$\E X_{n, \ell - 1, \ell, u}^2 = G_{\tilde V^\Gamma_\ell}(u, u) - G_{\tilde V^\Gamma_{\ell - 1}}(u, u)\,.$$
Since the distance between any vertex in $[v]_I$ and the left and right boundaries of $\tilde V_\ell^\Gamma$ is at least $\Gamma \delta a_\ell /8$, we have from Lemma~\ref{lem:greens_overshoot}
$$|G_{\tilde V^\Gamma_\ell}(u, u) - G_{\tilde V^\Gamma_{\ell, v_x}}(u, u)| \leq O(1)\e^{-\Theta(\Gamma \delta)}\,,$$
where $V^\Gamma_{\ell, v_x}$ is the sub-rectangle of $V^\Gamma_\ell$ formed between the vertical lines $x = v_x - \vertical_\ell\times \lfloor \tfrac{\Gamma\delta}{16}\rfloor$ and $x = v_x + \vertical_\ell\times \lfloor 
\tfrac{\Gamma\delta}{16}\rfloor$. Similarly we can define the sub-rectangle $V^\Gamma_{\ell, 2, j, v_x}$ of $V^\Gamma_{\ell, 2, j}$ (where $I \in \mathscr C_{\ell, j}$) and the following is true:
$$|G_{\tilde V^\Gamma_{\ell, 2, j}}(u, u) - G_{\tilde V^\Gamma_{\ell, 2, j, v_x}}(u, u)| \leq O(1)\e^{-\Theta(\Gamma \delta)}\,.$$
Plugging these bounds into the right hand side of \eqref{eq:increment1} we get
\begin{equation}
\label{eq:increment2}
A_{[v]_I, 2} \leq \e^{C_\delta \gamma \log \Gamma}\frac{\gamma^2}{2}\frac{d_{\gamma, \ell, [v]_I}}{|s_{I, \ell, 2}|}\sum_{u \in \{v_x\} \times s_{I, \ell, 2}}(G_{\tilde V^\Gamma_{\ell,v_x}}(u, u) - G_{\tilde V^\Gamma_{\ell, 2, j, v_x}}(u, u)) + O(\gamma^2)d_{\gamma, \ell, [v]_I}\e^{-\Theta(\Gamma \delta)}\,.
\end{equation}
As we mentioned in Subsection~\ref{subsec:descrip}, the spans of the rectangles in $\descend_{\ell, I}$ form the collection $\mathscr C_{\ell + 1, \Gamma, 0; d_I, \principal}$ for some $d_I$ between $\ell \% 200m_\Gamma + 2$ and $\ell \% 
200m_\Gamma + m + 2$. Hence from symmetry and Lemma~\ref{lem:Green_function_continuity} we have
\begin{equation}
\label{eq:increment3}
\Big |\frac{1}{|s_{I, \ell, 2}|}\sum_{u \in \{v_x\} \times s_{I, \ell, 2}} G_{\tilde V^\Gamma_{\ell,v_x}}(u, u) -\frac{1}{|s_{I, \ell, *}|}\sum_{u \in \{v_x\} \times \{s_{I, \ell, *}\}} G_{\tilde V^\Gamma_{\ell,v_x}}(u, u) \Big | \leq O\big(\frac{\epsilon \log \Gamma}{\delta^2\Gamma}\big)\,,
\end{equation}
where $s_{I, \ell, *}$ is the union of intervals in 
$\mathscr C_{\ell + 1, \Gamma, 0; d_I, \principal}$. 
Therefore, we can bound the first summand in \eqref{eq:increment2} as follows:
\begin{equation}
\label{eq:increment2*}
\e^{C_\delta \gamma \log \Gamma}\frac{\gamma^2}{2}d_{\gamma, \ell, [v]_I}\Big(\frac{1}{|s_{I, \ell, *}|}\sum_{u \in \{v_x\} \times s_{I, \ell, *}}G_{\tilde V^\Gamma_{\ell,v_x}}(u, u) - \frac{1}{|s_{I, \ell, 2}|}\sum_{u \in \{v_x\} \times s_{I, \ell, 2}}G_{\tilde V^\Gamma_{\ell, 2, j, v_x}}(u, u) \Big)\,.
\end{equation}
Since $\ell \geq 200a'm_\Gamma$ (see Subsection~\ref{subsec:induct_hypo}), from Lemma~\ref{lem:Brownian_scaling} we get
\begin{equation}
\label{eq:increment4}
A_{[v]_I, 2} \leq \big(\frac{\log 2}{\pi} + 10^{-6}\big)\gamma^2d_{\gamma, \ell, [v]_I} + o_{\gamma \to 0, \delta}(\gamma^2)d_{\gamma, \ell, [v]_I}\,.
\end{equation}
Combined with \eqref{eq:increment*}, \eqref{eq:increment_star} and Claim~\ref{claim:invariant_expect}, it yields that
\begin{equation}
\label{eq:increment4*}
\sum_{I \in \mathscr C_{\ell, 1} \cup \mathscr C_{\ell, 2}}\sum_{[v]_I \in \descend_{\ell, I}}\E (\increment_{[v]_I, 2, k_I}) \leq 2d_{\gamma, \ell - 1}(\frac{\log 2}{\pi} + 10^{-6})\gamma^2\,.
\end{equation}
Combined with \eqref{hypo:short_range_ratio} and \eqref{hypo:gadget_negligible}, the same argument implies \eqref{hypo:expected_weight} for the easy case.

Now let us estimate the expected value of the sum of maximum possible errors that we made in every stage of approximation described in 
Subsection~\ref{subsec:induct_step_hard}. Since we make at most $(2 + O(\delta))\alpha$ many switches, from the bounds \eqref{eq:error_bound_Taylor} to \eqref{eq:error_bound_diagonalize} given the last subsection we find this expectation to be bounded by 
$$O(\gamma^2)d_{\gamma, \ell - 1}\sqrt{\frac{\delta}{\alpha}} \alpha\leq 0.0001d_{\gamma, \ell - 1}\gamma^2\,.$$
Combined with \eqref{eq:optimization_BM3}, \eqref{eq:increment4*} and Lemma~\ref{lem:error_control}, it gives that
\begin{equation}
\label{eq:increment_tentative}
\E \tilde D^\star_{\gamma, \ell, 2, \mathscr C_{\ell, 1}} + \E \tilde D^\star_{\gamma, \ell, 2, \mathscr C_{\ell, 2}} \leq 2d_{\gamma, \ell - 1}' + d_{\gamma, \ell - 1}(\frac{2\log 2}{\pi}\gamma^2 - 0.499\gamma^2) \leq 2d_{\gamma, \ell - 1}' - 0.05d_{\gamma, \ell - 1}\gamma^2\,,
\end{equation}
where $d_{\gamma, \ell - 1}' = \E \tilde D_{\gamma, \ell 
- 1, 1} = \E \tilde D_{\gamma, \ell - 1, 2}$. 
In addition, from the construction of $\cross^{*, \midd, \ell, 2}$ and \eqref{hypo:short_range_ratio}  it follows that,
\begin{equation}
\label{eq:bridge_tentative}
\E \tilde D^\star_{\gamma, \ell, 2, \mathscr C_{\ell, 3}} \leq d_{\gamma, \ell - 1}\delta(1 + \delta^{-1/8}O(\log (\delta/ \epsilon))\gamma^2) = d_{\gamma, \ell - 1}(\delta + 0.004\gamma^2)\,.
\end{equation}
Finally, by \eqref{hypo:gadget_negligible} we get that the expected total weight of gadgets which we have used to construct $\cross^{*, \ell-1}_i$ from the rectangles in $\widetilde \descend_{\ell}$ is at most 
$d_{\gamma, \ell -1}O(\gamma^2\log \Gamma)$. Altogether, we obtain
$$d_{\gamma, \ell} \leq \sum_{j \in [3]}\E \tilde D_{\gamma, \ell, j} + d_{\gamma, \ell - 1}O(\gamma^2 \log \Gamma) \leq d_{\gamma, \ell - 1}(2 + \delta - 0.045\gamma^2)\,,$$
which verifies \eqref{hypo:expected_weight}.

Next we verify \eqref{hypo:short_range_ratio}. We just need to show that the expected gain from switchings can not be 
too big. To this end, note that for any strategy $\{k_{2, 1, j'}\}_{j' \in [\Gamma_{\ell, \beta}]}$ such that the total number of switches is bounded by $(2 + O(\delta))\alpha$, we have
\begin{eqnarray}
\label{eq:short_range_ratio1}
\E \Big(\frac{d_{\gamma, \ell-1}}{2\Gamma a_{\ell - 1}}\sum_{j' \in [\Gamma_{\ell, \beta}]}(-1)^{k_{2, 1, j'}}\gamma \Delta \tilde \nu_{2, 1, j'}\tilde \eta_{n, \ell, 1, j'}\Big) &\leq& O(\alpha)\E (T_{i, j, \gamma, \ell}^{0.5}) = O(\alpha)d_{\gamma, \ell - 1}O(\gamma^2/\alpha^{0.5}) \nonumber \\
&=& d_{\gamma, \ell - 1} O(\delta^{-1/8})\,.
\end{eqnarray}

The remaining task is to verify \eqref{hypo:gadget_negligible}. But this follows from the bounds given in \eqref{eq:expected_gadget_cost3}.

\subsection{Proof of Theorem~\ref{theo:main}}
Since $d_{\gamma, \ell} / d_{\gamma, \ell - 1} \leq (2 + \delta + \gamma^2)$ when $\ell \% 200m_\Gamma < m_\Gamma + 100m$ and $\leq (2 + \delta - 0.045\gamma^2)$ when $\ell \% 200m_\Gamma > m_\Gamma + 100m$ (for all $\ell$ larger $200a'm_\Gamma$), it immediately follows that $$d_{\gamma, n} \leq C_\gamma'(2 + \delta -0.01\gamma^2)^n\,,$$
for some $C_\gamma' > 0$. As a result one can show that the expected weight for the geodesic connecting two fixed vertices within $\lfloor \mathcal I_{n, \Gamma, 0} \rfloor \times \lfloor \mathcal I_{n, 1, 0} \rfloor$
has exponent strictly less than 1, by constructing a sequence of ($O(n)$ many) rectangles with geometrically growing size that connect these two vertices. This completes the proof of Theorem~\ref{theo:main} (see Remark~\ref{remark:main_theorem}).


\end{document}